\documentclass[12pt,oneside,reqno]{amsart}
\usepackage{relsize}
\makeatletter
\newcommand\xmathlarger[2][1]{%
  \mbox{\larger[#1]$\displaystyle#2\m@th$}%
}
\makeatother
\usepackage[margin=1in]{geometry}
\usepackage{color}
\usepackage{esint,amssymb}
\usepackage{graphicx}
\usepackage{MnSymbol}
\usepackage{mathtools}
\usepackage[colorlinks=true, pdfstartview=FitV, linkcolor=blue, citecolor=blue, urlcolor=blue,pagebackref=false]{hyperref}
\usepackage{microtype}
\usepackage{amsmath}
\usepackage{relsize}
\usepackage{cancel}
\usepackage{xifthen}
\usepackage{verbatim}
\definecolor{darkgreen}{rgb}{0,0.5,0}
\definecolor{darkblue}{rgb}{0,0,0.7}
\definecolor{darkred}{rgb}{0.9,0.1,0.1}

\newtheorem{theorem}{Theorem}
\newtheorem{proposition}[theorem]{Proposition}
\newtheorem{lemma}[theorem]{Lemma}
\newtheorem{corollary}[theorem]{Corollary}

\theoremstyle{definition}
\newtheorem{remark}[theorem]{Remark}

\newcommand{\tref}[1]{Theorem~\ref{t.#1}}
\newcommand{\pref}[1]{Proposition~\ref{p.#1}}
\newcommand{\lref}[1]{Lemma~\ref{l.#1}}
\newcommand{\cref}[1]{Corollary~\ref{c.#1}}

\newcommand{\sref}[1]{Section~\ref{s.#1}}

\newcommand{\eref}[1]{(\ref{e.#1})}

\numberwithin{equation}{section}
\numberwithin{theorem}{section}

\newcommand{\Z}{\mathbb{Z}}
\newcommand{\N}{\mathbb{N}}
\newcommand{\R}{\mathbb{R}}

\renewcommand{\subset}{\subseteq}

\newcommand{\mG}{\mathfrak{G}}

\newcommand{\eps}{\varepsilon}

\renewcommand{\fint}{\strokedint}

\newcommand{\test}[1][]{%
\ifthenelse{\equal{#1}{}}{omitted}{given}%
}

\newcommand{\derv}[3]{\partial_x^{\alpha{#1}}\partial_v^{\beta{#2}}Y^{\omega{#3}}}
\newcommand{\der}{\derv{}{}{}}

\renewcommand{\d}[1]{\ensuremath{\operatorname{d}\!{#1}}}
\DeclarePairedDelimiter{\norm}{\lVert}{\rVert}
\DeclareMathOperator{\boot}{boot}
\DeclareMathOperator{\aux}{Aux}

\DeclareMathOperator{\ini}{in}
\newcommand{\jap}[1]{\left\langle {#1} \right\rangle}
\renewcommand{\S}{\mathbb{S}}

\newcommand{\mf}{\mathfrak{f}}
\newcommand{\mg}{\mathfrak{g}}
\newcommand{\mh}{\mathfrak{h}}
\DeclareMathOperator{\sm}{small}
\renewcommand{\bar}{\overline}
\renewcommand{\tilde}{\widetilde}

\renewcommand{\part}{\partial}

\usepackage{dsfont}

\begin{document}
\title[Stability of vacuum for the Boltzmann Equation]{Stability of vacuum for the Boltzmann Equation\break with moderately soft potentials}
\begin{abstract}
We consider the spatially inhomogeneous non-cutoff Boltzmann equation with moderately soft potentials and any singularity parameter $s\in (0,1)$, i.e.~with~$\gamma+2s\in(0,2)$ on the whole space $\R^3$. We prove that if the initial data $f_{\ini}$ are close to the vacuum solution $f_{\text{vac}}=0$ in an appropriate weighted norm then the solution $f$ remains regular globally in time and approaches a solution to a linear transport equation. 

Our proof uses $L^2$ estimates and we prove a multitude of new estimates involving the Boltzmann kernel without angular cut-off. Moreover, we rely on various previous works including those of Gressman--Strain, Henderson--Snelson--Tarfulea and Silvestre.

From the point of view of the long time behavior we treat the Boltzmann collisional operator perturbatively. Thus an important challenge of this problem is to exploit the dispersive properties of the transport operator to prove integrable time decay of the collisional operator. This requires the most care and to successfully overcome this difficulty we draw inspiration from Luk's work [Stability of vacuum for the Landau equation with moderately soft potentials, Annals of PDE (2019) 5:11] and that of Smulevici [Small data solutions of the {V}lasov-{P}oisson system and the vector field method, Ann. PDE, 2(2):Art. 11, 55, 2016]. In particular, to get at least integrable time decay we need to consolidate the decay coming from the space-time weights and the decay coming from commuting vector fields. 
 \end{abstract}

\author[S. Chaturvedi]{Sanchit Chaturvedi}
\address[Sanchit Chaturvedi]{450 Jane Stanford Way, Bldg 380, Stanford, CA 94305}
\email{sanchat@stanford.edu}
\subjclass[2010]{}
\date{\today}

\maketitle
\section{Introduction}\label{s.Introduction}
We consider the Boltzmann equation, 
\begin{equation}\label{e.boltz}
\part_t f+v_i\part_{x_i}f=Q(f,f)
\end{equation}
for the particle density $f(t,x,v)\geq0$, with position $x\in \R^3$, velocity $v\in \R^3$ and time $t\in \R_{\geq 0}$. The right hand side of \eref{boltz} is the binary Boltzmann collision operator given by,
\begin{equation}\label{e.boltz_kernel}
Q(f,g)=\int_{\R^3}\int_{\S^2} B(v-v_*,\sigma)[f(v_*')g(v')-f(v_*)g(v)]\d \sigma\d v_*.
\end{equation}
In all the expressions above (and in the remainder of the paper), we have used the convention that repeated lower case Latin indices are summed over $i = 1$, $2$, $3$. 

The variables $v_*'$ and $v'$ are the particle velocities after collision and are given by the following formula for $\sigma\in \S^2$, 
\begin{equation}\label{e.col_var}
v'=\frac{v+v_*}{2}+\frac{|v-v_*|}{2}\sigma, \hspace{1em}v'_*=\frac{v+v_*}{2}-\frac{|v-v_*|}{2}\sigma.
\end{equation}

The Boltzmann colission kernel $B(v-v_*,\sigma)$, due to physical considerations, is a non-negative function which depends on the relative velocity $|v-v_*|$ and on the deviation angle $\theta$ through $\cos \theta=\jap{k,\sigma}$ where $k=(v-v_*)/|v-v_*|$. We also assume that $B(v-v_*,\sigma)$ is supported on $\jap{k,\sigma}\geq 0$, i.e. $0\leq \theta\leq \frac{\pi}{2}$. Otherwise we can reduce to this situtation by symmetrizing the kernel as follows,
\begin{equation}\label{e.sym_b}
\bar B= [B(v-v_*,\sigma)+B(v-v_*,-\sigma)]\mathds{1}_{\jap{k,\sigma}\geq 0}.
\end{equation}
See pg. 51 in \cite{Vil02} for more details on this.\\
\textbf{Modelling assumption for the kernel.} We consider the kernel without any angular cut-off and make the following assumptions on $B(v-v_*,\sigma)$ for the rest of the paper
\begin{itemize}
\item We suppose that $B(v-v_*,\sigma)$ can be decomposed as follows $$B(v-v_*,\sigma)=\Phi(|v-v_*|)b(\cos \theta).$$
\item The angular function $b(\cos \theta)$ has a singularity at $\theta=0$ and that for some $c_b$ we have,
$$c_b\theta^{-1-2s}\leq \sin\theta b(\cos\theta)\leq c_b^{-1}\theta^{-1-2s}, \hspace{1em} s\in (0,1), \hspace{1em} \forall \theta\in\left(0,\frac{\pi}{2}\right].$$
\item The kinetic factor $z\to \Phi(|z|)$ satisfies for some $C_\Phi\geq 0$ $$\Phi(|v-v_*|)=C_\Phi|v-v_*|^\gamma, \hspace{1em} \gamma+2s\in(0,2].$$
This restriction on $\gamma$ and $s$ refers to the moderately soft potentials regime. 
\end{itemize}

We refer the interested readers to \cite{Cer88, CerIllPul94, Vil02} and the references therein for a more detailed physics background for the Boltzmann equation.

Our main result is that for sufficiently regular and sufficiently small initial data, we have global existence and uniqueness for \eref{boltz}. That is, we have a unique and non-negative solution to the Cauchy problem for all times.
%

\begin{theorem}\label{t.global}
Fix $s\in (0,1)$, $\gamma$ such that $\gamma+2s\in(0,2)$ and a constant $d_0>0$. Then there exists an $\eps_0=\eps_0(d_0,\gamma,s)>0$ such that if $$\sum_{|\alpha|+|\beta|\leq 10}\norm{(1+|x-v|^2)\part^\alpha_x\part^\beta_v (e^{2d_0(1+|v|^2)} f_{\ini})}_{L^2_xL^2_v}^2\leq \eps^2$$  
for some $\eps\in[0,\eps_0]$ and $f_{\ini}$ is non-negative, then there exists a global solution, $f$ to \eref{boltz} with $f(0,x,v)=f_{\ini}(x,v)$. Further, the solution $f(t,x,v)$ remains non-negative for all times.

In addition, if we let $\delta$ such that $$\delta=\min\left\{\frac{1-s}{4},\frac{1}{10},\frac{\gamma+2s}{8}\right\},$$ we have that $fe^{d_0(1+(1+t)^{-\delta})(1+|v|^2)} $ is unique in the energy space $E^4_T\cap C^0([0,T);Y^4_{x,v})$ for all $T\in (0,\infty).$ Moreover, $fe^{d_0(1+(1+t)^{-\delta})(1+|v|^2)}\in E_T\cap C^0([0,T);Y_{x,v})$ for all $T\in (0,\infty).$ 
\end{theorem}
See \sref{notation} for definition of the relevant energy spaces.

\subsection{Related works}  In this section we give a rather inexhaustive list of pertinent results. 
\begin{enumerate}
\item \textbf{Stability of vacuum for collisional kinetic models:} The earlier works were concerned with the Boltzmann equation with an angular cut-off. The first work was by Illner-Shinbrot \cite{IlSh84}. There were many other follow up works of \cite{IlSh84}; see for instance \cite{Ar11,BaDeGo84,BeTo85,Guo01,Ha85,HeJi17,Po88,To86,ToBe84}. Perturbations to travelling global Maxwellians were studied in \cite{AlGa09,Go97,To88, BaGaGoLe16} and it was shown that the long-time dynamics is governed by dispersion.

The first stability problem for vacuum result with long-range interactions was only recently obtained by Luk, in \cite{Lu18}, who proved the result in the case of moderately soft potentials ($\gamma \in (-2,0)$). Luk combined $L^\infty$ and $L^2$ methods to prove global existence of solutions near vacuum. Moreover, Luk also proved that the long-time dynamics is governed by dispersion. The stability of vacuum for Landau equation with hard potentials (i.e. $\gamma\in[0,1]$) was considered by the author in \cite{Cha20}.

Although the present result is comparable to that of Luk in \cite{Lu18} in the sense that the range of potentials considered here ($\gamma+2s\in(0,2)$) is an analogue of the moderately soft potentials for Landau ($\gamma\in (-2,0)$), the Boltzmann kernel poses various technical difficulties which are not present in \cite{Lu18}. Another major difference is that while Luk uses maximum principle in conjunction with energy estimates in \cite{Lu18}, we only use $L^2$ estimates, as in \cite{Cha20}. We also need to combine the vector field approach used by Smulevici in \cite{Sm16} with the space-time weights in our energy norm to get enough time decay for us to close the estimates. 

\begin{remark}
In the case of inhomogeneous equations, this is the first global existence result for the non-cutoff Boltzmann equation for any $s\in(0,1)$ such that the solutions do not converge to a Maxwellian. 
\end{remark}
\item \textbf{Dispersion and stability for collisionless models:} The dispersion properties of the transport operator, which we leverage to prove our global existence result, have been instrumental in proving stability results for close-to-vacuum regime in collisionless models; see \cite{BaDe85,GlSc88,GlSt87} for some early results. Relations between these results and the stability of vacuum for the Boltzmann equation with angular cutoff is discussed in \cite{BaDeGo84}. More recent results can be found in \cite{Bi17, Bi18,Bi19.1, Bi19.2, Bi19.3, FaJoSm17.1,Sm16,Wa18,Wa18.1,Wa18.2,Wo18}. See also \cite{FaJoSm17,LiTa17,Ta17, BiFaJoSmTh20} for proof of the stability of the Minkowski spacetime for the Einstein-Vlasov system. 
\item \textbf{Regularity theory for cut-off and spatially homogeneous Boltzmann \break equation:} Historically, Boltzmann equation has been considered with an angular cut-off. Grad proposed, in \cite{Grad63}, an angular cut-off which requires $b(\cos \theta)$ to be bounded. Since then a lot has been done regarding Boltzmann equation with some sort of an angular cut-off. 

There is also a huge literature on spatially homogeneous Boltzmann equation without cut-off. On the other hand literature for inhomogenous Boltzmann equation without angular cut-off is somewhat sparse.
\begin{itemize}
\item \textbf{Boltzmann equation with angular cut-off:} The theory for the cut-off case has been studied in \cite{AlGa09, BaDe85, BaDeGo84, BaGaGoLe16, BeTo85, Guo01, Guo02, Guo03.2,Guo03,Ha85,HeJi17,IlSh84,Levermore,Po88,StGu04,StGu06, To87, To88,ToBe84, DiLi89, DeVi05}.
\item \textbf{Spatially homogenous Boltzmann equation:} The angular noncut-off case was first considered in the spatially homogeneous case. The striking difference between cut-off and noncut-off case is the regularizing effect; see \cite{Li94,De95,AlEl05, DeWe04, Vil99}. For issues relating existence and uniqueness see \cite{AlDeViWe00, DeMo09, AlVi04}.
\end{itemize}
\item \textbf{Regularity theory for noncut-off and spatially inhomogeneous Boltzmann equation:} The noncut-off case for spatially inhomogeneous collisional equations is much less understood compared to the cut-off case or the spatially homogeneous case. 
\begin{itemize}
\item \textbf{Local existence:} Local existence and regularization effect for the noncut-off case was first studied in \cite{AMUXY10} by Alexandre--Morimoto--Ukai--Xu--Yang (AMUXY). The result focused on the case $s\in (0,\frac{1}{2})$ and $\gamma$ such that $\gamma+2s<1$. See also \cite{AMUXY13} for an existence result with less stringent regularity assumptions. The local existence for the whole singularity range $s\in (0,1)$ was only established recently by Henderson--Snelson--Tarfulea in \cite{HeSnTa19}. In addition to relaxing the singularity restriction, Henderson--Snelson--Tarfulea also considered $\gamma$ such that $\max\{-3,-\frac{3}{2}-2s\}<\gamma<0$. Among many key insights, one of them was a clever integration by parts that we also exploit.
\item \textbf{Conditional regularity:} Recently there have been some results that concern the regularity of solutions to Boltzmann equation assuming a priori pointwise control of the mass density, energy density and entropy density; see \cite{Sil14,ImbSil16,ImbSil18,ImbSil19,ImMoSi18,ImMoSi19}. In particular we rely heavily on the integro-differential kernel formulation of the Boltzmann kernel introduced by Silvestre in \cite{Sil14} and the estimates further developed in \cite{ImbSil16}.
\item \textbf{Global nonlinear stability of Maxwellians:} The stability problem for\break Maxwellians has been solved thanks to the works of Gressman--Strain \cite{GrSt11, GrSt11.2} and the AMUXY group \cite{AMUXY12.3, AMUXY12.2, AMUXY12}. We borrow the dyadic decomposition of the angular singularity in the Boltzmann kernel and various other estimates developed in the \cite{GrSt11}.
\end{itemize}
\end{enumerate}
\subsection{Proof Strategy}
In the near vacuum regime we expect that the long time behavior of the solution to \eref{boltz} is dictated by dispersion coming from the transport operator. Hence we hope to be able to treat the Boltzmann collision operator perturbatively. This requires us to optimally exploit the time decay from the transport operator. To achieve this we  need to amalgamate the space-time weights used by Luk in \cite{Lu18} and the vector field method in the context of transport equations developed by Smulevici \cite{Sm16}. Next we give a general proof strategy employed in this paper.
\begin{enumerate}
\renewcommand{\labelenumi}{(\theenumi)}
\item \textbf{Local existence:} The local existence problem was first studied by the AMUXY group with the restriction that $s\in \left(0,\frac{1}{2}\right)$. In a very recent paper Henderson--Snelson--Tarfulea have been able to treat the whole physical regime as far as the angular singularity is concerned (i.e. $s\in (0,1)$). This is due to the clever integration by parts discovered by the authors in \cite{HeSnTa19}\footnote{In contrast to the Gaussian weights employed by AMUXY, Henderson--Snelson--Tarfulea can prove local existence with just polynomial velocity weights.}.

Two major difficulties one has to face while proving existence for the Boltzmann equation are derivative loss and moment (in velocity) loss. As in \cite{AMUXY10} and \cite{HeSnTa19}, we will also use an $L^2$ based approach to avoid the potential derivative loss issue. However, for the purpose of establishing time decay we need to treat the additional difficulties arising from the space-time weights. We will come back to these issues in \sref{proof_details}. 

An $L^2$ based approach lets us use the subtle integration by parts exploited in \cite{HeSnTa19} to prove existence for all $s\in(0,1)$.To take care of the moment loss issue we use a time-dependent Gaussian in $\jap{v}$ as in \cite{AMUXY10}. Our choice of weights will decrease as $t$ increases, but it needs to decay in a sufficiently slow manner so that it is non-degenerate as $t\to \infty$. More precisely, we define 
\begin{equation}\label{e.g}
g:=e^{d(t)\jap{v}^2}f, \hspace{.5em} d(t):=d_0(1+(1+t)^{-\delta})
\end{equation}
for appropriate $d_0>0$, $\delta>0$ and estimate $g$ instead of $f$, which satisfies the following equation\footnote{For the local problem one can just use $e^{-(d_0-\kappa t)\jap{v}^2}$ instead of our weight. We use our Gaussian weight instead since such a weight would be required for the global problem.},
\begin{equation} \label{e.rough_eq_for_g}
\begin{split}
\partial_tg+v_i\partial_{x_i}g+\frac{\delta d_0}{(1+t)^{1+\delta}}\jap{v}^{\color{black}{2}}g&=Q(f,g)+\text{other terms}.
\end{split}
\end{equation}

\item \textbf{Dispersive properties of the transport equation:}\label{2} The long time limit of the solution of \eref{boltz} in the near-vacuum regime is expected to be governed by dispersion coming from the transport equation. With this in mind, we would like to treat the Boltzmann collision operator perturbatively. To be able to successfully do that we need to establish enough time decay for the collisional kernel.

The essential behaviour of the Boltzmann collision operator can be captured by a fractional flat diffusion. That is, 
$$F\mapsto Q(g,F)\approx -(-\Delta_v)^sF+\text{ lower order terms}.$$
More precisely, the following estimate has been proved in \cite{GrSt11, ImbSil16, AMUXY10} in various forms\footnote{In contrast to \cite{GrSt11} where it is enough to estimate $f$ in $L^2_v$, we need to make sure to esimate it in $L^1$. This is because we hope to gain time decay when we go from $L^1_v$ to $L^2_v$ via our space-time weights.},
\begin{equation}\label{e.boltz_est}
\int_{\R^3}\int_{\R^3}Q(f,g)h\d v\d x\lesssim \norm{f\jap{v}^{\gamma+2s}}_{L^\infty_xL^1_v}\norm{g\jap{v}^{\frac{\gamma}{2}+s}}_{L^2_xH^{2s}_v}\norm{h\jap{v}^{\frac{\gamma}{2}+s}}_{L^2_xL^2_v}.
\end{equation}

Recognizing this we give a heuristic argument as to why we expect in the near-vacuum regime that the solutions to Boltzmann equation approach solutions to the linear transport equation in the long time limit. We first remind ourselves of the decay estimates for the transport equation,
\begin{equation}\label{e.lin_transport}
\part_t f+v_i\part_{x_i}f=0.
\end{equation}
Consider a solution $f_{\text{free}}$ of the transport equation \eref{lin_transport}. We have an explicit solution to \eref{lin_transport} and it is given by 
\begin{equation}\label{e.lin_soln}
f_{\text{free}}(t,x,v)=f_{\text{data}}(x-tv,v).
\end{equation}
Then if the initial data $f_{\text{data}}$ is sufficiently localized in $x$ and $v$, then for all $l\in \mathbb{N}\cup \{0\}$,
\begin{equation}\label{e.lin_Linf}
\norm{\jap{v}^lf_{\text{free}}}_{L^\infty_xL^1_v}\lesssim (1+t)^{-3}, \hspace{1em}|\part^\alpha_x\part_v^\beta f|\lesssim (1+t)^{|\beta|}.
\end{equation}
We also have the following $L^2_v$ bound
\begin{equation}\label{e.lin_L2}
\color{black}{\norm{\jap{v}^l\part^\alpha_x\part^\beta_v f_{\text{free}}}_{L^2_xL^1_v}}\lesssim (1+t)^{-\frac{3}{2}+|\beta|},\hspace{1em} \norm{\jap{v}^l\part^\alpha_x\part^\beta_v f_{\text{free}}}_{L^2_xL^2_v}\lesssim (1+t)^{|\beta|}.
\end{equation}

Note that each $\part_v$ worsens the decay estimate by $t$ but $\part_x$ leaves the decay estimate untouched. 

Now assuming that $f$ and $g$ in \eref{boltz_est} satisfy the linear estimate \eref{lin_Linf} and \eref{lin_L2} \textcolor{black}{and $h$ is bounded in the weighted $L^2_xL^2_v$ norm}, we have for Boltzmann equation,
\begin{equation}\label{e.decay}
\int_{\R^3}\int_{\R^3}Q(f,g) h\d v\d x\lesssim_h (1+t)^{-3}\times (1+t)^{2s}.
\end{equation}
Since $s\in(0,1)$ the time decay in the case of Boltzmann equation is actually integrable!\\
This is in contrast to the non-integrable time decay one would get for the Landau equation, see \cite{Lu18} and \cite{Cha20}. To remedy this borderline non-integrable time decay one has to exploit the null structure of the Landau kernel.
%
%
\item \textbf{Decay and heuristics for the $L^2$ estimates and vector field method:}\label{3}
For the case of Landau equation, Luk in \cite{Lu18} uses $L^\infty$ estimates and the maximum principle to get time decay. Unfortunately using maximum principle with energy estimates as in \cite{Lu18} seems hard in the case of Boltzmann equation. This is because in comparison to Landau equation, where we have a Laplace-like term, Boltzmann equation has a singularity in angle in the collision kernel in place of the Laplace-like term for Landau. Thus, proving maximum prinicple and propagating $L^\infty$ bounds seems like a daunting task in this case. 

Instead of a maximum principle approach we use a purely $L^2$ based approach as in \cite{Cha20}. To close the estimates in \cite{Cha20}, $L^2$ estimates suffice because the nonlinear structure of the Landau equation can be exploited to a greater extent for the hard potentials than for soft potentials. Hence Luk has to resort to $L^\infty$ estimates to get the requisite amount of time decay in \cite{Lu18} for the case of soft potentials. Since we want to treat the moderately soft potentials case with purely $L^2$ estimates we run into major difficulties relating to time decay.

We saw in point (\ref{2}) that we have enough time decay for the Boltzmann collisional operator assuming \eref{lin_Linf}. Since we want to bypass the use of $L^\infty$ estimates, we need to be able to propagate \eref{lin_Linf} assuming only $L^2$ control which does not seem to be enough. Thus the best estimate we can hope for is \eref{lin_L2}. Now note though that using \eref{boltz_est} and \eref{lin_L2} in conjunction with Sobolev embedding we get the decay,
\begin{align*}
\int_{\R^3}\int_{\R^3}&Q(f,g) h\d v\d x\\
&\lesssim_v \norm{f}_{L^\infty_xL^1_v}\norm{g}_{L^2_xH^{2s}_v}\norm{h}_{L^2_xL^2_v}\\
&\lesssim_v \sum_{|\alpha|\leq 2}\norm{\part^\alpha_x f}_{L^2_xL^1_v}\norm{g}_{L^2_xH^{2s}_v}\norm{h}_{L^2_xL^2_v}\\
&\lesssim_h (1+t)^{-\frac{3}{2}}\times (1+t)^{2s}.
\end{align*}
This $(1+t)^{-\frac{3}{2}}$ decay comes from the dispersion due to the transport operator and is enforced for the Boltzmann equation via the space-time weights $\sqrt{1+|x-(t+1)v|^2}$ (see \lref{L1_to_L2}). That said, for $s>\frac{1}{4}$, this decay is not sufficient to close our estimates. We somehow need to get an extra $(1+t)^{-\frac{3}{2}}$ decay to be able to treat all physically relevant singularity, i.e. $s\in(0,1)$. To this end we use the vector field approach used by Smulevici in the context of Vlasov-Poisson equation; \cite{Sm16}. Smulevici's work builds up on vector field method developed by Klainerman in \cite{Kla85}.

More precisely, we use the following bound, 
$$\norm{f}_{L^\infty_xL^2_v}\lesssim (1+t)^{-\frac{3}{2}}\sum_{|\omega|\leq 3}\norm{Y^\omega f}_{L^2_xL^2_v},$$
where $Y_i=(t+1)\part_{x_i}+\part_{v_i}$; see \lref{vect_trick} for a proof. It can be easily checked that $Y$ commutes with the transport operator ($\part_{t}+v_i\part_{x_i}$) and thus taking a $Y$ derivative does not worsen the time decay unlike $\part_v$. This in conjunction with the gain in time decay we get by going from $L^1_v$ to $L^2_v$, thanks to our space-time weights, we are able to treat the main term without needing to propagate any $L^\infty$ bounds\footnote{We believe our $L^2$ approach can also be applied to the case of moderately soft potentials for Landau equation to give a purely $L^2$ based stability of vacuum proof as in \cite{Lu18}. We would still need to use the null structure observed by Luk in \cite{Lu18} but we can bypass the use of $L^\infty$ theory.}. 

\item \textbf{Additional technical difficulties:} There are various other difficulties that we run into, these include- moment loss issue for the space-time weights at the top order, not enough time decay for the commutator term arising due to the space-time weights and a coupled derivative loss and insufficent time decay issue for the top and penultimate order terms. See \sref{proof_details} for further discussions on these difficulties.
\end{enumerate}
\subsection{Paper organization} The remainder of the paper is structured as follows.\\
In \sref{notation}, we introduce some notations that will be in effect throughout the paper. In \sref{proof_details} we discuss the various technical difficulties which we encountered. In \sref{set_up} we set up the energy estimates. In \sref{sing} we review the singularity decomposition from \cite{GrSt11} and the machinery developed in \cite{Sil14} and \cite{ImbSil16}. \sref{main_est} is the main section where we prove the estimates involving the Boltzmann kernel. In \sref{local} we prove the local existence result. In \sref{bootstrap} we state the bootstrap assumptions and the theorem we wish to establish as a part of out bootstrap scheme. In \sref{errors}, we estimate the errors and prove the bootstrap theorem from \sref{bootstrap}. Finally, in \sref{put} we put everything together and prove \tref{global}.
\subsection{Acknowledgements} I would like to thank Jonathan Luk for encouraging me to do this problem and for numerous helpful discussions. I am also very thankful to the anonymous reviewer for their invaluable comments on the manuscript.
\section{Notations and Spaces}\label{s.notation}
We introduce some notations that will be used throughout the paper.

\textbf{Norms}. We will use mixed $L^p$ norms, $1\leq p<\infty$ defined in the standard way:
$$\norm{h}_{L^p_v}:=(\int_{\R^3} |h|^p \d v)^{\frac{1}{p}}.$$
For $p=\infty$, define
$$\norm{h}_{L^\infty_v}:=\text{ess} \sup_{v\in \R^3}|h|(v).$$
For mixed norms, the norm on the right is taken first. For example,
$$\norm{h}_{L^p_xL^q_v}:=(\int_{\R^3}(\int_{\R^3} |h|^q(x,v) \d v)^{\frac{p}{q}} \d x)^{\frac{1}{p}}$$
and
$$\norm{h}_{L^r([0,T];L^p_xL^q_v)}:=(\int_0^T(\int_{\R^3}(\int_{\R^3} |h|^q(x,v) \d v)^{\frac{p}{q}} \d x)^{\frac{r}{p}})^{\frac{1}{r}}$$
with obvious modifications when $p=\infty$, $q=\infty$ or $r=\infty$. All integrals in phase space are over either $\R^3$ or $\R^3\times \R^3$ and the explicit dependence is dropped henceforth.\\

\textbf{Japanese brackets}. Define $$\jap{\cdot}:=\sqrt{1+|\cdot|^2}.$$

\textbf{Multi-indices}. Given a multi-index $\alpha=(\alpha_1,\alpha_2,\alpha_3)\in (\N\cup \{0\})^3$, we define $\part_x^\alpha=\part^{\alpha_1}_{x_1}\part^{\alpha_2}_{x_2}\part^{\alpha_3}_{x_3}$ and similarly for $\part^\beta_v$. Let $|\alpha|=\alpha_1+\alpha_2+\alpha_3$. Multi-indices are added according to the rule that if $\alpha'=(\alpha'_1,\alpha'_2,\alpha'_3)$ and $\alpha''=(\alpha''_1,\alpha''_2,\alpha''_3)$, then $\alpha'+\alpha''=(\alpha'_1+\alpha''_1,\alpha'_2+\alpha''_2,\alpha'_3+\alpha''_3)$.\\

\textbf{Global energy norms}. For any $T>0$ the energy norm we use in $[0,T)\times\R^3\times\R^3$ is as follows 
\begin{equation}\label{e.Energy_norm}
\begin{split}
\norm{h}^2_{E^m_T}&:=\sum_{|\alpha|+|\beta|+|\omega|\leq m}(1+T)^{-2|\beta|}\norm{\jap{x-(t+1)v}^2\der h}^2_{L^\infty([0,T];L^2_xL^2_v)}\\
&\quad +\sum_{|\alpha|+|\beta|+|\omega|\leq m}(1+T)^{-2|\beta|}\norm{(1+t)^{-\frac{1+\delta}{2}}\jap{v}\jap{x-(t+1)v}^2\der h}^2_{L^2([0,T];L^2_xL^2_v)}.
\end{split}
\end{equation}
When \underline{$m=10$}, it is dropped from the superscript, i.e. $E_T=E^{10}_T.$

It will also be convenient to define some other energy type norms. Namely,
$$\norm{h}^2_{Y^m}(T):=\sum \limits_{|\alpha|+|\beta|+|\omega|\leq m} \norm{\jap{x-(t+1)v}^{2}\der h}^2_{L^2_xL^2_v}(T)$$
and 
$$\norm{h}^2_{X^m_{T}}:=\sum \limits_{|\alpha|+|\beta|+|\omega|\leq m} \norm{\jap{x-(t+1)v}^{2}\jap{v}\der h}^2_{L^2([0,T];L^2_xL^2_v)}.$$

\textbf{Instead of writing $g(v_*)$, $g(v_*')$, $g(v)$ and $g(v')$, we write $g_*$, $g_*'$, $g$ and $g'$ respectively}.

\textbf{For two quantitites, $A$ and $B$ by $A\lesssim B$, we mean $A\leq C(d_0,\gamma,s)B$, where $C(d_0,\gamma,s)$ is a positive constant depending only on $d_0$, $\gamma$ and $s$}.
\section{Technical difficulties}\label{s.proof_details}
In addition to the time decay issue which we already discussed in \sref{Introduction} we have various other technical issues that can potentially prevent us from closing our estimates.
\begin{itemize}
\item First recall that $Y_i=(t+1)\part_{x_i}+\part_{v_i}$. Now differentiating \eref{rough_eq_for_g} by $\der$ such that $|\alpha|+|\beta|+|\omega|=10$, we roughly get the following equation,
\begin{equation} \label{e.rough_eq_for_g_diff}
\begin{split}
\partial_t\der g&+v_i\partial_{x_i}\der g+\frac{\delta d_0}{(1+t)^{1+\delta}}\jap{v}^{\color{black}{2}}\color{black}{\der g}\\
&=\sum_{\substack{|\alpha'|+|\alpha''|=|\alpha|\\|\beta'|+|\beta''|=|\beta|\\|\omega'|+|\omega''|=|\omega|}}Q(\derv{'}{'}{'}f,\derv{''}{''}{''}g)+\text{other terms}.
\end{split}
\end{equation}
The estimate \eref{boltz_est} is only helpful when $|\alpha''|+|\beta''|+|\omega''|\leq 8$. That is we have a potential derivative loss issue. Thus we need to treat the top order and the penultimate order separately. In both cases, we use the subtle integration by parts observed in \cite{HeSnTa19} but carried out in a way to get adequate time decay.
\item Another techinal difficulty can be attributed to the space-time weights (which are needed to get time decay). \textcolor{black}{More precisely,} we need to be able to treat the following term,
\begin{align*}
\int_{\R^3}[Q(f,g)\jap{x-(t+1)v}^{2n}h-Q(f,\jap{x-(t+1)v}^{n}g)\jap{x-(t+1)v}^{n}h]\d v.
\end{align*}
In addition, we need to treat the top order and the non top order differently. This is because for non top orders we are able to put a derivative on $g$ while for the top order we need to put the derivative on $f$.
\item We also have a moment loss issue for the space-time weights at the top order. 
\item The Gaussian in velocity weight complicates matter and was the reason for the restriction $\gamma+2s<1$ in \cite{AMUXY10}. To be able to go to push $\gamma+2s$ to $2$, we have to decompose the angular singularity as in \cite{GrSt11} and use the estimates developed there in away from the angular singularity.
\end{itemize} 
\begin{enumerate}
\item \textbf{Main term for the top and the penultimate order.}
For the top order the main term is of the form $Q(f,G)G$. In this case, we are unable to use \eref{boltz_est} without a loss in derivatives. Thus to overcome this issue we use an appropriate variant of classical cancellation lemma (\lref{canc_lemma}) that lets us deal with the angular singularity without having to put any derivative on $G$ and instead puts the derivative on $f$.

For the penultimate order the main term is of the form $Q(\part f,G)\part G$, where \break $\part\in\{\part_{x_i},\part_{v_i},Y_i\}$. Again, we cannot apply \eref{boltz_est} for $s>\frac{1}{2}$\footnote{This is the main reason of the restriction, $s\in \left(0,\frac{1}{2}\right)$ in \cite{AMUXY10}.}. Although the main term is not as symmetric as for the top order term, we still have some symmetry. Henderson--Snelson--Tarfulea deal with the penultimate term using integration by parts in \cite{HeSnTa19} which is in turn motivated by their work on Landau equation in \cite{HeSnTa17}. They use the theory developed by Silverstre in \cite{Sil14} and Imbert--Silvestre in \cite{ImbSil16} and break up $Q(\part f,G)\part G$ into a term with the singularity $Q_1(\part f,G)\part G$ and another term $Q_2(\part f,G)\part G$, which can be treated by the cancellation lemma, \lref{canc_lemma}.

Using integration by parts (see \pref{pen_term_main}) we get that 
\begin{align*}
\int_{\R^3}\int_{\R^3}Q_1(\part f,G)\part G\d v\d x&=\int_{\R^3}\int_{\R^3}\int_{\R^3}K_{\part^2 f}(G'-G)^2\d v'\d v\d x\\
&\qquad+\int_{\R^3}\int_{\R^3}\int_{\R^3}(K_{\part f}-K_{\part f}')(G-G')\part G\d v'\d v\d x.
\end{align*}
The first term can be treated in the same way as in \cite{HeSnTa19}, see \lref{sym_bound_weight}. In \cite{HeSnTa19}, the authors use the cancellation from $K_{\part f}-K'_{\part f}$ by estimating $\part f$ in a H\"{o}lder space in velocity; see Proposition 3.1 (iv) in \cite{HeSnTa19}. For the local problem this poses no issue but for the global problem this is disastrous as far as the time decay is concerned. This is because we would have to apply Morrey's inequality at the cost of more than $\frac{3}{2}$ velocity derivatives and they also have to bound $G$ in $H^s$, which they get by using cancellation from $|G'-G|$. Then by \eref{lin_L2}, we would be losing $(1+t)^{\frac{3}{2}+2s}$ and we only gain $(1+t)^{-3}$ from the vector field method and the velocity averages; see Point \ref{3}. Thus we would be far from being integrable in time for $s\geq \frac{1}{4}$!

To overcome this issue we have to develop estimates similar to \lref{boltz_kernel_approx} for the difference of the kernels $|K_{\part f}-K_{\part f}'|$; see \lref{diff_K_outside_ball}. This lets us estimate $\part f$ in $H^1_v$ (\lref{diff_f_diff_G}) instead of a H\"{o}lder space. Moreover, we can exploit the cancellation from $|G-G'|$ at the cost of $2s-1$ velocity derivatives. Hence, in total we only lose $(1+t)^{2s}$ of the $(1+t)^{-3}$ decay, which means that we can still close the estimates.
\item \textbf{Commutator estimate and need for null structure.} Formally, Landau equation can be thought of as a limit of the Boltzmann equation as $s\to 1$. Thus we can expect an estimate similar to \eref{boltz_est} with $s$ replaced by $1$. Then it is easy to see, at least at a formal level, from \eref{decay} (with $s$ replaced by $1$) that we get barely non-integrable time decay in the case of Landau. Indeed, Luk in \cite{Lu18} and the author in \cite{Cha20} have to exploit the nonlinear structure of the Landau collisional kernel to get enough time decay.

Although the main terms in our analysis for the Boltzmann equation do not require us to exploit the null structure, the commutator terms need special care to establish integrable time decay. Due to our spacetime weights, we are unable to use the commutator estimate established in \cite{AMUXY10} (Lemma 2.4) and used in \cite{HeSnTa19} for the local problem. 

In our decomposition for
\begin{align*}
\int_{\R^3}[Q(f,g)\jap{x-(t+1)v}^{2n}h-Q(f,\jap{x-(t+1)v}^{n}g)\jap{x-(t+1)v}^{n}h]\d v
\end{align*}
we encounter a term of the following form for $n=2$ (see \eref{comm_est_weight_prod_diff}),  
\begin{align*}
\int_{\R^3}\int_{\R^3}\int_{\S^2}&Bf_*g[\jap{x-(t+1)v'}-\jap{x-(t+1)v}]\\
&\qquad\times [\jap{x-(t+1)v'}^2-\jap{x-(t+1)v}^2]\jap{x-(t+1)v'}h'\d \sigma\d v_*\d v.
\end{align*}
To be able to take care of the angular singularity, we need to take advantage of the cancellation from both differences of the weights. Using Taylor's theorem, we have the bound,
\begin{align*}
\jap{x-(t+1)v'}^2-\jap{x-(t+1)v}^2&=|v-v'|\part_{v_i}(\jap{x-(t+1)v}^2)\mid_z\\
&\lesssim|v-v'| (1+t)\jap{x-(t+1)z},
\end{align*}
where $z=\eta v+(1-\eta)v'$ for $\eta\in [0,1]$.\\
Similarly, $$\jap{x-(t+1)v'}-\jap{x-(t+1)v}\lesssim |v-v'|(1+t).$$
Now \eref{col_var} implies that $|v-v'|^2=|v-v_*|^2\sin^2 \frac{\theta}{2}$ with $\cos \theta=\left\langle\frac{v-v_*}{|v-v_*|},\sigma\right\rangle$. This factor of $\sin^2 \frac{\theta}{2}$ makes sure that the integral in $\d \sigma$ is bounded. In total we have the bound, 
\begin{equation}\label{e.show_decay}
\begin{split}
\int_{\R^3}\int_{\R^3}\int_{\S^2}&B(v-v_*,\sigma)f_*g[\jap{x-(t+1)v'}-\jap{x-(t+1)v}]\\
&\quad\times [\jap{x-(t+1)v'}^2-\jap{x-(t+1)v}^2]\jap{x-(t+1)v'}h'\d \sigma\d v_*\d v\\
&\qquad\lesssim (1+t)^2\int_{\R^3}\int_{\R^3}|v-v_*|^{2+\gamma}f_*\jap{x-(t+1)v}g\jap{x-(t+1)v'}^2h'\d v_*\d v,
\end{split}
\end{equation}
where we also used that 
\begin{align*}
\jap{x-(t+1)z}&\lesssim \jap{x-(t+1)v}+\jap{x-(t+1)v'}\\
&\lesssim \color{black}{ \jap{x-(t+1)v}\jap{x-(t+1)v'}}
\end{align*}
which follows from \eref{col_var}.\\
Thus we are able to successfully take care of the angular singularity but unfortunately, we have now gained $(1+t)^2$, which means that we are barely non-integrable as in the case of Landau equation; see \cite{Lu18}. 

We can remedy this by using the null structure for the nonlinearity found by Luk.  Thus we review the null structure that was recognized in \cite{Lu18}. For a sufficiently localized and regular data, the decay estimates \eref{lin_L2} are sharp only when $\frac{x}{t}\sim v$. For $\left|\frac{x}{t}-v\right|\geq t^{-\alpha}$ with $\alpha\in (0,1)$, we have better time decay.

There are three scenarios possible:
\begin{itemize}
\item $\frac{x}{t}$ is not too close to $v$.
\item $\frac{x}{t}$ is not too close to $v_*$.
\item $|v-v_*|$ is small.
\end{itemize}  
For the first two cases we get extra decay by the observation above. For the last case we get extra decay for \eref{show_decay} as $|v-v_*|$ is small and $2+\gamma>0$. This null structure is captured by using $\jap{x-(t+1)v}$ weights in the norm as in \cite{Lu18}. More precisely, we use that 
\begin{align*}
|v-v_*|^\delta&\lesssim (1+t)^{-\delta}|x-(t+1)v-(x-(t+1)v_*)|^\delta\\
&\lesssim (1+t)^{-\delta}\jap{x-(t+1)v}^\delta\jap{x-(t+1)v_*}^\delta,
\end{align*}
for $0<\delta<2+\gamma$. 

Putting all this together, we get
\begin{equation*}
\begin{split}
\int_{\R^3}&\int_{\R^3}\int_{\S^2}B(v-v_*,\sigma)f_*g[\jap{x-(t+1)v'}-\jap{x-(t+1)v}]\\
&\quad\times [\jap{x-(t+1)v'}^2-\jap{x-(t+1)v}^2]\jap{x-(t+1)v'}h'\d \sigma\d v_*\d v\\
&\qquad\lesssim (1+t)^2\int_{\R^3}\int_{\R^3}|v-v_*|^{2+\gamma}f_*\jap{x-(t+1)v}g\jap{x-(t+1)v'}^2h'\d \sigma\d v_*\d v\\
&\lesssim (1+t)^{2-\delta}\int_{\R^3}\int_{\R^3}|v-v_*|^{2+\gamma-\delta}\jap{x-(t+1)v_*}^\delta f_*\jap{x-(t+1)v}^2g\jap{x-(t+1)v'}^2h'\d \sigma\d v_*\d v.
\end{split}
\end{equation*}
\item \textbf{Commutator estimate for the top order and moment loss for the space-time weights.} For the commutator estimate we need to estimate $\jap{x-(t+1)v}g$ in $H^{2s-1}_v$; see \lref{comm_estimates}. We apply this result to $g=\derv{''}{''}{''} g$ and so this estimate poses no issue for all orders other than the top order. For the top order we need to exploit the symmetry as in \cite{HeSnTa19}. But if we proceed in the same way as Henderson--Snelson--Tarfulea, we have
\begin{align*}
&\int_{\R^3}[Q(f,g)\jap{x-(t+1)v}^{2n}g-Q(f,\jap{x-(t+1)v}^{n}g)\jap{x-(t+1)v}^{n}g]\d v\\
&=\int_{\R^3}\int_{\R^3}\int_{\S^2}(\jap{x-(t+1)v'}^n-\jap{x-(t+1)v}^n)f_*g\jap{x-(t+1)v}^ng\d\sigma \d v_*\d v\\
&\quad+ \int_{\R^3}\int_{\R^3}\int_{\S^2}(\jap{x-(t+1)v'}^n-\jap{x-(t+1)v}^n)f_*g\\
&\hspace{7em}\times[\jap{x-(t+1)v'}^ng'-\jap{x-(t+1)v}^ng]\d\sigma \d v_*\d v.
\end{align*}
The first term does not pose much difficulty. For the second term Henderson--Snelson--Tarfulea use Young's inequality to get,
\begin{align*}
\int_{\R^3}&\int_{\R^3}\int_{\S^2}(\jap{x-(t+1)v'}^n-\jap{x-(t+1)v}^n)f_*g[\jap{x-(t+1)v'}^ng'-\jap{x-(t+1)v}^ng]\d\sigma \d v_*\d v\\
&\lesssim \frac{1}{2} \int_{\R^3}\int_{\R^3}\int_{\S^2}f_*[\jap{x-(t+1)v'}^ng'-\jap{x-(t+1)v}^ng]^2\d\sigma \d v_*\d v\\
&\quad +\frac{1}{2} \int_{\R^3}\int_{\R^3}\int_{\S^2}(\jap{x-(t+1)v'}^n-\jap{x-(t+1)v}^n)^2f_*g^2\d\sigma \d v_*\d v
\end{align*}
But then for the second term we will necessarily have a loss of moment issue (for the starred variables)\footnote{The authors in \cite{HeSnTa19} can remedy this by putting $f$ in high weighted $L^\infty_v$ space. We are unable to do so since the proof for propagation of $L^{\infty,m}$ bounds relies heavily on the fact that the weights are exactly $\jap{v}$.}.

Instead, we choose a different route and decompose the commutator term as,
\begin{align*}
&\int_{\R^3}[Q(f,g)\jap{x-(t+1)v}^{2n}g-Q(f,\jap{x-(t+1)v}^{n}g)\jap{x-(t+1)v}^{n}g]\d v\\
&=\int_{\R^3}\int_{\R^3}\int_{\S^2}Bf'_*g' \jap{x-(t+1)v}^{\frac{n}{2}} g[\jap{x-(t+1)v}^{\frac{n}{2}}-\jap{x-(t+1)v'}^{\frac{n}{2}}]\\
&\qquad\times[\jap{x-(t+1)v}^n-\jap{x-(t+1)v'}^n]\d \sigma\d v_*\d v\\
&\quad +\int_{\R^3}\int_{\R^3}\int_{\S^2}B_kf'_*\jap{x-(t+1)v'}^{\frac{n}{2}} g'\jap{x-(t+1)v}^{\frac{n}{2}}  g\\
&\qquad\times[\jap{x-(t+1)v}^n-\jap{x-(t+1)v'}^n]\d \sigma\d v_*\d v.
\end{align*}
It can be easily seen that for $n\leq 2$, we do not have a moment loss issue. This is why we have the restriction on the power of weights that we can propagate\footnote{It is possible that using a hierarchy of weights or finding an appropriate analogue of $L^{\infty,m}$ bound from \cite{HeSnTa19} might help in removing this restriction.}.

The first one we can treat using the cancellation coming from the product of the differences of the weights and by employing the null structure. For the second term we use pre-post collisional change of variables (\lref{pre_post}) to get,
\begin{align*}
&\int_{\R^3}\int_{\R^3}\int_{\S^2}B_k f'_*\jap{x-(t+1)v'} g'\jap{x-(t+1)v}  g\\
&\qquad\times[\jap{x-(t+1)v}^2-\jap{x-(t+1)v'}^2]\d \sigma\d v_*\d v\\
&=\frac{1}{2}\int_{\R^3}\int_{\R^3}\int_{\S^2}B_k(f'_*-f_*)\jap{x-(t+1)v'} g'\jap{x-(t+1)v} g\\
&\qquad\times[\jap{x-(t+1)v}^2-\jap{x-(t+1)v'}^2]\d \sigma\d v_*\d v.
\end{align*}
Thus we use the cancellation coming from $|f_*'-f_*|$ to take care of the angular singularity. Since we already get a $(1+t)$ factor from the difference $|\jap{x-(t+1)v}^2-\jap{x-(t+1)v'}^2|$, we need to make sure that we put $f$ in $W^{1,\eta}$ for $\eta<1$. To achieve this we heavily rely on the machinery from \cite{Sil14} and \cite{ImbSil16}. In addition, we also need to use our adaptation of \lref{boltz_kernel_approx} to the difference of the kernels $|K_{\part f}-K_{\part f}'|$; see \lref{diff_K_outside_ball}. Using this we are indeed able to estimate $f$ in $W^{1,s}$ in \lref{top_order_diff_f}. It is possible to get a similar estimate using Littlewood--Paley theory but then we would have to be satisfied with putting $f$ in $W^{p,s}$ for any $p>1$.
\end{enumerate}
\section{Set-up for the energy estimates}\label{s.set_up}
Let
\begin{equation}\label{e.delta}
\delta=\min\left\{\frac{1-s}{4},\frac{1}{10},\frac{\gamma+2s}{8}\right\}.
\end{equation}

We now derive the equation for $g=e^{d(t)\jap{v}^2}f$. Here $d(t)=d_0(1+(1+t)^{-\delta})$ and $d_0>0$. We let $\mu=e^{-d(t)\jap{v}^2}$.

\begin{equation}\label{e.boltz_gaussian}
\begin{split}
&\part_t g+v_i\part_{x_i} g+\frac{\delta d_0}{(1+t)^{1+\delta}}\color{black}{\jap{v}^2 g}=\Gamma(g,g),\\
&\Gamma(g,g)=\mu^{-1}\int_{\R^3}\int_{\S^2} B(\mu'_*g_*'\mu' g'-\mu_* g_*\mu g)\d \sigma\d v_*.
\end{split}
\end{equation}
By energy conservation we have that $\mu_*=\mu^{-1}\mu_*'\mu'$.\\
Thus we get that, 
\begin{align*}
\Gamma(g,g)&=\int_{\R^3}\int_{\S^2} B \mu_*(g_*'g'-g_*g)\d \sigma\d v_*.
\end{align*}

Now using the change of variables $u=v_*-v$ and $u^{\pm}=(u\pm|u|\sigma)/2$, we have that,
\begin{align*}
\Gamma(g,g)&=\int_{\R^3}\int_{\S^2} B(u,\sigma) \mu(u+v)(g(v+u^+)g(v+u^-)-g(v+u)g(v))\d \sigma\d u,
\end{align*}
where we have suppressed the dependence of the functions on $t$ and $x$.

Now differentiating $\Gamma$ by $\der$ and applying the inverse change of coordinates for $u$, we get $$\der \Gamma(g,g)=\sum_{\substack{|\alpha'|+|\alpha''|=|\alpha|\\|\beta'|+|\beta''|+|\beta'''|=|\beta|\\ |\omega'|+|\omega''|+|\omega'''|=|\omega|}}C^{\alpha,\beta,\omega}_{\alpha',\beta',\omega'}\Gamma_{\beta''',\omega'''}(\derv{'}{'}{'} g,\derv{''}{''}{''} g),$$
where $$\Gamma_{\beta,\omega}(f,h)=\int_{\R^3}\int_{\S^2}B(|v-v_*|,\omega)(\part_{v}^{\beta}Y^{\omega} \mu)_*(f_*'h'-f_*h)\d \sigma\d v_*,$$
where the constant $C^{\alpha,\beta,\omega}_{\alpha',\beta',\omega'}$ denotes the constants coming from the Leibnitz rule.\\
For ease of notation we let $\part_v^\beta Y^\omega \mu=\mu_{\beta,\omega}$.

To set up for the energy estimates we differentiate \eref{boltz_gaussian} by $\der$ to get, \\
\begin{equation}\label{e.diff_boltz}
\begin{split}
&\part_t \der g+v_i\part_{x_i}\der g+\frac{\delta d_0}{(1+t)^{1+\delta}}\jap{v}^2\der g\\ &=[\part_t+v_i\part_{x_i},\der]g-\frac{\delta d_0}{(1+t)^{1+\delta}}(\der(\jap{v}^2 g)-\jap{v}^2\der g)\\
&\quad+\mathlarger{\sum}_{\substack{|\alpha'|+|\alpha''|=|\alpha|\\|\beta'|+|\beta''|+|\beta'''|=|\beta|\\|\omega'|+|\omega''|+|\omega'''|=|\omega|}}C^{\alpha,\beta,\omega}_{\alpha',\beta',\omega'}\Gamma_{\beta''',\omega'''}(\derv{'}{'}{'} g,\derv{''}{''}{''} g).
\end{split}
\end{equation}

\begin{lemma}\label{l.eng_set_up}
Assume that $g$ solves \eref{boltz_gaussian} on $[0,T_*)\times\R^3\times\R^3$. Then for $|\alpha|+|\beta|+|\omega|=10$ and some $T\in[0,T_*)$, we have
\begin{align*}
&\norm{\jap{x-(t+1)v}^2\der g}^2_{L^\infty([0,T];L^2_xL^2_v)}+\norm{(1+t)^{-\frac{1+\delta}{2}}\jap{v}\jap{x-(t+1)v}^2\der g}^2_{L^2([0,T];L^2_xL^2_v)}\\
&\leq \norm{\jap{x-v}^2\der g_{\ini}}^2_{L^2_xL^2_v}+\text{Comm}_1+\text{Comm}_2\\
&\quad +\mathlarger{\sum}_{\substack{|\alpha'|+|\alpha''|=|\alpha|\\|\beta'|+|\beta''|+|\beta'''|=|\beta|\\|\omega'|+|\omega''|+|\omega'''|=|\omega|}}\int_0^{T}\int_{\R^3} \int_{\R^3} \Gamma_{\beta''',\omega'''}(\derv{'}{'}{'} g,\derv{''}{''}{''} g)\jap{x-(t+1)v}^4\der g\d v\d x\d t,
\end{align*}
where,
\begin{equation}\label{e.comm_1}
\text{Comm}_1:=\sum_{\substack{|\alpha'|\leq |\alpha|+1\\ |\beta'|\leq |\beta|-1}}\norm{\jap{x-(t+1)v}^4|\der g||\derv{'}{'}{} g|}_{L^1([0,T];L^1_xL^1_v)},
\end{equation}
and 
\begin{equation}\label{e.comm_2}
\text{Comm}_2:=\sum_{\substack{|\beta'|\leq |\beta|,|\omega'|\leq |\omega|\\ |\beta'|+|\omega'|\leq |\beta|+|\omega|-1}}\norm{\frac{\jap{v}}{(1+t)^{1+\delta}}\jap{x-(t+1)v}^4|\der g||\derv{'}{'}{} g|}_{L^1([0,T];L^1_xL^1_v)}.
\end{equation}
\end{lemma}
\begin{proof}
We multiply \eref{diff_boltz} by $\jap{x-(t+1)v}^{4}\der g$ and then integrate in time, space and velocity. Integrating by parts on the LHS we get,
\begin{align*}
 &\frac{1}{2}\int_{\R^3}\int_{\R^3} \jap{x-(T+1)v}^4(\der g)^2(T_*,x,v)\d v \d x- \frac{1}{2}\int\int \jap{x-v}^4(\der g)^2(0,x,v)\d v \d x \\
&\quad+\int_0^T\int_{\R^3}\int_{\R^3}\jap{x-(t+1)v}^4\frac{\delta d_0\jap{v}^2}{(1+t)^{1+\delta}}(\der g)^2\d v\d x\d t\\
&=\int_0^{T}\int_{\R^3} \int_{\R^3} \jap{x-(t+1)v}^4[\part_t+v_i\part_{x_i},\der]g \der g\d v \d x \d t\\
&\quad+\int_0^{T}\int_{\R^3}\int_{\R^3}\frac{\delta d_0}{(1+t)^{1+\delta}}(\der(\jap{v}^2 g)-\jap{v}^2\der g)\jap{x-(t+1)v}^4\der g\d v\d x\d t\\
&\quad +\smash{\mathlarger{\sum}_{\substack{|\alpha'|+|\alpha''|=|\alpha|\\|\beta'|+|\beta''|+|\beta'''|=|\beta|\\|\omega'|+|\omega''|+|\omega'''|}}}\color{black}{C^{\alpha,\beta,\omega}_{\alpha',\beta',\omega'}}\int_0^{T}\int_{\R^3} \int_{\R^3} \Gamma_{\beta''',\omega'''}(\derv{'}{'}{'} g,\derv{''}{''}{''} g)\jap{x-(t+1)v}^4\\
&\hspace{25em}\times\der g\d v\d x\d t.
\end{align*}
Now $$[\part_t+v_i\part_{x_i},\der]g=\sum_{\substack{|\beta'|+|\beta''|=|\beta|\\|\beta'|=1}}\part_x^{\beta'}\part_v^{\beta''}\part^\alpha_xY^\omega g.$$
Therefore, we can bound this contribution by $\text{Comm}_1$.

Next, the other commutator term arises from $\part_v$ or $(t+1)\part_x+\part_v$ hitting $\jap{v}^2$, this means
\begin{align*}
|\der(\jap{v}^2 g)-\jap{v}^2\der g|&\lesssim \sum_{\substack{|\beta'|\leq|\beta|, |\omega'|\leq |\omega|\\|\beta'|+|\omega'|\leq |\beta|+|\omega|-1}} \jap{v} |\derv{}{'}{'} g|+\sum_{\substack{|\beta'|\leq|\beta|, |\omega'|\leq |\omega|\\|\beta'|+|\omega'|\leq |\beta|+|\omega|-2}} |\derv{}{'}{'} g|.
\end{align*}
Thus, we can bound this contribution by $\text{Comm}_2$.
\end{proof}


\section{Physical decomposition and an alternate definition of \break the collision operator}\label{s.sing}
We first recall the decomposition of the singularity in $\theta$ for the collision operator as in \cite{GrSt11}. Let $\{\chi_k\}_{k=-\infty}^{\infty}$ be a partition of unity on $(0,\infty)$ such that $\norm{\chi_k}_{L^\infty}\leq 1$ and $\text{supp}(\chi_k)\subset [2^{-k-1},2^{-k}].$ For each $k$, we define,
$$B_k:=B_k(|v-v_*|,\sigma):=|v-v_*|^\gamma b\left(\left\langle\frac{v-v_*}{|v-v_*|},\sigma \right\rangle\right)\chi_k(|v-v'|).$$ 
Note that 
\begin{equation}\label{e.v_minus_v'}
|v-v'|^2=\frac{|v-v_*|^2}{2}\left(1-\left\langle\frac{v-v_*}{|v-v_*|},\sigma \right\rangle\right)=|v-v_*|^2\sin^2\frac{\theta}{2}.
\end{equation}
Thus, the condition that $|v-v'|\approx 2^{-k}$ implies that the angle between $\sigma$ and $\frac{v-v_*}{|v-v_*|}$ is comparable to $2^{-k}|v-v_*|^{-1}$.
 
Before we can start estimating the various terms outlined in \sref{set_up} we first recall the machinery developed in \cite{ImbSil16} and \cite{Sil14}. 

We begin by noting that $Q(f,g)$ can be decomposed as $Q_1(f,g)+Q_2(f,g)$, where $Q_1$ contains all the singularity present in $Q$ and 
\begin{align*}
Q_1(f,g)=\int \int f_*'(g'-g)B\d \sigma\d v_*,\\
Q_2(f,g)=(\int \int (f_*'-f_*)B \d \sigma \d v_*)g.
\end{align*}
\begin{lemma}[\cite{Sil14}, Lemma 4.1]\label{l.boltz_kernel}
\textcolor{black}{The term $Q_1$ can be rewritten in the following form}$$Q_1(f,g)=\int_{\R^3}(g'-g)K_f(v,v')\d v'$$
where
$$K_f(v,v')=\frac{2^{2}}{|v'-v|}\int_{w\perp(v'-v)}f(v+w)B(r,\cos \theta)r^{-1}\d w.$$
Note that the integral is supported for $w$ on the hyperplane perpendicular to $v-v'$. Now we have
$$r^2=|v'-v|^2+|w|^2,$$ $$\cos \frac{\theta}{2}=\frac{|w|}{r},$$
$$v'_*=v+w$$
and $$v_*=v'+w.$$
\end{lemma}
\begin{lemma}[\cite{Sil14}, Corollary 4.2]\label{l.boltz_kernel_approx}
$$K_f(v,v')\approx \left(\int_{\{w\cdot(v'-v)=0\}}f(v+w)|w|^{\gamma+2s+1}\d w\right)|v'-v|^{-3-2s}.$$
Here $A\approx B$ means that $A\lesssim B$ and $B\lesssim A$.
\end{lemma}
\begin{lemma}[\cite{ImbSil16},Lemma 3.4 and Lemma 3.5]\label{l.K_bound}
Assume $\gamma+2s\leq 2$. Then for $K_f$ as in \eref{boltz_kernel} and $r>0$ we have the following bounds,
$$\int_{\R^3\backslash B_r(v)}K_f(v,v')\d v'\lesssim r^{-2s}\int_{\R^3}f(z)|z-v|^{\gamma+2s}\d z,$$
and 
$$\int_{\R^3\backslash B_r(v')}K_f(v,v')\d v\lesssim r^{-2s}\int_{\R^3}f(z)|z-v'|^{\gamma+2s}\d z.$$
\end{lemma}
\begin{remark}
Instead of the first inequality, Silvestre proves $$\int_{B_{2r}(v)\backslash B_r(v)}K_f(v,v')\d v'\lesssim r^{-2s}\int_{\R^3}f(z)|z-v|^{\gamma+2s}\d z$$
in \cite{Sil14} (see Corollary 4.4). But from the proof it is clear that our version of the inequality holds too.
\end{remark}
\begin{lemma}[\cite{ImbSil16},Lemma 3.7]\label{l.subt_can}
The two following property holds true for any $R>r>0$,
$$ \int_{B_R(v)\backslash B_r(v)}(v'-v)K_f(v,v')\d v'=0.$$
\end{lemma}
\begin{proof}
The lemma in \cite{ImbSil16} is slightly different and it proves that 
$$ \text{PV}\int_{B_R(v)}(v'-v)K_f(v,v')\d v'=0.$$ 
But we claim that the proof for the lemma at hand is essentially the same. Indeed, we first use the change of variables $v'\to v+\tilde v$. This gives us,
$$\int_{B_R(v)\backslash B_r(v)}(v'-v)K_f(v,v')\d v'=\int_{B_R(0)\backslash B_r(0)}\tilde v K_f(v,v+\tilde v)\d {\tilde v}.$$

Next using polar coordinates and the symmetry property of $K_f$ which asserts that \break $K_f(v,v+\tilde v)=K(v,v-\tilde v)$, we get
\begin{align*}
\int_{B_R(0)\backslash B_r(0)}\tilde v K_f(v,v+\tilde v)\d {\tilde v}&=\int_{r}^R \int_{\part B_z} zw K_f(v,v+zw)\d w\d z \\
&=0.
\end{align*}

We used the symmetry property of $K_f$ to assert that the integral over the surface of the sphere $B_z$ is $0$ for any $z$. 
\end{proof}
\begin{lemma}[Pre-post collision change of variables, \cite{Vil02} pg. 53]\label{l.pre_post}
The change of variables\\
$(v,v_*,\sigma)\to (v',v_*',k)$ has a unit Jacobian, where $k=\frac{v-v_*}{|v-v_*|}$. Moreover, for any function $F(v,v_*,v',v'_*)$, we have 
\begin{align*}
&\int_{\R^3}\int_{\R^3}\int_{\S^2}F(v,v_*,v',v_*')B(|v-v_*|,\sigma)\d v\d v_*\d \sigma\\
&\qquad=\int_{\R^3}\int_{\R^3}\int_{\S^2}F(v',v'_*,v,v_*)B(|v-v_*|,\sigma)\d v\d v_*\d \sigma.
\end{align*}
\end{lemma}
\begin{lemma}[\cite{Sil14}, Lemma A.1]\label{l.change_of_variables_sil}
For any integrable function $F$ (in terms of $v_*$, $v$, $r=|v-v_*|$, $v'$ and/or $v_*'$)
$$\int_{\R^3} \int_{\mathbb{S}^{2}}F \d \sigma \d v_*=2^{2}\int_{\R^N} \frac{1}{|v-v'|}\int_{\{w:w\cdot v'=0\}} \frac{F}{r}\d w\d v'.$$
In the right side, we must write $v_*=v'+w$ and write the values of $r$, $\theta$ and $v_*'$ accordingly.
\end{lemma}
\begin{remark}
In \cite{Sil14} Silvestre states and proves the above lemma for a non-negative function but the proof is the same for any integrable function.
\end{remark}
\begin{lemma}[\cite{ImbSil16}, Lemma A.10]\label{l.change_of_var_2}
Let $F:\R^3\to \R$ be any integrable function and $\mathfrak{n}\in \R$. Then the following identities hold,
$$\int_{\part B_r}\int_{\{w:w\perp \sigma\}}F(w)\d w\d\sigma=2\pi r^{2}\int_{\R^3}\frac{F(z)}{|z|}\d z,$$
$$\int_{\part B_r}\int_{\{w:w\perp \sigma\}}|w|^{\mathfrak{n}}F(\sigma+w)\d w\d\sigma=2\pi r^{2}\int_{\R^3\backslash B_r}F(z)\frac{(z^2-r^2)^{\frac{\mathfrak{n}}{2}}}{|z|}\d z,$$
and 
$$\int_{\part B_r}\int_{\{w:w\perp \sigma\}}\color{black}{\sigma F(\sigma+w)\d w\ d\sigma}=2\pi r^{4}\int_{\R^3\backslash B_r}zF(z)\frac{1}{|z|^3}\d z.$$
Note that the integrals on the left hand side are on spheres and hyperplanes, thus $\d w$ and $\d \sigma$ stand for differential of surface.
\end{lemma}
\begin{remark}
In \cite{ImbSil16} the authors state the change of variables for $\mathfrak{n}=0$. The weight $|w|$ only provides the extra weight factor $(z^2-r^2)^{\frac{1}{2}}$. In fact, the authors in \cite{ImbSil16} actually use the above result with $\mathfrak{n}=\gamma+2s+1$ in Lemma 3.5. 
\end{remark}
\begin{lemma}[\cite{GrSt11}, Proposition A.1]\label{l.change_of_var_GrSt}
For $$\tilde C(v_*)=\int_{\R^n}\int_{\S^{n-1}}\Phi(|v-v_*|)b(\jap{k,\sigma})H(v,v_*,v',v_*')\d \sigma\d v$$ we have $$\tilde C(v_*)=2^{2}\int_{\R^3}\int_{E^{v'}_{v_*}} \frac{\Phi(|v-v_*|)}{|v'-v_*|}\frac{b\left(\left\langle\frac{v-v_*}{|v-v_*|},\frac{2v'-v-v_*}{|2v'-v-v_*|}\right\rangle\right)}{|v-v_*|}H\d {\pi_v}\d v'.$$
Above $H=H(v,v_*,v',v+v_*-v')$ and $E_{v_*}^{v'}$ is the hyperplane $$E_{v_*}^{v'}:=\{v\in\R^n:\jap{v_*-v',v-v'}=0\}.$$
Further, $\d \pi_v$ denotes the Lebesgue measure on this hyperplane.
\end{lemma}
\begin{lemma}[\cite{GrSt11}, Proposition 3.2]\label{l.bound_on_B_tilde}
For $$\tilde B_k:=2^2 \frac{|v-v_*|^\gamma b\left(\left\langle\frac{v-v_*}{|v-v_*|},\frac{2v'-v-v_*}{|2v'-v-v_*|}\right\rangle\right)\chi_k(|v-v'|)}{|v-v_*||v'-v_*|},$$ we have the following bound $$\int_{E^{v'}_{v_*}} \tilde B_k\d \pi_v\lesssim 2^{2sk}|v'-v_*|^{\gamma+2s},$$
where $E_{v_*}^{v'}$ is the hyperplane $$E_{v_*}^{v'}:=\{v\in\R^n:\jap{v_*-v',v-v'}=0\}.$$
\end{lemma}
\begin{remark}
The lemma we assert is not explicitly written as is in \cite{GrSt11} but the assertion is proved as a part of Proposition 3.2.
\end{remark}
\section{Estimates involving the Boltzmann kernel}\label{s.main_est}
The main lemma we will prove in this section is as follows
\begin{lemma}\label{l.main_boltz_estimates}
For $\gamma+2s\in(0,2]$ and $|\alpha|+|\beta|+|\omega|\leq 10$, we have the following bound for any $T\in \R_{\geq 0}$,
\begin{align*}
\sum_{\substack{|\alpha'|+|\alpha''|=|\alpha|\\|\beta'|+|\beta''|+|\beta'''|=|\beta|\\ |\omega'|+|\omega''|+|\omega'''|=|\omega|}}&\int_{0}^{T}\int_{\R^3}\int_{\R^3}\Gamma_{\beta''',\omega'''}(\derv{'}{'}{'} g,\derv{''}{''}{''} g)\jap{x-(t+1)v}^4 \der g\\
&\lesssim \mathbb{I}+\mathbb{II}+\mathbb{III}+\mathbb{IV}+\mathbb{V}+\mathbb{VI}+\mathbb{VII}+\mathbb{VIII}+\mathbb{IX}+\mathbb{X}+\mathbb{XI}+\mathbb{XII},
\end{align*}
where $\delta$ is the same as in \eref{delta} and we adopt the convention that $\mathbb{V}=\mathbb{VII}=\mathbb{X}=0$ for $s<\frac{1}{2}$. 
\begin{equation}\label{e.main_bound_no_sing}
\begin{split}
\mathbb{I}:=\smash{\sum_{\substack{|\alpha'|+|\alpha''|\leq|\alpha|\\|\beta'|+|\beta''|\leq|\beta|\\|\omega'|+|\omega''|\leq|\omega|\\|\alpha'|+|\beta'|+|\omega'|\leq 7}}}&\int_0^{T}(1+t)^{1+\delta}\norm{(1+t)^{-\frac{1}{2}-\frac{\delta}{2}}\jap{x-(t+1)v}^2\jap{v}\derv{''}{''}{''} g}_{L^2_xL^2_v} \\
&\qquad \times \norm{\jap{x-(t+1)v}^2\derv{'}{'}{'}g}_{L^\infty_xL^2_v}\\
&\qquad\times\norm{(1+t)^{-\frac{1}{2}-\frac{\delta}{2}}\jap{x-(t+1)v}^2\jap{v}\der g}_{L^2_xL^2_v}\d t\\[10pt]
\quad+\smash{\sum_{\substack{|\alpha'|+|\alpha''|\leq|\alpha|\\|\beta'|+|\beta''|\leq|\beta|\\|\omega'|+|\omega''|\leq|\omega|\\|\alpha'|+|\beta'|+|\omega'|\geq 8}}}&\int_0^{T}(1+t)^{1+\delta}\norm{(1+t)^{-\frac{1}{2}-\frac{\delta}{2}}\jap{x-(t+1)v}^2\jap{v}\derv{''}{''}{''} g}_{L^\infty_x L^2_v} \\
&\qquad \times \norm{\jap{x-(t+1)v}^2\derv{'}{'}{'}g}_{L^2_xL^2_v}\\
&\qquad\times\norm{(1+t)^{-\frac{1}{2}-\frac{\delta}{2}}\jap{x-(t+1)v}^2\jap{v}\der g}_{L^2_xL^2_v}\d t,
\end{split}
\end{equation}

\begin{equation}\label{e.main_bound_Q2}
\begin{split}
\mathbb{II}:=\smash{\sum_{\substack{|\alpha'|+|\alpha''|\leq|\alpha|\\|\beta'|+|\beta''|\leq|\beta|\\|\omega'|+|\omega''|\leq|\omega|\\|\alpha'|+|\beta'|+|\omega'|\leq 1}}}
&\int_0^T (1+t)^{1+2\delta}\norm{\derv{'}{'}{'} f\jap{v}^{\gamma+2s}}^s_{L^\infty_xW^{2,1}_v}\norm{\derv{'}{'}{'}f\jap{v}^{\gamma+2s}}^{1-s}_{L^\infty_xL^1_v}\\
&\hspace{4em}\times\norm{(1+t)^{-\frac{1+2\delta}{2}}\jap{v}^{\frac{\gamma+2s}{2}}\jap{x-(t+1)v}^2 \derv{''}{''}{''} g}_{L^2_xL^2_v}\\
&\hspace{5em}\times\norm{(1+t)^{-\frac{1+2\delta}{2}}\jap{v}^{\frac{\gamma+2s}{2}}\jap{x-(t+1)v}^2 \der g}_{L^2_xL^2_v}\d t,
\end{split}
\end{equation}

\begin{equation}\label{e.main_bound_comm_1}
\begin{split}
\mathbb{III}:=\smash{\sum_{\substack{|\alpha'|+|\alpha''|\leq|\alpha|\\|\beta'|+|\beta''|\leq|\beta|\\|\omega'|+|\omega''|\leq|\omega|\\|\alpha'|+|\beta'|+|\omega'|\leq 7}}}&\int_0^{T}(1+t)^{2+\delta}\norm{(1+t)^{-\frac{1+\delta}{2}}\jap{x-(t+1)v}^2\jap{v}\derv{''}{''}{''} g}_{L^2_xL^2_v}\\
&\qquad\times \norm{\jap{x-(t+1)v}\derv{'}{'}{'} g}_{L^\infty_xL^1_v}\\
&\qquad\times\norm{(1+t)^{-\frac{1+\delta}{2}}\jap{x-(t+1)v}^2\jap{v}\der g}_{L^2_xL^2_v}\d t\\[10pt]
\quad+\smash{\sum_{\substack{|\alpha'|+|\alpha''|\leq|\alpha|\\|\beta'|+|\beta''|\leq|\beta|\\|\omega'|+|\omega''|\leq|\omega|\\|\alpha'|+|\beta'|+|\omega'|\geq 8}}}&\int_0^{T}(1+t)^{2+\delta}\norm{(1+t)^{-\frac{1+\delta}{2}}\jap{x-(t+1)v}^2\jap{v}\derv{''}{''}{''} g}_{L^\infty_xL^2_v}\\
&\qquad\times \norm{\jap{x-(t+1)v}\derv{'}{'}{'} g}_{L^2_xL^1_v}\\
&\qquad\times\norm{(1+t)^{-\frac{1+\delta}{2}}\jap{x-(t+1)v}^2\jap{v}\der g}_{L^2_xL^2_v}\d t,
\end{split}
\end{equation}

\begin{equation}\label{e.main_bound_comm_2}
\begin{split}
\mathbb{IV}:=\smash{\sum_{\substack{|\alpha'|+|\alpha''|\leq|\alpha|\\|\beta'|+|\beta''|\leq|\beta|\\|\omega'|+|\omega''|\leq|\omega|\\|\alpha'|+|\beta'|+|\omega'|\leq 7}}}&\int_0^{T}(1+t)^{3-\delta}\norm{(1+t)^{-\frac{1+\delta}{2}}\jap{x-(t+1)v}^2\jap{v}\derv{''}{''}{''} g}_{L^2_xL^2_v}\\
&\qquad\times \norm{\jap{x-(t+1)v}^{2\delta}\derv{'}{'}{'} g}_{L^\infty_xL^1_v}\\
&\qquad\times\norm{(1+t)^{-\frac{1+\delta}{2}}\jap{x-(t+1)v}^2\jap{v}\der g}^2_{L^2_xL^2_v}\\[10pt]
\quad+\smash{\sum_{\substack{|\alpha'|+|\alpha''|\leq|\alpha|\\|\beta'|+|\beta''|\leq|\beta|\\|\omega'|+|\omega''|\leq|\omega|\\|\alpha'|+|\beta'|+|\omega'|\geq 8}}}&\int_0^{T_*}(1+t)^{3-\delta}\norm{(1+t)^{-\frac{1+\delta}{2}}\jap{x-(t+1)v}^2\jap{v}\derv{''}{''}{''} g}_{L^\infty_xL^2_v}\\
&\qquad\times \norm{\jap{x-(t+1)v}^{2\delta}\derv{'}{'}{'} g}_{L^2_xL^1_v}\\
&\qquad\times\norm{(1+t)^{-\frac{1+\delta}{2}}\jap{x-(t+1)v}^2\jap{v}\der g}_{L^2_xL^2_v},
\end{split}
\end{equation}

\begin{equation}\label{e.main_bound_top}
\begin{split}
\mathbb{V}:=&\mathds{1}_{s\in \left[\frac{1}{2},1\right)} \int_0^{T}(1+t)^{2+\delta}\norm{g}^{\frac{3}{2}-s}_{L^\infty_xL^1_v}\norm{\part_{v_i}g}^{s-\frac{1}{2}}_{L^\infty_xL^1_v}\norm{(1+t)^{-\frac{1}{2}-\frac{\delta}{2}}\jap{v}^{\frac{2s+\gamma}{2}}\der g}_{L^2_xL^2_v}\d t\\
&\quad + \mathds{1}_{s\in \left[\frac{1}{2},1\right)} \int_0^{T}(1+t)^{2+\delta}\norm{g}^{1-s}_{L^\infty_xL^1_v}\norm{\part_{v_i}g}^{s}_{L^\infty_xL^1_v}\norm{(1+t)^{-\frac{1}{2}-\frac{\delta}{2}}\jap{v}^{\frac{2s+\gamma}{2}}\der g}^2_{L^2_xL^2_v}\d t,
\end{split}
\end{equation}

\textcolor{black}{\begin{equation}\label{e.main_bound_top_2}
\begin{split}
\mathbb{VI}:=\smash{\sum_{\substack{|\alpha'|+|\beta'|+|\omega'|\leq 1\\|\alpha'|+|\alpha''|\leq|\alpha|\\|\beta'|+|\beta''|\leq|\beta|\\|\omega'|+|\omega''|\leq|\omega|}}}& \int_0^{T}(1+t)^{1+\delta}\norm{\part_{v_i} \derv{'}{'}{'}g}_{L^\infty_xL^1_v}\norm{(1+t)^{-\frac{1+\delta}{2}}\jap{x-(t+1)v}^2\jap{v}\derv{''}{''}{''} g}_{L^2_xL^2_v}\\
&\qquad\times\norm{(1+t)^{-\frac{1+\delta}{2}}\jap{x-(t+1)v}^2\jap{v}\der g}_{L^2_xL^2_v} \d t,
\end{split}
\end{equation}}

\begin{equation}\label{e.main_bound_pen_1}
\begin{split}
\mathbb{VII}:= \smash{\sum_{\substack{|\alpha'|+|\beta'|+|\omega'|=1\\|\alpha'|+|\alpha''|\leq|\alpha|\\|\beta'|+|\beta''|\leq|\beta|\\|\omega'|+|\omega''|\leq|\omega|}}}&\mathds{1}_{s\in \left[\frac{1}{2},1\right)} \left[\int_0^{T}(1+t)^{1+2\delta}\norm{(1+t)^{-\frac{1+2\delta}{2}}\jap{v}\jap{x-(t+1)v}^2\derv{''}{''}{''} g}_{L^2_xH^{(2s-1)^+}_v}\right.\\
&\qquad\left.\times\norm{\part_{v_i} \derv{'}{'}{'} g}_{L^\infty_xL^1_v}\norm{(1+t)^{-\frac{1+2\delta}{2}}\jap{v}\jap{x-(t+1)v}^2\der g}_{L^2_xL^2_v}\d t\right.\\
&\quad\left.+\int_0^{T}(1+t)^{1+2\delta}\norm{(1+t)^{-\frac{1+2\delta}{2}}\jap{v}\jap{x-(t+1)v}^2\derv{''}{''}{''} g}_{L^2_xH^{(2s-1)}_v}\right.\\
&\qquad\left.\times\norm{\part_{v_i} \derv{'}{'}{'} g}_{L^\infty_xL^1_v}\norm{(1+t)^{-\frac{1+2\delta}{2}}\jap{v}\jap{x-(t+1)v}^2\der g}_{L^2_xL^2_v}\d t\right],
\end{split}
\end{equation}

\begin{equation}\label{e.main_bound_pen_2}
\begin{split}
\mathbb{VIII}:=\smash{\sum_{\substack{|\alpha'|+|\beta'|+|\omega'|=1\\|\alpha'|+|\alpha''|\leq|\alpha|\\|\beta'|+|\beta''|\leq|\beta|\\|\omega'|+|\omega''|\leq|\omega|}}}&\int_0^{T}(1+t)^{|\beta'|+1+2\delta}\norm{\derv{'}{'}{'} g}_{L^\infty_xL^1_v}\\
&\qquad\times\norm{(1+t)^{-\frac{1+2\delta}{2}}\jap{x-(t+1)v}^2\jap{v}\derv{''}{''}{''} g}^2_{L^2_xH^s_v}\d t,
\end{split}
\end{equation}

\begin{equation}\label{e.main_bound_pen_2_prime}
\begin{split}
\mathbb{IX}:= \smash{\sum_{\substack{|\alpha'|+|\beta'|+|\omega'|=1\\|\alpha'|+|\alpha''|\leq|\alpha|\\|\beta'|+|\beta''|\leq|\beta|\\|\omega'|+|\omega''|\leq|\omega|}}}& \int_0^{T}(1+t)^{|\beta'|+1+\delta}\norm{\derv{'}{'}{'} f\jap{v}^{\gamma+2s}}^s_{L^\infty_xW^{2,1}_v}\norm{\derv{'}{'}{'}f\jap{v}^{\gamma+2s}}^{1-s}_{L^\infty_xL^1_v}\\
&\qquad\times\norm{(1+t)^{-\frac{1+\delta}{2}}\jap{x-(t+1)v}^2\jap{v}\derv{''}{''}{''} g}^2_{L^2_xL^2_v}\d t,
\end{split}
\end{equation}

\begin{equation}\label{e.main_bound_comm_3}
\begin{split}
\mathbb{X}:=\smash{\sum_{\substack{|\alpha'|+|\alpha''|\leq|\alpha|\\|\beta'|+|\beta''|\leq|\beta|\\|\omega'|+|\omega''|\leq|\omega|\\|\alpha'|+|\beta'|+|\omega'|\leq 7}}}&\mathds{1}_{s\in \left[\frac{1}{2},1\right)} \left[\int_0^{T}(1+t)^{2+2\delta}\norm{(1+t)^{-\frac{1+2\delta}{2}}\derv{''}{''}{''}g\jap{x-(t+1)v}^2\jap{v}}_{L^2_xH^{(2s-1)^+}_v}\right.\\
&\qquad \times\left.\norm{\derv{'}{'}{'} g}_{L^\infty_xL^1_v}\norm{(1+t)^{-\frac{1+2\delta}{2}}\der g\jap{x-(t+1)v}^2\jap{v}}_{L^2_xL^2_v}\d t\right.\\
&\quad+\left.\int_0^{T}(1+t)^{2+2\delta}\norm{(1+t)^{-\frac{1+2\delta}{2}}\derv{''}{''}{''}g\jap{x-(t+1)v}^2\jap{v}}_{L^2_xH^{2s-1}_v}\right.\\
&\qquad \times\left.\norm{\derv{'}{'}{'} g}_{L^\infty_xL^1_v}\norm{(1+t)^{-\frac{1+2\delta}{2}}\der g\jap{x-(t+1)v}^2\jap{v}}_{L^2_xL^2_v}\d t\right]\\[10pt]
\quad+\smash{\sum_{\substack{|\alpha'|+|\alpha''|\leq|\alpha|\\|\beta'|+|\beta''|\leq|\beta|\\|\omega'|+|\omega''|\leq|\omega|\\|\alpha'|+|\beta'|+|\omega'|\geq 8}}}&\mathds{1}_{s\in \left[\frac{1}{2},1\right)} \left[\int_0^{T}(1+t)^{2+2\delta}\norm{(1+t)^{-\frac{1+2\delta}{2}}\derv{''}{''}{''}g\jap{x-(t+1)v}^2\jap{v}}_{L^\infty_xH^{(2s-1)^+}_v}\right.\\
&\qquad \left.\times\norm{\derv{'}{'}{'} g}_{L^2_xL^1_v}\norm{(1+t)^{-\frac{1+2\delta}{2}}\der g\jap{x-(t+1)v}^2\jap{v}}_{L^2_xL^2_v}\d t\right.\\
&\quad\left.+\int_0^{T}(1+t)^{2+2\delta}\norm{(1+t)^{-\frac{1+2\delta}{2}}\derv{''}{''}{''}g\jap{x-(t+1)v}^2\jap{v}}_{L^\infty_xH^{2s-1}_v}\right.\\
&\qquad\left. \times\norm{\derv{'}{'}{'} g}_{L^2_xL^1_v}\norm{(1+t)^{-\frac{1+2\delta}{2}}\der g\jap{x-(t+1)v}^2\jap{v}}_{L^2_xL^2_v}\d t\right],
\end{split}
\end{equation}

\begin{equation}\label{e.main_bound_comm_4}
\begin{split}
\mathbb{XI}:=\smash{\sum_{\substack{|\alpha'|+|\alpha''|\leq|\alpha|\\|\beta'|+|\beta''|\leq|\beta|\\|\omega'|+|\omega''|\leq|\omega|\\|\alpha''|+|\beta''|+|\omega''|<|\alpha|+|\beta|+|\omega|\\\ 1\leq|\alpha'|+|\beta'|+|\omega'|\leq 7}}}&\int_0^{T}(1+t)^{1+2\delta}\norm{(1+t)^{-\frac{1+2\delta}{2}}\jap{x-(t+1)v}^2\jap{v}\part_{v_i} \derv{''}{''}{''} g}_{L^2_xL^2_v}\\
&\qquad\times\norm{\derv{'}{'}{'} g}_{L^\infty_xL^1_v}\\
&\qquad\times\norm{(1+t)^{-\frac{1+2\delta}{2}}\jap{x-(t+1)v}^2\jap{v}\der g}_{L^2_xL^2_v}\d t\\[10pt]
\quad+\smash{\sum_{\substack{|\alpha'|+|\alpha''|\leq|\alpha|\\|\beta'|+|\beta''|\leq|\beta|\\|\omega'|+|\omega''|\leq|\omega|\\|\alpha''|+|\beta''|+|\omega''|<|\alpha|+|\beta|+|\omega|\\\|\alpha'|+|\beta'|+|\omega'|\geq 8}}}&\int_0^{T}(1+t)^{1+2\delta}\norm{(1+t)^{-\frac{1+2\delta}{2}}\jap{x-(t+1)v}^2\jap{v}\part_{v_i} \derv{''}{''}{''} g}_{L^\infty_xL^2_v}\\
&\qquad \times\norm{\derv{'}{'}{'} g}_{L^2_xL^1_v}\\
&\qquad\times\norm{(1+t)^{-\frac{1+2\delta}{2}}\jap{x-(t+1)v}^2\jap{v}\der g}_{L^2_xL^2_v}\d t,
\end{split}
\end{equation}

\begin{equation}\label{e.main_bound_H2s}
\begin{split}
\mathbb{XII}:=\smash{\sum_{\substack{|\alpha'|+|\alpha''|\leq|\alpha|\\|\beta'|+|\beta''|\leq|\beta|\\|\omega'|+|\omega''|\leq|\omega|\\2\leq |\alpha'|+|\beta'|+|\omega'|\leq 7}}}&\int_0^{T}(1+t)^{1+2\delta}\norm{(1+t)^{-\frac{1+2\delta}{2}}\jap{x-(t+1)v}^2\jap{v} \derv{''}{''}{''} g}_{L^2_xH^{2s}_v}\\
&\qquad \times\norm{\derv{'}{'}{'} g}_{L^\infty_xL^1_v}\norm{(1+t)^{-\frac{1+2\delta}{2}}\jap{x-(t+1)v}^2\jap{v}\der g}_{L^2_xL^2_v}\d t\\[10pt]
\quad+\smash{\sum_{\substack{|\alpha'|+|\alpha''|\leq|\alpha|\\|\beta'|+|\beta''|\leq|\beta|\\|\omega'|+|\omega''|\leq|\omega|\\|\alpha'|+|\beta'|+|\omega'|\geq 8}}}&\int_0^{T}(1+t)^{1+2\delta}\norm{(1+t)^{-\frac{1+2\delta}{2}}\jap{x-(t+1)v}^2\jap{v} \derv{''}{''}{''} g}_{L^\infty_xH^{2s}_v}\\
&\qquad \times\norm{\derv{'}{'}{'} g}_{L^2_xL^1_v}\times\norm{(1+t)^{-\frac{1+2\delta}{2}}\jap{x-(t+1)v}^2\jap{v}\der g}_{L^2_xL^2_v}\d t.
\end{split}
\end{equation}
\end{lemma}
We first prove some preliminary embedding results that will be useful throughout this section.
\begin{lemma}\label{l.H-L-S_for_infinity_lp}
Let $\nu \in \left(0,\frac{3}{2}\right)$ and $I_v f:=\int|v-v_*|^{-\nu}f(v_*) \d v_*$ then for $f$ in $L^1_v\cap L^{p}_v$ and $2\geq p> \frac{3}{3-\nu}$ we have that 
$$\norm{I_v f}_{L^\infty_v}\lesssim \norm{f}_{L^1_v}+\norm{f}_{L^{p}_v}.$$
\end{lemma}
\begin{proof}
We know by duality of Fourier transform that $I_\nu f=c\int \widehat{I_\nu f} e^{2\pi i v\cdot \xi}\d \xi$.

Now note that $\widehat{I_\nu f}(\xi)=c\xi^{-(3-\nu)}\hat f$. See chapter 3 in \cite{Stein} for more details on Reisz kernel.\\
Thus we have 
\begin{align*}
|I_\nu f(v)|&\leq c\int |\widehat{I_\nu f}| \d \xi\\
&\lesssim \int |\xi^{-3+\nu}||\hat f|\d \xi\\
&\lesssim \int_{|\xi|\leq 1} |\xi^{-3+\nu}||\hat f| \d \xi+ \int_{|\xi|\geq 1} |\xi^{-3+\nu}||\hat f| \d \xi.
\end{align*}
Since $\nu>0$, we have that $$\int_{|\xi|\leq 1} |\xi^{-3+\nu}||\hat f| \d \xi \lesssim \norm{\hat f}_{L^\infty_\xi}\lesssim \norm{f}_{L^1_v}.$$

For the second term we use the fact that $f \in L^p_v$. Indeed,
\begin{align*}
\int _{|\xi|>1} |\xi^{-3+\nu}||\hat f| \d \xi &\leq \int \jap{\xi}^{-3+\nu} |\hat f| \d \xi\\
&\leq \norm{\hat f}_{L^{p'}_\xi}\norm{\jap{\xi}^{-3+\nu}}_{L^p_\xi}\\
&\lesssim \norm{f}_{L^p_v},
\end{align*}
where the second inequality follows by the Holder's inequality. Since $p>\frac{3}{3-\nu}$, we have $p(3-\nu)>3$, which implies that $\norm{\jap{\xi}^{-3+\nu}}_{L^p_\xi}$ is bounded.\\
We also used the following inequality for fourier transform
$$||\hat f||_{L^{p'}}\lesssim ||f||_{L^p},$$
where $p$ and $p'$ are Holder conjugates of $1$ and $2\geq p\geq 1$.

Thus we have that $|I_\nu f(v)|\leq ||f||_{L^1_v}+||f||_{L^p_v}.$
\end{proof}
\begin{lemma}\label{l.exp_bound}
For every $l\in \N$ and $m\in \N\cup\{0\}$, the following estimate holds with a constant depending on $l,\gamma$ and $d_0$ for any $(t,x,v) \in [0,T)\times \R^3\times \R^3$
$$\jap{v}^l\jap{x-(t+1)v}^m|\der f|(t,x,v)\lesssim \sum \limits_{|\beta'|\leq |\beta|}\jap{x-(t+1)v}^m|\derv{}{'}{} g|(t,x,v).$$
\end{lemma}
\begin{proof}
This is immediate from differentiating $g$ and using that $\jap{v}^ne^{-d(t)\jap{v}}\lesssim_n 1$ for all $n\in \N$.
\end{proof}
\subsection{Estimates for the top order} This corresponds to $|\alpha''|+|\beta''|+|\omega''|=|\alpha|+|\beta|+|\omega|$ in \eref{diff_boltz}.

Let \underline{$G_{\alpha,\beta,\omega}:=\jap{x-(t+1)v}^2\der g$}. In the sequel, the integral is taken as\\
 $\int_0^{T_*}\int \int \int \int_{\mathbb{S}^2} B(v-v_*,\sigma)H(t,x,v,v_*,v') \d \sigma \d v_*\d v\d x\d t$, but the explicit dependence is dropped for brevity. In addition, we also drop the dependence of $G$ on $\alpha$, $\beta$ and $\omega$.

Since we will be using the following for all orders, we present it more generally for \textbf{arbitrary functions} $\mf$, $\mg$ and $\mh$. These are not necessarily the same as $f$ and $g$ from \eref{boltz} and \eref{boltz_gaussian}, respectively.\\
First we use the decomposition of singularity from \sref{sing} to get
\begin{equation}\label{e.sing_decom_top_order}
	B(\color{black}{\partial_v^\beta\mu)}_*(\mathfrak{f}_*'\mathfrak{g}'-\mathfrak{f}_*\mathfrak{g})=\sum_{k=-\infty}^{\infty} B_k (\color{black}{\partial_v^\beta\mu)_*}(\mathfrak{f}_*'\mathfrak{g}'-\mathfrak{f}_*\mathfrak{g}).
\end{equation}
Let us define 
\begin{equation}\label{e.triv_est_1}
\color{black}{I_k^+(\mf,\mg,\mh):=\int_0^{T}\int_{\R^3} \int_{\R^3}\int_{\R^3}\int_{\S^2}B_k (\partial_v^\beta\mu)_*{\mf}_*'\mg'\jap{x-(t+1)v}^4 \mh\d \sigma\d v_*\d v \d x\d t},
\end{equation}
and 
\begin{equation}\label{e.triv_est_2}
\color{black}{I_k^-(\mf,\mg,\mh):=\int_0^{T}\int_{\R^3} \int_{\R^3}\int_{\R^3}\int_{\S^2}B_k (\partial_v^\beta\mu)_*\mf_*\mg\jap{x-(t+1)v}^4  \mh\d \sigma\d v_*\d v \d x\d t}.
\end{equation}

When $k\leq 0$ (i.e. when $|v-v'|\geq \frac{1}{2}$), we trivially estimate the collisional term by estimating $I_k^-$ and $I_k^+$ on their own. But when $k>0$, we exploit cancellations due to the structure of the collisions to care of the angular singularity.

\textbf{In the following we drop the integration to keep the notation lean and use $\equiv$ to denote equality after integrating over $\d \sigma$, $\d v_*$, $\d v$, $\d x$ and $\d t$.}
\begin{align}
B_k\mu_*& \left(g_*'(\der g)'-g_*\der g\right)\jap{x-(t+1)v}^4\der g\nonumber\\
&\equiv B_k \left(\mu_*'g_*'(\der g)'-\mu_*g_*\der g\right)\jap{x-(t+1)v}^2G\nonumber\\
&\quad-B_k(\mu_*'-\mu_*)g_*'(\der g)'\jap{x-(t+1)v}^2 G\label{e.top_order_Gaussian_s_less_than_half}.
\end{align}
Futher we have,
\begin{align}
B_k &\left(\mu_*'g_*'(\der g)'-\mu_*g_*\der g\right)\jap{x-(t+1)v}^2G\nonumber\\
&\equiv Q_k(\mu g,G)G\label{e.top_order_symm}\\
&\quad+B_k(\mu g)'_*(\der g)' G[\jap{x-(t+1)v}^2-\jap{x-(t+1)v'}^2]\label{e.top_order_s_less_than_half}\\
&\equiv Q_k(\mu g,G)G\nonumber\\
&\quad +B_k(\mu g)'_*(\der g)' \jap{x-(t+1)v}\der g[\jap{x-(t+1)v}-\jap{x-(t+1)v'}]\nonumber\\
&\qquad\times[\jap{x-(t+1)v}^2-\jap{x-(t+1)v'}^2]\label{e.top_order_diff_weights_mult}\\
&\quad +B_k(\mu g)'_*\jap{x-(t+1)v'}(\der g)'\jap{x-(t+1)v} \der g\nonumber\\
&\qquad\times[\jap{x-(t+1)v}^2-\jap{x-(t+1)v'}^2]\nonumber\\
&\equiv Q_k(\mu g,G)G\nonumber\\
&\quad +B_k(\mu g)'_*(\der g)' \jap{x-(t+1)v}\der g[\jap{x-(t+1)v}-\jap{x-(t+1)v'}]\nonumber\\
&\qquad\times[\jap{x-(t+1)v}^2-\jap{x-(t+1)v'}^2]\nonumber\\
&\quad +\frac{1}{2}B_k[(\mu g)'_*-(\mu g)_*]\jap{x-(t+1)v'}(\der g)'\jap{x-(t+1)v} \der g\nonumber\\
&\qquad\times[\jap{x-(t+1)v}^2-\jap{x-(t+1)v'}^2]\label{e.top_order_diff_f},
\end{align}
where $Q_k(f,g)$ is basically the same as $Q(f,g)$ in \eref{boltz_kernel} with $B$ replaced with $B_k$. Moreover, we used pre-post collision change of variables from \lref{pre_post} to get \eref{top_order_diff_f}.

For the term involving difference of Gaussians we have,
\begin{align}
B_k(\mu_*'-&\mu_*)g_*'(\der g)'\jap{x-(t+1)v}^2G \nonumber\\
&\equiv B_k(\mu_*'-\mu_*)g_*'G' G\nonumber\\
&\quad +B_k (\mu_*'-\mu_*)g_*' (\der g)' G[\jap{x-(t+1)v}^2-\jap{x-(t+1)v'}^2]\label{e.top_order_Gaussian_1}\\
&\equiv \frac{1}{2}B_k(\mu_*'-\mu_*)(g_*'-g_*)G' G\label{e.top_order_Gaussian_2}\\
&\quad +B_k (\mu_*'-\mu_*)g_*' (\der g)' G[\jap{x-(t+1)v}^2-\jap{x-(t+1)v'}^2]\nonumber.
\end{align}
We again used pre-post collision change of variables (\lref{pre_post}) to get \eref{top_order_Gaussian_2}.

For $s<\frac{1}{2}$, we use 
\begin{align*}
\int_0^{T}\int_{\R^3} \int_{\R^3}\int_{\R^3}\int_{\S^2}B_k \mu_*\left(g_*'(\der g)'-g_*\der g\right)&\jap{x-(t+1)v}^4 \der g\d \sigma\d v_*\d v \d x\d t\\
&=\eref{top_order_symm}+\eref{top_order_s_less_than_half}+\eref{top_order_Gaussian_s_less_than_half}.
\end{align*}

For $s\geq \frac{1}{2}$, we use 
\begin{align*}
\int_0^{T}\int_{\R^3} \int_{\R^3}\int_{\R^3}\int_{\S^2}B_k \mu_*(g_*'(\der g)'-&g_*\der g)\jap{x-(t+1)v}^4 \der g\d \sigma\d v_*\d v \d x\d t\\
&=\eref{top_order_symm}+\eref{top_order_diff_weights_mult}+\eref{top_order_diff_f}+\eref{top_order_Gaussian_1}+\eref{top_order_Gaussian_2}.
\end{align*}

Now we start bounding the various terms from above. We remark that the following two lemmas have already been proved in \cite{GrSt11} the only difference is that we have space-time weights which does not complicate the proof much. That said, we reproduce the result for the ease of the reader.
\begin{lemma}\label{l.trivial_est_1}
Let  $\mf$, $\mg$ and $\mh$ be any smooth functions and $I_k^-(\mf,\mg,\mh)$ be defined as in \eref{triv_est_1} we have the following estimates
\begin{align*}
|I_k^-(\mf,\mg,\mh)|\lesssim 2^{2sk}\int_0^{T}&(1+t)^{1+\delta}\norm{\mf}_{L^\infty_xL^2_v}\norm{(1+t)^{-\frac{1}{2}-\frac{\delta}{2}}\jap{x-(t+1)v}^2\jap{v}^{\frac{\gamma+2s}{2}}\mg}_{L^2_xL^2_v}\\
&\qquad \times\norm{(1+t)^{-\frac{1}{2}-\frac{\delta}{2}}\jap{x-(t+1)v}^2\jap{v}^{\frac{\gamma+2s}{2}}\mh}_{L^2_xL^2_v},
\end{align*}
and
\begin{align*}
|I_k^-(\mf,\mg,\mh)|\lesssim 2^{2sk}\int_0^{T}&(1+t)^{1+\delta}\norm{\mf}_{L^2_xL^2_v}\norm{(1+t)^{-\frac{1}{2}-\frac{\delta}{2}}\jap{x-(t+1)v}^2\jap{v}^{\frac{\gamma+2s}{2}}\mg}_{L^\infty_xL^2_v}\\
&\qquad \times\norm{(1+t)^{-\frac{1}{2}-\frac{\delta}{2}}\jap{x-(t+1)v}^2\jap{v}^{\frac{\gamma+2s}{2}}\mh}_{L^2_xL^2_v}.
\end{align*}
\end{lemma}
\begin{proof}
First note that
$$\int_{\S^2}B_k\d \sigma\lesssim |v-v_*|^\gamma\int_{2^{-k-1}|v-v_*|^{-1}}^{2^{-k}|v-v_*|^{-1}} \theta^{-1-2s}\d \theta\lesssim 2^{2sk}|v-v_*|^{\gamma+2s}.$$

Thus we can take care of the second term easily using Cauchy--Schwarz,
\begin{align*}
\int_{\R^3}\int_{\R^3}\int_{\S^2}B_k|(\partial_v^\beta\mu \mf)_*&\mg \jap{x-(t+1)v}^4\mh|\d \sigma \d v_*\d v\\
&\lesssim 2^{2sk}\int_{\R^3}\int_{\R^3}|v-v_*|^{\gamma+2s}|(\color{black}{\partial_v^\beta\mu} \mf)_*\mg\jap{x-(t+1)v}^4 \mh| \d v_*\d v\\
&\lesssim 2^{2sk}\left(\int_{\R^3}\int_{\R^3}\color{black}{\partial_v^\beta\mu}_*\mf_*^2 \jap{x-(t+1)v}^4|\mg|^2\jap{v}^{\gamma+2s}\d v_*\d v\right)^{\frac{1}{2}}\\
&\quad \times\left(\int_{\R^3}\jap{x-(t+1)v}^4|\mh|^2\jap{v}^{-\gamma-2s}\int_{\R^3}\color{black}{\partial_v^\beta\mu}_* |v-v_*|^{2\gamma+4s}\d v_*\d v\right)^{\frac{1}{2}}\\
&\lesssim 2^{2sk}\norm{\mf}_{L^2_v}\norm{\jap{x-(t+1)v}^2\jap{v}^{\frac{\gamma+2s}{2}}\mg}_{L^2_v}\norm{\jap{x-(t+1)v}^2\jap{v}^{\frac{\gamma+2s}{2}}\mh}_{L^2_v},
\end{align*}
where we used the inequality
\begin{align*}
\int_{\R^3}{\partial_v^\beta\mu}_* |v-v_*|^{2\gamma+4s}\d v_*&\lesssim \int_{\R^3}{\partial_v^\beta\mu}_* (\jap{v}\jap{v_*})^{2\gamma+4s}\d v_* \\
&\lesssim \jap{v}^{2\gamma+4s}.
\end{align*}We get the required result by Cauchy--Schwarz in space and by multiplying and dividing by $(1+t)^{1+\delta}.$
\end{proof}
\begin{lemma}\label{l.trivial_est_2}
Let  $\mf$, $\mg$ and $\mh$ be any smooth functions and $I_k^+(\mf,\mg,\mh)$ be defined as in \eref{triv_est_2} we have the following estimates
\begin{align*}
|I_k^+(\mf,\mg,\mh)|\lesssim 2^{2sk}\int_0^{T}&(1+t)^{1+\delta}\norm{\jap{x-(t+1)v}^2 \mf}_{L^\infty_xL^2_v}\norm{(1+t)^{-\frac{1}{2}-\frac{\delta}{2}}\jap{x-(t+1)v}^2\jap{v}^{\frac{\gamma+2s}{2}}\mg}_{L^2_xL^2_v}\\
&\qquad \times\norm{(1+t)^{-\frac{1}{2}-\frac{\delta}{2}}\jap{x-(t+1)v}^2\jap{v}^{\frac{\gamma+2s}{2}}\mh}_{L^2_xL^2_v},
\end{align*}
and
\begin{align*}
|I_k^+(\mf,\mg,\mh)|\lesssim 2^{2sk}\int_0^{T}&(1+t)^{1+\delta}\norm{\jap{x-(t+1)v}^2 \mf}_{L^2_xL^2_v}\norm{(1+t)^{-\frac{1}{2}-\frac{\delta}{2}}\jap{x-(t+1)v}^2\jap{v}^{\frac{\gamma+2s}{2}}\mg}_{L^\infty_xL^2_v}\\
&\qquad \times\norm{(1+t)^{-\frac{1}{2}-\frac{\delta}{2}}\jap{x-(t+1)v}^2\jap{v}^{\frac{\gamma+2s}{2}}\mh}_{L^2_xL^2_v}.
\end{align*} 
\end{lemma}
\begin{proof}
We begin by applying pre-post collision change of variables to get that 
$$I_k^+(\mf,\mg,\mh)=\int_0^{T}\int_{\R^3} \int_{\R^3}\int_{\R^3}\int_{\S^2}B_k (\color{black}{\part_v^\beta\mu})_*'\mf_*\mg\jap{x-(t+1)v'}^4 \mh'\d \sigma\d v_*\d v \d x\d t.$$

Next since $v'=\frac{v+v_*}{2}+\frac{|v-v_*|\sigma}{2}$, we have $$\jap{x-(t+1)v'}^2\lesssim \jap{x-(t+1)v}^2+\jap{x-(t+1)v_*}^2.$$
Indeed, $$x-(t+1)v'=\frac{x-(t+1)v+x-(t+1)v_*}{2}-\frac{|x-(t+1)v-(x-(t+1)v_*)|\sigma}{2}$$ and thus we get by the required inequality by applying triangle inequality. \textcolor{black}{Also note that since $\jap{\cdot}\geq 1$, $$\jap{x-(t+1)v}^2+\jap{x-(t+1)v_*}^2\lesssim \jap{x-(t+1)v}^2\jap{x-(t+1)v_*}^2.$$ }We further have,
$${\color{black}{\part^\beta_v \mu}}\lesssim \sqrt{\mu}.$$

Using these observation, we get
\begin{align*}
\int_{\R^3}&\int_{\R^3}\int_{\S^2}B_k|({\color{black}{\partial_v^\beta\mu}})_*' \mf_*\mg \jap{x-(t+1)v'}^4\mh'|\d \sigma \d v_*\d v\\
&\lesssim \int_{\R^3}\int_{\R^3}\int_{\S^2}B_k{\color{black}{\sqrt{\mu}_*'}}\mf_*[\jap{x-(t+1)v_*}^2+\jap{x-(t+1)v}^2]\mg\jap{x-(t+1)v'}^2\mh|\d \sigma \d v_*\d v\\
&\lesssim \left(\int_{\R^3}\int_{\R^3}\int_{\S^2}\frac{B_k}{|v'-v_*|^{\gamma+2s}}{\color{black}{\sqrt{\mu}_*'}} \jap{x-(t+1)v_*}^4\mf_*^2 \jap{x-(t+1)v'}^4{\mh'}^2\jap{v}^{\gamma+2s}\d \sigma\d v_*\d v\right)^{\frac{1}{2}}\\
&\quad \times\left(\int_{\R^3}\int_{\R^3}\int_{\S^2}\frac{B_k}{|v'-v_*|^{-\gamma-2s}}\jap{x-(t+1)v}^4|\mg|^2\jap{v}^{-\gamma-2s}{\color{black}{\sqrt{\mu}_*'}}\d \sigma\d v_*\d v\right)^{\frac{1}{2}}
\end{align*}

For the second factor we first apply pre-post collision change of variables and then use \lref{change_of_var_GrSt} to get
\begin{align*}
\int_{\R^3}\int_{\R^3}\int_{\S^2}&\frac{B_k}{|v'-v_*|^{-\gamma-2s}}\jap{x-(t+1)v}^4|\mg|^2\jap{v}^{-\gamma-2s}{\color{black}{\sqrt{\mu}_*'}} \d \sigma\d v_*\d v\\
&=\int_{\R^3}\int_{\R^3}\int_{E^{v'}_{v_*}}\frac{\tilde B_k}{|v'-v_*|^{-\gamma-2s}}\jap{x-(t+1)v'}^4|\mg'|^2\jap{v'}^{-\gamma-2s}{\color{black}{\sqrt{\mu}_*}} \d \pi_v\d v_*\d v'.
\end{align*}
Above we used the fact that $|v'-v_*|=|v_*'-v|.$ Now using \lref{bound_on_B_tilde} we get that 
\begin{align*}
\int_{\R^3}\int_{\R^3}\int_{E^{v'}_{v_*}}&\frac{\tilde B_k}{|v'-v_*|^{-\gamma-2s}}\jap{x-(t+1)v'}^4|\mg'|^2\jap{v'}^{-\gamma-2s}{\color{black}{\sqrt{\mu}_*}} \d \pi_v\d v_*\d v'\\
&\lesssim 2^{2sk}\int_{\R^3}\int_{\R^3}|v'-v_*|^{2\gamma+4s}{\color{black}{\sqrt{\mu}_*}}\jap{v'}^{-\gamma-2s}\jap{x-(t+1)v'}^4|\mg'|^2\d v_*\d v'\\
&\lesssim 2^{2sk}\norm{\jap{x-(t+1)v}^2 \jap{v}^{\frac{\gamma+2s}{2}}\mg}^2_{L^2_v},
\end{align*}
where we used the fact that $\int_{\R^3}{\color{black}{\sqrt{\mu}_*}}|v'-v_*|^{2\gamma+2s}\d v_*\lesssim \jap{v'}^{2\gamma+4s}.$

For the first factor we need to work a little more. If $|v'|^2\lesssim \frac{1}{2}(|v|^2+|v_*|^2)$ then from the collisional conservation laws, $\mu_*'\leq \sqrt{\mu \mu_*}$. Thus, we have $\jap{v'}^{\gamma+2s}\mu_*'\lesssim \jap{v}^{-m}\jap{v_*}^{-m}$ for any fixed positive m. For $|v'|^2\geq \frac{1}{2}(|v|^2+|v_*|^2)$, it follows from collisional geometry that $|v'|^2\approx |v|^2+|v_*|^2$ and thus $\jap{v}^{\gamma+2s}\lesssim \jap{v'}^{\gamma+2s}.$ Thus in either case we have 
\begin{align*}
\int_{\R^3}\int_{\R^3}\int_{\S^2}&\frac{B_k}{|v'-v_*|^{\gamma+2s}}{\color{black}{\sqrt{\mu}_*'}} \jap{x-(t+1)v_*}^4\mf_*^2 \jap{x-(t+1)v'}^4{\mh'}^2\jap{v}^{\gamma+2s}\d \sigma\d v_*\d v\\
&\lesssim \int_{\R^3}\int_{\R^3}\int_{\S^2}\frac{B_k}{|v'-v_*|^{\gamma+2s}} \jap{x-(t+1)v_*}^4\mf_*^2 \jap{x-(t+1)v'}^4{\mh'}^2\jap{v'}^{\gamma+2s}\d \sigma\d v_*\d v.
\end{align*}
Now using the Carleman change of variables from \lref{change_of_var_GrSt} and the bound from \lref{bound_on_B_tilde}, we get 
\begin{align*}
\int_{\R^3}\int_{\R^3}\int_{\S^2}&\frac{B_k}{|v'-v_*|^{\gamma+2s}} \jap{x-(t+1)v_*}^4\mf_*^2 \jap{x-(t+1)v'}^4{\mh'}^2\jap{v'}^{\gamma+2s}\d \sigma\d v_*\d v\\
&\lesssim 2^{2sk}\int_{\R^3}\int_{\R^3} \jap{x-(t+1)v_*}^4\mf_*^2 \jap{x-(t+1)v'}^4{\mh'}^2\jap{v'}^{\gamma+2s}\d v_*\d v'\\
&\lesssim 2^{2sk}\norm{\jap{x-(t+1)v}^2\mf}^2_{L^2_v}\norm{\jap{x-(t+1)v}^2 \jap{v}^{\frac{\gamma+2s}{2}}\mh}^2_{L^2_v}
\end{align*}
Thus we get the required result by combining the above two bound and then using Cauchy--Schwarz in space followed by multiplying and dividing by $(1+t)^{1+\delta}.$
\end{proof}

Next we prove a variant of the cancellation lemma from \cite{AlDeViWe00}.
\begin{lemma}\label{l.canc_lemma}
Let  $\mf$, $\mg$ and $\mh$ be any smooth functions and $\gamma+2s\in (0,2]$, then we have the following bounds on $$\text{Can}_k(\mf,\mg,\mh):=\left|\int_0^{T}\int_{\R^3}\int_{\R^3}\int_{\R^3}\int_{\S^2}B_k (\mf_*-\mf_*')\mg \mh\d \sigma\d v_*\d v\d x \d t\right|,$$
\begin{align*}
\sum_{k=0}^{k=\infty}|\text{Can}_k(\mf,\mg,\mh)|\lesssim\int_0^T (1+t)^{1+2\delta}&\left(\norm{\mf\jap{v}^{\gamma+2s}}\right)^s_{L^\infty_xW^{2,1}_v}\norm{\mf\jap{v}^{\gamma+2s}}^{1-s}_{L^\infty_xL^1_v}\\
&\qquad\times\norm{(1+t)^{-\frac{1+2\delta}{2}}\mg\jap{v}^{\frac{\gamma+2s}{2}}}_{L^2_xL^2_v}\norm{(1+t)^{-\frac{1+2\delta}{2}}\mh\jap{v}^{\frac{\gamma+2s}{2}}}_{L^2_xL^2_v}\d t.
\end{align*}
\end{lemma}
\begin{proof}
Fix an $1\geq R>0$. We consider two cases depending on the size of $R$:
\\
\emph{Case 1:} $R>2^{-k}$. In this case we change the integrand as follows,
\begin{align*}
B_k (\mf_*-\mf_*')\mg \mh=B_k (\mf_*-\mf_*'-(v_*-v_*')_i\part_{v_i} \mf(v_*))\mg \mh+B_k(v_*'-v_*)_i\part_{v_i} \mf(v_*)\mg \mh.
\end{align*}

Using integral form of Taylor's theorem, we know that
$$|\mf_*-\mf_*'-(v_*'-v_*)_i\part_{v_i} \mf(v_*)|\lesssim |v_*-v_*'|^2\int_0^1|\part^2_{v_jv_k} \mf(u)|\d \eta,$$
where $u=\eta v_*+(1-\eta)v_*'$. Further, by conservation of momentum we know that $|v_*-v_*'|=|v-v'|$ and thus \eref{v_minus_v'} implies that $|v_*-v_*'|=|v-v_*|\sin \frac{\theta}{2}.$

The above observations imply, 
\begin{align*}
|\int_{\R^3}\int_{\R^3}\int_{\S^2} B_k &(\mf_*-\mf_*'-(v_*-v_*')_i\part_{v_i} \mf(v_*))\mg \mh\d \sigma\d v_*\d v|\\
&\lesssim \int_0^1 \int_{\R^3}\int_{\R^3}\int_{\S^2}|v_*-v_*'|^2 |B_k \part^2_{v_jv_k} \mf(u) \mg \mh|\d \sigma\d v_*\d v\d \eta\\
&\lesssim \int_0^1 \int_{\R^3}\int_{\R^3}\int_{\S^2}|v-v_*|^2\sin^2\frac{\theta}{2} |B_k \part^2_{v_jv_k} \mf(u) \mg \mh|\d \sigma\d v_*\d v\d \eta.
\end{align*}
Next we apply the change of variables $u=\eta v_*+(1-\eta)v_*'$. Due to the collisional variables \eref{col_var}, we see that $$\frac{\d u_i}{\d v_{*_j}}=\eta \delta_{ij}+(1-\eta)\frac{\d v_{*_i}'}{\d v_{*_j}}=\left(\frac{1+\eta}{2}\right)\delta_{ij}+\frac{1-\eta}{2}k_j\sigma_i,$$
where $k=(v-v_*)/ |v-v_*|$. Thus the jacobian is $$\left|\frac{\d u_i}{\d v_{*_j}}\right|=\left(\frac{1+\eta}{2}\right)^2\left\{\left(\frac{1+\eta}{2}\right)+\frac{1-\eta}{2}\jap{k,\sigma}\right\}.$$
Since $b(\jap{k,\sigma})=0$, when $\jap{k,\sigma}\leq 0$ from \eref{sym_b} and $\eta\in [0,1]$, it follows that the Jacobian is bounded from below on support of the integral of the factor.

Further note that 
\begin{align*}
|v-u|&=\left|\frac{1+\eta}{2}(v-v_*)+\frac{1-\eta}{2}|v-v_*|\sigma\right|\\
&=|v-v_*|\left|\left(\frac{1+\eta}{2}\right)^2+\left(\frac{1-\eta}{2}\right)^2+\frac{1-\eta^2}{2}k\cdot \sigma\right|^{\frac{1}{2}}\\
&=|v-v_*|\left|\eta^2+(1-\eta^2)\cos^2 \frac{\theta}{2}\right|^{\frac{1}{2}}\geq \frac{|v-v_*|}{\sqrt 2}.
\end{align*}

In conjunction to these observations associated to this change of variables we also use that $$\int_{\S^2}B_k \sin^2 \frac{\theta}{2}\d \sigma\lesssim\int_{2^{-k-1}|v-v_*|^{-1}}^{2^{-k}|v-v_*|^{-1}} |v-v_*|^\gamma \theta^{1-2s}\d \theta\lesssim 2^{(2s-2)k}|v-v_*|^{\gamma+2s-2}$$
to get,
\begin{align*}
\int_0^1 \int_{\R^3}\int_{\R^3}\int_{\S^2}|v-v_*|^2&\sin^2\frac{\theta}{2} |B_k \part^2_{v_jv_k} \mf(u) \mg \mh|\d \sigma\d v_*\d v\d \eta\\
&\lesssim \int_0^1 \int_{\R^3}\int_{\R^3}\int_{\S^2}|v-u|^2\sin^2\frac{\theta}{2} |B_k \part^2_{v_jv_k}\mf(u) \mg \mh|\d \sigma\d u\d v\d \eta\\
&\lesssim 2^{(2s-2)k}\int_0^1 \int_{\R^3}\int_{\R^3}|v-u|^{\gamma+2s} |\part^2_{v_jv_k} \mf(u) \mg \mh|\d u\d v\d \eta.
\end{align*}

Next, applying triangle inequality to get $|v-u|^{\gamma+2s}\lesssim \jap{v}^{\gamma+2s}+\jap{u}^{\gamma+2s}$ followed by Cauchy--Schwarz, we get,
\begin{align*}
 2^{(2s-2)k}\int_0^1 \int_{\R^3}\int_{\R^3}&|v-u|^{\gamma+2s}|\part^2_{v_jv_k} \mf(u) \mg \mh|\d u\d v\d \eta\\
&\lesssim 2^{(2s-2)k} \norm{\part^2_{v_iv_k}\mf\jap{v}^{\gamma+2s}}_{L^1_v}\norm{\jap{v}^{\frac{\gamma+2s}{2}}\mg}_{L^2_v}\norm{\jap{v}^{\frac{\gamma+2s}{2}}\mh}_{L^2_v}.
\end{align*}

Now we look at the term $B_k (v_*'-v_*)_i\part_{v_i} \mf(v_*)\mg \mh$. \textcolor{black}{ Conservation of momentum implies that $v_*'+v'=v+v'$. Notice that all the terms are independent of primed variables. Hence exploiting the symmetry of $B_k$ with respect to $\sigma$ around $\frac{v-v_*}{|v-v_*|},$ we get that all the components of $v-v'$ vanish expect the one in the direction of $\frac{v-v_*}{|v-v_*|}.$ Thus we may replace $v-v'$ by $\frac{v-v_*}{|v-v_*|}\left\langle v-v',\frac{v-v_*}{|v-v_*|}\right\rangle$}.

Since $\jap{v-v',v'-v_*}=0$, the vector above reduces to $\frac{v-v_*}{|v-v_*|}\frac{|v-v'|^2}{|v-v_*|}$. Using this, we get
\begin{align*}
\int_{\S^2} B_k \frac{v-v_*}{|v-v_*|}\frac{|v-v'^2|}{|v-v_*|}\d \sigma&\lesssim\int_{2^{-k-1}|v-v_*|^{-1}}^{2^{-k}|v-v_*|^{-1}} |v-v_*|^{\gamma+1} \sin^{2}\frac{\theta}{2} \theta^{-2-2s}\sin \theta\d \theta\\
&\lesssim |v-v_*|^{\gamma+1}\int_{2^{-k-1}|v-v_*|^{-1}}^{2^{-k}|v-v_*|^{-1}}\theta^{1-2s}\d \theta\\
&\lesssim 2^{(2s-2)k}|v-v_*|^{\gamma+2s-1}.
\end{align*}
Now we consider the cases $\gamma+2s-1\geq 0$ and $\gamma+2s-1<0$ separately.\\
\emph{Case a:} $1\geq \gamma+2s-1\geq 0$. In this case we have $|v-v_*|^{\gamma+2s-1}\leq \jap{v_*}^{\gamma+2s-1}\jap{v}^{\gamma+2s-1}$. Thus we have the bound 
\begin{align*}
|\int_{\S^2}\int_{\R^3}\int_{\R^3} &B_k(v_*'-v_*)_i\part_{v_i} \mf(v_*)\mg \mh\d \sigma\d v_*\d v|\\
&\lesssim 2^{(2s-2)k}\int_{\R^3}\int_{\R^3}|v-v_*|^{\gamma+2s-1}\part_{v_i}\mf_* \mg \mh \d v_*\d v\\
&\leq  2^{(2s-2)k}\norm{\part_{v_i} f\jap{v_*}^{\gamma+2s-1}}_{L^1_v}\norm{\jap{v}^{\frac{\gamma+2s-1}{2}} \mg}_{L^2_v}\norm{\jap{v}^{\frac{\gamma+2s-1}{2}} \mh}_{L^2_v}.
\end{align*}

\emph{Case b:} $-1<\gamma+2s-1<0$. For this case we will apply \lref{H-L-S_for_infinity_lp} to get
\begin{align*}
|\int_{\S^2}\int_{\R^3}\int_{\R^3} &B_k(v_*'-v_*)_i\part_{v_i} \mf(v_*)\mg \mh\d \sigma\d v_*\d v|\\
&\lesssim 2^{(2s-2)k}\int_{\R^3}\int_{\R^3}|v-v_*|^{\gamma+2s-1}\part_{v_i}\mf_* \mg \mh \d v_*\d v\\
&\lesssim  2^{(2s-2)k}\norm{\int_{\R^3}|v-v_*|^{\gamma+2s-1}\part_{v_i}\mf_*\d v_*}_{L^\infty_v}\norm{\mg}_{L^2_v}\norm{\mh}_{L^2_v}\\
&\leq 2^{(2s-2)k}[\norm{\part_{v_i} \mf}_{L^1_v}+\norm{\part_{v_i} \mf}_{L^p_v}]\norm{\mg}_{L^2_v}\norm{\mh}_{L^2_v},
\end{align*}
where $p=\frac{3}{2+\gamma+2s}+\bar\delta$ for $\bar \delta>0$ small enough so that $p\leq 2$.\\
\emph{Case 2:} $R\in [2^{-(k+1)},2^{-k}]$ or $2^{-k+1}>R$. In this case we estimate the two parts of the difference separately. More precisely, we bound
\begin{align*}|\int_{\R^3}\int_{\R^3}\int_{\S^2}&B_k (\mf_*-\mf_*)\mg \mh\d \sigma\d v_*\d v|\\
&\leq \int_{\R^3}\int_{\R^3}\int_{\S^2}|B_k \mf_*\mg \mh|\d \sigma\d v_*\d v+\int_{\R^3}\int_{\R^3}\int_{\S^2}|B_k \mf_*'\mg \mh|\d \sigma\d v_*\d v.
\end{align*}

First we note that  $$\int_{\S^2}B_k\d \sigma\lesssim\int_{2^{-k-1}|v-v_*|^{-1}}^{2^{-k}|v-v_*|^{-1}} |v-v_*|^\gamma \theta^{-1-2s}\d \theta\lesssim 2^{2sk}|v-v_*|^{\gamma+2s}.$$
This immediately implies that we can bound the first term by an application of triangle inequality followed by Cauchy--Schwarz as follows,
$$ \int_{\R^3}\int_{\R^3}\int_{\S^2}|B_k \mf_*\mg \mh|\d \sigma\d v_*\d v\lesssim 2^{2sk}\norm{\mf\jap{v}^{\gamma+2s}}_{L^1}\norm{\mg\jap{v}^{\frac{\gamma+2s}{2}}}_{L^2}\norm{\mh\jap{v}^{\frac{\gamma+2s}{2}}}_{L^2}.$$

For the second term, we first perform a change of variables $v_*\to v_*'$ which is well-defined following the same argument as in Case 1. Further we also know by the same argument that $|v-v_*|\lesssim |v-v_*'|.$ Hence we can bound the second term in the same way as the first term. We omit the details.

Now we choose $R=\frac{\norm{\mf\jap{v}^{\gamma+2s}}_{L^1_v}^{\frac{1}{2}}}{\norm{\mf\jap{v}^{\gamma+2s}}_{W^{2,1}_v}^{\frac{1}{2}}}.$ Let $k_1$ be such that $R\in [2^{-(k_1+1)},2^{-k_1}].$ Then 
$$\sum_{k=0}^{k=\infty}|\text{Can}_k(\mf,\mg,\mh)|=\sum_{k=0}^{k=k_1}|\text{Can}_k(\mf,\mg,\mh)|+\sum_{k=k_1+1}^{k=\infty}|\text{Can}_k(\mf,\mg,\mh)|.$$

For the first term we use Case 2, to get,
\begin{align*}
\sum_{k=0}^{k=k_1}|\text{Can}_k(\mf,\mg,\mh)|&\lesssim \norm{\mf\jap{v}^{\gamma+2s}}_{L^1}\norm{\mg\jap{v}^{\frac{\gamma+2s}{2}}}_{L^2}\norm{\mh\jap{v}^{\frac{\gamma+2s}{2}}}_{L^2}\sum_{k=0}^{k=k_1} 2^{2sk}\\
&\lesssim 2^{2sk_1}\norm{\mf\jap{v}^{\gamma+2s}}_{L^1}\norm{\mg\jap{v}^{\frac{\gamma+2s}{2}}}_{L^2}\norm{\mh\jap{v}^{\frac{\gamma+2s}{2}}}_{L^2}\\
&\lesssim R^{-2s}\norm{\mf\jap{v}^{\gamma+2s}}_{L^1}\norm{\mg\jap{v}^{\frac{\gamma+2s}{2}}}_{L^2}\norm{\mh\jap{v}^{\frac{\gamma+2s}{2}}}_{L^2}\\
&\lesssim \norm{\mf\jap{v}^{\gamma+2s}}_{L^1}^{1-s}\norm{\mf\jap{v}^{\gamma+2s}}^s_{W^{2,1}_v}\norm{\mg\jap{v}^{\frac{\gamma+2s}{2}}}_{L^2}\norm{\mh\jap{v}^{\frac{\gamma+2s}{2}}}_{L^2}.
\end{align*}

For the second term we use Case 1, to get,
\begin{align*}
\sum_{k=k_1}^{k=\infty}|\text{Can}_k(\mf,\mg,\mh)|&\lesssim \left(\sum_{k=k_1}^{k=\infty}2^{(2s-2)k}\right)[ \norm{\part^2_{v_iv_j}\mf\jap{v}^{\gamma+2s}}_{L^1_v}+\norm{\part_{v_i} \mf\jap{v}^{\gamma+2s}}_{L^1_v}\\
&\hspace{10em}+\norm{\part_{v_i} \mf}_{L^p_v}]\norm{\jap{v}^{\frac{\gamma+2s}{2}}\mg}_{L^2_v}\norm{\jap{v}^{\frac{\gamma+2s}{2}}\mh}_{L^2_v}\\
&\lesssim 2^{(2s-2)k_1}[ \norm{\part^2_{v_iv_j}\mf\jap{v}^{\gamma+2s}}_{L^1_v}+\norm{\part_{v_i} \mf\jap{v}^{\gamma+2s}}_{L^1_v}\\
&\hspace{10em}+\norm{\part_{v_i} \mf}_{L^p_v}]\norm{\jap{v}^{\frac{\gamma+2s}{2}}\mg}_{L^2_v}\norm{\jap{v}^{\frac{\gamma+2s}{2}}\mh}_{L^2_v}\\
&\lesssim R^{2-2s}[ \norm{\part^2_{v_iv_j}\mf\jap{v}^{\gamma+2s}}_{L^1_v}+\norm{\part_{v_i} \mf\jap{v}^{\gamma+2s}}_{L^1_v}\\
&\hspace{10em}+\norm{\part_{v_i} \mf}_{L^p_v}]\norm{\jap{v}^{\frac{\gamma+2s}{2}}\mg}_{L^2_v}\norm{\jap{v}^{\frac{\gamma+2s}{2}}\mh}_{L^2_v}\\
&\lesssim \norm{\mf\jap{v}^{\gamma+2s}}^{(1-s)}_{L^1_v}\left[\norm{\mf\jap{v}^{\gamma+2s}}^s_{W^{2,1}_v}+\frac{\norm{\part_{v_i} \mf}_{L^p_v}}{\norm{\mf\jap{v}^{\gamma+2s}}^{(1-s)}_{W^{2,1}_v}}\right]\\
&\hspace{5em}\times\norm{\jap{v}^{\frac{\gamma+2s}{2}}\mg}_{L^2_v}\norm{\jap{v}^{\frac{\gamma+2s}{2}}\mh}_{L^2_v},
\end{align*}
where $p=\frac{3}{2+\gamma+2s}+\bar\delta$ for $\bar \delta>0$ small enough.

Next we use Gagliardo--Nirenberg inequality to get,
$$\norm{\part_{v_i} \mf}_{L^p_v}\lesssim \norm{\part^2_{v_iv_j} \mf}^\theta_{L^1_v} \norm{\mf}^{1-\theta}_{L^1_v},$$
where $\theta=\frac{4p-3}{2p}$. Note that $\frac{4p-3}{2p}=1$ for $p=\frac{3}{2}$ but since $\gamma+2s>0$, we can arrange our $p$ so that $p=\frac{3}{2}-\bar\delta$ for some $\bar \delta>0$ small enough. This implies that we can have $\theta\leq 1-\bar\delta$. In particular, since $\norm{\part^2_{v_iv_j} \mf}_{L^1_v}\lesssim \norm{\mf \jap{v}^{\gamma+2s}}_{W^{2,1}_v}$ we have that
\begin{align*}
\norm{\part_{v_i} \mf}_{L^p_v}&\lesssim\norm{\mf \jap{v}^{\gamma+2s}}_{W^{2,1}_v}^{1-\bar\delta}\norm{\mf \jap{v}^{\gamma+2s}}_{L^1_v}^{\bar\delta}.
\end{align*}

Using the above observation we get,
 \begin{align*}
\frac{\norm{\part_{v_i} \mf}_{L^p_v}}{\norm{\mf\jap{v}^{\gamma+2s}}^{(1-s)}_{W^{2,1}_v}}&\lesssim \norm{\mf\jap{v}^{\gamma+2s}}_{L^1_v}^{\bar\delta}\norm{\mf\jap{v}^{\gamma+2s}}_{W^{2,1}_v}^{s-\bar \delta}\\
&\lesssim \norm{\mf\jap{v}^{\gamma+2s}}_{W^{2,1}_v}^{s}
\end{align*}

Finally using Cauchy--Schwarz inequality in space and multiplying and dividing by $(1+t)^{1+\delta}$, we get the required result.
\end{proof}
\begin{lemma}\label{l.sym_term}
For $\gamma+2s\in(0,2]$ and $f(t,x,v)\geq 0$, we have the following bound
\begin{align*}   
\int_{0}^{T}\int_{\R^3}\int_{\R^3}\sum_{k=0}^{k=\infty}Q_k(f,G)G&\d v\d x\d t\\
&\lesssim \int_0^T (1+t)^{1+2\delta}\norm{f\jap{v}^{\gamma+2s}}^s_{L^\infty_xW^{2,1}_v}\norm{f\jap{v}^{\gamma+2s}}^{1-s}_{L^\infty_xL^1_v}\\
&\hspace{5em}\times\norm{(1+t)^{-\frac{1+2\delta}{2}}\jap{v}^{\frac{\gamma+2s}{2}}\jap{x-(t+1)v}^2 \der g}^2_{L^2_xL^2_v}\d t,
\end{align*}
where $G=\jap{x-(t+1)v}^2 \der g$.
\end{lemma}
\begin{proof}
First note that,
\begin{align*}
\int_{0}^{T}\int_{\R^3}\int_{\R^3}\sum_{k=0}^{k=\infty}Q_k(f,G)G&\d v\d x\d t\\
&=\int_{0}^{T}\int_{\R^3}\int_{\R^3}\chi_{|v-'v|\leq 1}Q(f,G)G\d v\d x\d t.
\end{align*}
By an application of pre-post collision change of variables (\lref{pre_post}) we have
\begin{align*}
\int_{0}^{T}\int_{\R^3}\int_{\R^3}&\chi_{|v'-v|\leq 1}Q(f,G)G\d v\d x\d t\\
&=-\frac{1}{2}D+\frac{1}{2}\int_0^{T_*}\int_{\R^3} \int_{\R^3}\int_{\R^3}\int_{\S^2}\chi_{|v'-v|\leq 1}B f_*(G'^2-G^2)\d \sigma\d v_*\d v\d x\d t,
\end{align*}
where $$D=\int_0^{T}\int_{\R^3} \int_{\R^3}\int_{\R^3} \int_{\S^2}\chi_{|v'-v|\leq 1}(G(v')-G(v))^2 f(v_*)B(|v-v_*|,\sigma)\d \sigma\d v_*\d v\d x\d t.$$
Since $f\geq 0$, $D_k$ is non-negative too. Hence we can drop it from our analysis.
Thus, 
\begin{align*}
\sum_{k=0}^{k=\infty} \int_{0}^{T}\int_{\R^3}\int_{\R^3}&\chi_{|v'-v|\leq 1}Q_k(f,G)G\d v\d x\d t\\&\lesssim\sum_{k=0}^{k=\infty}\left|\int_0^{T_*}\int_{\R^3} \int_{\R^3}\int_{\R^3}\int_{\S^2}B_k f_*(G'^2-G^2)\d \sigma\d v_*\d v\d x\d t\right|.
\end{align*}
Now note that by pre-post collisional change of variables, we have 
\begin{align*}
\int_0^{T_*}\int_{\R^3} \int_{\R^3}\int_{\R^3}\int_{\S^2}&B_k f_*(G'^2-G^2)\d \sigma\d v_*\d v\d x\d t\\
&=\int_{0}^{T}\int_{\R^3}\int_{\R^3}\int_{\R^3}\int_{\S^2} B_k G^2 (f_*'-f_*) \d \sigma\d v_*\d v\d x\d t\\
&:= C_k(f,G).
\end{align*}

Using \lref{canc_lemma}, we get the bound
\begin{align*}
\sum_{k=0}^{k=\infty} |C_k(f,G)|&\lesssim \int_0^T (1+t)^{1+2\delta}\norm{f\jap{v}^{\gamma+2s}}^s_{L^\infty_xW^{2,1}_v}\norm{f\jap{v}^{\gamma+2s}}^{1-s}_{L^\infty_xL^1_v}\\
&\hspace{5em}\times\norm{(1+t)^{-\frac{1+2\delta}{2}}\jap{v}^{\frac{\gamma+2s}{2}}\jap{x-(t+1)v}^2 \der g}^2_{L^2_xL^2_v}\d t.
\end{align*}
This proves the required lemma.
\end{proof}
\begin{lemma}\label{l.top_order_comm_weights}
Let $|\alpha|+|\beta|+|\omega|=n$, $\gamma+2s\in(0,2]$ and $k\geq 0$ in the singularity decompostion for $B$, we have the following bounds,
\begin{align*}
|\eref{top_order_diff_weights_mult}| \lesssim 2^{(2s-2)k}\int_0^{T}&(1+t)^{3-\delta}\norm{g\jap{x-(t+1)v}^{2\delta}}_{L^\infty_xL^1_v}\\
&\qquad\times \norm{(1+t)^{-\frac{1}{2}-\frac{\delta}{2}}\der g\jap{x-(t+1)v}^2\jap{v}}^2_{L^2_xL^2_v}\d t,
\end{align*}
and 
\begin{align*}
\sum_{k=0}^{k=\infty}|\eref{top_order_diff_weights_mult}| \lesssim\int_0^{T}&(1+t)^{3-\delta}\norm{g\jap{x-(t+1)v}^{2\delta}}_{L^\infty_xL^1_v}\\
&\qquad\times \norm{(1+t)^{-\frac{1}{2}-\frac{\delta}{2}}\der g\jap{x-(t+1)v}^2\jap{v}}^2_{L^2_xL^2_v}\d t,
\end{align*}
where $\delta$ is the same as in \eref{delta}.
\end{lemma}
\begin{proof}
We begin by applying Taylor expansion to the difference in weights
\begin{align*}
\jap{x-(t+1)v'}^2-\jap{x-(t+1)v}^2&=(v'-v)_i\part_{v_i}(\jap{x-t+1)v}^2)\left.\right |_{\eta v+(1-\eta)v'}\\
&\lesssim (1+t)\sin \frac{\theta}{2}|v-v_*|(\jap{x-(t+1)v}+\jap{x-(t+1)v'}).
\end{align*}
where $\eta\in (0,1).$ 

And similarly, $$\jap{x-(t+1)v'}-\jap{x-(t+1)v}\lesssim (1+t)\sin \frac{\theta}{2}|v-v_*|.$$

Thus, 
\begin{align*}
(\jap{x-(t+1)v'}^2&-\jap{x-(t+1)v}^2)\cdot (\jap{x-(t+1)v'}-\jap{x-(t+1)v})\\
&\lesssim (1+t)^2\sin^2 \frac{\theta}{2}|v-v_*|^2(\jap{x-(t+1)v}+\jap{x-(t+1)v'}).
\end{align*}

Next note for some $\tau\in[0,1]$,
\begin{equation}\label{e.decay_from_diff_velocity}
\begin{split}
|v-v_*|^\tau&=(1+t)^{-\tau}|x-(t+1)v-(x-(t+1)v_*)|^\tau\\
&\lesssim_\tau (1+t)^{-\tau}(\jap{x-(t+1)v}^\tau\cdot\jap{x-(t+1)v_*}^\tau).
\end{split}
\end{equation}
Indeed, if $|x-(t+1)v-(x-(t+1)v_*)|\lesssim 1$, then we trivially have the inequality. Otherwise, if $|x-(t+1)v-(x-(t+1)v_*)|\geq 1$, then we can use triangle inequality and the fact that $\jap{.}\geq 1$ to conclude that 
\begin{align*}
|x-(t+1)v-(x-(t+1)v_*)|^\tau&\lesssim (|x-(t+1)v|+|x-(t+1)v_*|)^\tau\\
&\lesssim (\jap{x-(t+1)}+\jap{x-(t+1)v_*})^\tau\\
&\lesssim (\jap{x-(t+1)v}\jap{x-(t+1)v_*})^\tau.
\end{align*}

We let $\tau=2\delta$ in \eref{decay_from_diff_velocity} such that $0<2\delta<\gamma+2s$, then 
\begin{align*}
(1+t)^2\sin^2 \frac{\theta}{2}&|v-v_*|^2(\jap{x-(t+1)v}+\jap{x-(t+1)v'})\\
&\lesssim (1+t)^{2-2\delta}|v-v_*|^{2-2\delta}\sin^2 \frac{\theta}{2}(\jap{x-(t+1)v}^2\jap{x-(t+1)v'}\jap{x-(t+1)v_*}^{2\delta}).
\end{align*}

Hence we have,
\begin{align*}
\mid B_k&(\mu g)_*(\jap{x-(t+1)v}\der g)' \der g\\
&\qquad\times[\jap{x-(t+1)v'}^2-\jap{x-(t+1)v}^2][\jap{x-(t+1)v'}-\jap{x-(t+1)v}]\mid\\
&\hspace{4em}\lesssim (1+t)^{2-2\delta}B_k|v-v_*|^{2-2\delta}\sin^2 \frac{\theta}{2}\jap{x-(t+1)v_*}^{2\delta})\mu_*g_*GG'.
\end{align*}
Also note that since $|v-v'|=|v_*-v_*'|\leq 1$, $\mu_*\approx \mu_*'$ and $\jap{v}\approx \jap{v'}.$

Integrating in $\d \sigma\d v_* \d v$ and using Cauchy-Schwarz we get 
\begin{align*}
&\int_{\R^3}\int_{\R^3}\int_{\S^2}(1+t)^{2-2\delta}B_k|v-v_*|^{2-2\delta}\sin^2 \frac{\theta}{2}\jap{x-(t+1)v_*}^{2\delta})f_*GG'\d \sigma\d v_*\d v\\
&\quad\lesssim (1+t)^{3-\delta}\left(\int_{\R^3}\int_{\R^3}\int_{\S^2}(1+t)^{-1-\delta}{\color{black}{B_k |v-v_*|^{2-2\delta}}}\right.\\
&\hspace{15em} \times\left.\sin^2 \frac{\theta}{2}\jap{x-(t+1)v_*}^{2\delta}f_*G^2\d \sigma\d v_*\d v\right)^{\frac{1}{2}}\\
&\qquad \times \left(\int_{\R^3}\int_{\R^3}\int_{\S^2}(1+t)^{-1-\delta}\sin^2 \frac{\theta}{2}{\color{black}{B_k |v-v_*|^{2-2\delta}}}\jap{x-(t+1)v_*}^{2\delta}f_*G'^2\d \sigma\d v_*\d v\right)^{\frac{1}{2}}.
\end{align*}
For the first factor we use the fact that $$\int_{\S^2}B_k \sin^2\frac{\theta}{2}\d \sigma\lesssim\int_{2^{-k-1}|v-v_*|^{-1}}^{2^{-k}|v-v_*|^{-1}} |v-v_*|^\gamma \theta^{-2s+1}\d \theta\lesssim 2^{(2s-2)k}|v-v_*|^{\gamma+2s-2}.$$
Hence we have,
\begin{align*}
\int_{\R^3}\int_{\R^3}&\int_{\S^2}(1+t)^{-1-\delta}{\color{black}{B_k |v-v_*|^{2-2\delta}}}\sin^2 \frac{\theta}{2}\jap{x-(t+1)v_*}^{2\delta}f_*G^2\d \sigma\d v_*\d v\\
&\lesssim 2^{(2s-2)k}\int_{\R^3}\int_{\R^3}\int_{\S^2}(1+t)^{-1-\delta}\jap{v}^{\gamma+2s-2\delta}\jap{v_*}^{\gamma+2s-2\delta}\jap{x-(t+1)v_*}^{2\delta}f_*G^2\d v_*\d v\\
&\lesssim 2^{(2s-2)k}\norm{\jap{x-(t+1)v}^{2\delta}g}_{L^1_v}\norm{(1+t)^{-\frac{1+\delta}{2}}\jap{x-(t+1)v}^2\jap{v}\der g}^2_{L^2_v}.
\end{align*}

For the second factor we use the regular change of variables $v'\to v$ and that $\color{black}{|v'-v_*|\approx |v-v_*|}$ \textcolor{black}{($|v'-v_*|=|v-v_*|\cos \left(\frac{\theta}{2}\right)$)} and then proceed in the same way as before.

Finally we sum over $k$ from $k=0$ to $k=\infty$ to get the desired result.
\end{proof}

Before we bound \eref{top_order_diff_f} we need an auxiliary lemma which will also be useful later. 
\begin{lemma}\label{l.diff_K_outside_ball}
Let $\mg$ be any smooth function then $K_{\mg}(v,v')$ (defined as in \lref{boltz_kernel}) satisfies the following bounds for all $r>0$
$$\int_{\R^3\backslash B_r(v')}|K_{\mg}(v,v')-K_{\mg}(v',v)|\d v\lesssim r^{1-2s}\int_{\R^3} |\part_{v_i}{\mg}|(z)|z-v'|^{\gamma+2s}\d z$$
and
$$\int_{\R^3\backslash B_r(v)}|K_{\mg}(v,v')-K_{\mg}(v',v)|\d v'\lesssim r^{1-2s}\int_{\R^3} |\part_{v_i}{\mg}|(z)|z-v|^{\gamma+2s}\d z.$$
\end{lemma}
\begin{proof}
We have $$K_{\mg}(v,v')=\frac{2^2}{|v-v'|}\left(\int_{w\perp (v-v')} {\mg}(v+w)B(|v-v_*|,\theta)|v-v_*|^{-1}\d w\right).$$

Thus we can write the difference $K_{\mg}(v,v')-K_{\mg}(v',v)$ as $$\frac{2^2}{|v-v'|}\left(\int_{w\perp (v-v')} ({\mg}(v+w)-{\mg}(v'+w))B(|v-v_*|,\theta)|v-v_*|^{-1}\d w\right).$$

Now using \lref{boltz_kernel_approx}, we get that $$|K_{\mg}(v,v')-K_{\mg}(v',v)|\approx |v-v'|^{-3-2s}\left(\int_{w\perp (v-v')} ({\mg}(v+w)-{\mg}(v'+w))|w|^{\gamma+2s+1}\d w\right).$$

As an intermediate step we define, $$\zeta_w(\nu)=\nu(v+w)+(1-\nu)(v'+w).$$
Now, we have by the integral form of Taylor's theorem 
\begin{align*}
|g(v+w)-g(v'+w)|&=\left|\int_0^1 \frac{\d \zeta_w}{\d \nu}\cdot (\nabla {\mg})(v'+w+\nu(v-v'))\d \nu\right|\\
&\leq |v-v'|\int_0^1 |\part_{v_i} {\mg}|(v'+w+\nu(v-v'))|\d \nu.
\end{align*}

Thus the following inequality holds,
\begin{align*}
\int_{\R^3\backslash B_r(v')}& |K_{\mg}(v,v')-K_{\mg}(v',v)|\d v\\
&\lesssim \int_0^1\int_{\R^3\backslash B_r(v')} \int_{w\perp (v-v')}|v-v'|^{-2-2s}|\part_{v_i} {\mg}|(v'+w+\nu(v-v'))|w|^{\gamma+2s+1}\d w\d v\d \nu.
\end{align*}

We can assume without loss of generality that $v'=0$. We also use the change of variables $\nu v=\alpha$. The jacobian for this change of variables is $\frac{1}{\nu^3}.$ With these reductions, the inequality above reads 
\begin{align*}
\int_{\R^3\backslash B_{r}}& |K_{\mg}(v,v')-K_{\mg}(v',v)|\d v\\
&\lesssim \int_0^1\int_{\R^3\backslash B_{\nu r}} \int_{w\perp \alpha} \nu^{2s-1}|\alpha|^{-2-2s}|\part_{v_i} {\mg}|(w+\alpha)|w|^{\gamma+2s+1}\d w\d \alpha\d \nu\\
&=\int_0^1\int_{r\nu}^\infty \nu^{2s-1}\rho^{-2-2s}\int_{\part B_\rho}\int_{w\perp \alpha}|\part_{v_i} {\mg}|(\alpha+w)\cdot|w|^{\gamma+2s+1}\d w\d S(\alpha) \d \rho \d \nu.
\end{align*}

Using \lref{change_of_var_2}, we can rewrite the above inequality as 
\begin{align*}
\int_{\R^3\backslash B_{r}}& |K_{\mg}(v,v')-K_{\mg}(v',v)|\d v\\
&\lesssim \int_0^1\nu^{2s-1}\int_{\nu r}^{\infty}\rho^{-2s}\int_{\R^3\backslash B_\rho} |\part_{v_i} {\mg}|(z) \frac{(|z|^2-\rho^2)^{\frac{1+\gamma+2s}{2}}}{|z|}\d z\d \rho\d \nu\\
&=\int_{0}^1 \nu^{2s-1}\int_{\R^3\backslash B_{\nu r}} \frac{|\part_{v_i} {\mg}|(z)}{|z|}\left(\int_{\nu r}^{|z|} \rho^{-2s}(|z|^2-\rho^2)^{\frac{1+\gamma+2s}{2}}\d \rho\right)\d z\d \nu\\
&\leq \int_0^1 \nu^{2s-1}\int_{\R^3\backslash B_{\nu r}}\frac{|\part_{v_i} {\mg}|(z)}{|z|}[(\nu r)^{-2s+1}|z|^{1+\gamma+2s}]\d z \d \nu\\
&\leq r^{1-2s}\int_{\R^3}|\part_{v_i} {\mg}|(z)|z|^{\gamma+2s}\d z.
\end{align*}
\end{proof}
\begin{corollary}\label{c.diff_f_cor}
Let $g$ be any smooth function then $K_{\mg}(v,v')$ defined as in \lref{boltz_kernel} satisfies the following bounds for all $r>0$
$$\int_{B_r(v')}|v-v'||K_{\mg}(v,v')-K_{\mg}(v',v)|\d v\lesssim r^{2-2s}\int_{\R^3} |\part_{v_i}{\mg}|(z)|z-v'|^{\gamma+2s}\d z$$
and
$$\int_{B_r(v)}|v-v'||K_{\mg}(v,v')-K_{\mg}(v',v)|\d v'\lesssim r^{2-2s}\int_{\R^3} |\part_{v_i}{\mg}|(z)|z-v|^{\gamma+2s}\d z.$$
\end{corollary}
\begin{proof}
\textcolor{black}{We only prove the first inequality since the second is essentially the same. First we use dyadic decomposition to $B_r(v')$ to get that
$$\int_{B_r(v')}|v-v'||K_{\mg}(v,v')-K_{\mg}(v',v)|\d v=\sum_{n=0}^{n=\infty}\int_{A_n(v')}|v-v'||K_{\mg}(v,v')-K_{\mg}(v',v)|\d v,$$
where $A_n(v')=B_{\frac{r}{2^n}}(v')\backslash B_{\frac{r}{2^{n+1}}}(v').$
Now note that $|v-v'|\lesssim \frac{r}{2^n}$ and $$\int_{A_n(v')}|v-v'||K_{\mg}(v,v')-K_{\mg}(v',v)|\d v\lesssim \frac{r}{2^n}\int_{\R^3\backslash B_{\frac{r}{2^n}}(v')}|K_{\mg}(v,v')-K_{\mg}(v',v)|\d v.$$
\lref{diff_K_outside_ball} implies that 
$$\int_{\R^3\backslash B_{\frac{r}{2^n}}(v')}|K_{\mg}(v,v')-K_{\mg}(v',v)|\d v\lesssim \left(\frac{r}{2^n}\right)^{1-2s}\int_{\R^3} |\part_{v_i}{\mg}|(z)|z-v'|^{\gamma+2s}\d z.$$
Putting together the above observations we get,
$$\int_{A_n(v')}|v-v'||K_{\mg}(v,v')-K_{\mg}(v',v)|\d v\lesssim \left(\frac{r}{2^n}\right)^{2-2s}\int_{\R^3} |\part_{v_i}{\mg}|(z)|z-v'|^{\gamma+2s}\d z.$$
Finally, summing over $n\geq 0$ and using the fact that $2-2s>0$, we get the required estimate.}
\end{proof}
\begin{lemma}\label{l.top_order_diff_f}
Let $|\alpha|+|\beta|+|\omega|=n$, $\gamma+2s\in(0,2]$ and $s\in\left(\frac{1}{2},1\right)$, we have the following bound for $\sum_{k=0}^{k=\infty}|\eref{top_order_diff_f}|$,
\begin{align*}
\sum_{k=0}^{k=\infty}|\eref{top_order_diff_f}|&\lesssim \int_0^{T}(1+t)^{2+\delta}\norm{g}^{1-s}_{L^\infty_xL^1_v}\norm{\part_{v_i}g}^{s}_{L^\infty_xL^1_v}\norm{(1+t)^{-\frac{1}{2}-\frac{\delta}{2}}\jap{v}^{\frac{\gamma+2s}{2}}\der g}^2_{L^2_xL^2_v}\d t\\
&\quad + \int_0^{T}(1+t)^{2+\delta}\norm{g}^{\frac{3}{2}-s}_{L^\infty_xL^1_v}\norm{\part_{v_i}g}^{s-\frac{1}{2}}_{L^\infty_xL^1_v}\norm{(1+t)^{-\frac{1}{2}-\frac{\delta}{2}}\jap{v}^{\frac{\gamma+2s}{2}}\der g}^2_{L^2_xL^2_v}\d t.
\end{align*}
\end{lemma}
\begin{proof}
Using the change of variables from \lref{change_of_variables_sil} and then using the definition of kernel $K_f$ from \lref{boltz_kernel}, we get that 
\begin{align*}
|\eref{top_order_diff_f}|&=\vert\int_{0}^{T}\int_{\R^3}\int_{\R^3}\int_{\R^3}\chi_k(|v'-v|)(K_f(v,v')-K_f(v',v))\jap{x-(t+1)v}\der g\\
&\qquad\times\jap{x-(t+1)v'}(\der g)'[\jap{x-(t+1)v}^2-\jap{x-(t+1)v'}^2]\d v'\d v\d x\d t\vert\\
&\lesssim \int_{0}^{T_*}\int_{\R^3}\int_{\R^3}\int_{\R^3}\chi_k(|v'-v|)|(K_f(v,v')-K_f(v',v))\jap{x-(t+1)v}\der g\\
&\qquad\times\jap{x-(t+1)v'}(\der g)'[\jap{x-(t+1)v}^2-\jap{x-(t+1)v'}^2]|\d v'\d v\d x\d t.
\end{align*}
Now using monotone convergence again as in \lref{top_order_comm_weights}, we get that
\begin{align*}
\sum_{k=0}^{k=\infty}|\eref{top_order_diff_f}|&\lesssim \int_{0}^{T}\int_{\R^3}\int_{\R^3}\int_{\R^3}|(K_f(v,v')-K_f(v',v))\jap{x-(t+1)v}\der g\jap{x-(t+1)v'}\\
&\qquad\times(\der g)'[\jap{x-(t+1)v}^2-\jap{x-(t+1)v'}^2]|\d v'\d v\d x\d t.
\end{align*}

Hence it suffices to bound the term on the right side of the above inequality. First note that as in \lref{top_order_comm_weights}, we have
\begin{equation}\label{e.obs}
[\jap{x-(t+1)v}^2-\jap{x-(t+1)v'}^2]\lesssim (1+t)|v-v'|\jap{x-(t+1)v}\jap{x-(t+1)v'}].
\end{equation}

Now we estimate the above differently when $|v-v'|>R$ and when $|v-v'|\leq R$. To keep the notation lean, we let \underline{$|K_f(v,v')-K_f(v',v)|=T_f(v',v).$} Further we drop integration over space and time.

We first treat the case $|v-v'|\leq R$. Using \eref{obs} and Cauchy--Schwarz we get, 
\begin{align*}
&\int_{\R^3}\int_{B_R(v)}|T_f(v,v') \jap{x-(t+1)v}\der g\jap{x-(t+1)v'}(\der g)'\\
 &\hspace{10em}\times[\jap{x-(t+1)v}^2-\jap{x-(t+1)v'}^2]|\d v'\d v\\
&\quad\lesssim (1+t)\int_{\R^3}\int_{B_R(v)} |v-v'|\cdot|T_f(v,v')  G'G|\d v'\d v\\
&\quad\lesssim (1+t)^{2+\delta}\norm{(1+t)^{-\frac{1}{2}-\frac{\delta}{2}}\jap{v}^{\frac{\gamma+2s}{2}}\jap{x-(t+1)v}^2 \der g}_{L^2_v}\\
&\quad \times \left(\int_{\R^3}(1+t)^{-1-\delta}\jap{v}^{-\gamma-2s}\left(\int_{B_R(v)} |v-v'|T_f(v,v')  G'\d v'\right)^2\d v\right)^{\frac{1}{2}}\\
&\quad\lesssim (1+t)^{2+\delta}\norm{(1+t)^{-\frac{1}{2}-\frac{\delta}{2}}\jap{v}^{\gamma+2s}\jap{x-(t+1)v}^2 \der g}_{L^2_v}\\
&\quad\times\left[\int_{\R^3}\left(\int_{B_R(v)}|v-v'|\jap{v}^{-\gamma-2s}T_f(v,v')\d v'\right)\left(\int_{B_R(v)}(1+t)^{-1-\delta}|v-v'|T_f(v,v')G'^2\d v'\right)\d v\right]^{\frac{1}{2}}
\end{align*}

By \cref{diff_f_cor} we get that $$\int_{B_R(v)} \jap{v}^{-2s-\gamma}|v-v'|T_f(v,v')\d v'\lesssim R^{2-2s}\norm{\part_{v_i} f\jap{v}^{\gamma+2s}}_{L^1_v},$$
and 
$$\int_{B_R(v)} \jap{v'}^{-2s-\gamma}|v-v'|T_f(v,v')\d v\lesssim R^{2-2s}\norm{\part_{v_i} f\jap{v}^{\gamma+2s}}_{L^1_v}.$$

Using this and Fubini-Tonelli we get
\begin{align}
&(1+t)^{2+\delta}\norm{(1+t)^{-\frac{1}{2}-\frac{\delta}{2}}\jap{v}^{\frac{\gamma+2s}{2}}\jap{x-(t+1)v}^2 \der g}_{L^2_v}\nonumber\\
&\quad\times\left[\int_{\R^3}\left(\int_{B_R(v)}|v-v'|\jap{v}^{-\gamma-2s}T_f(v,v')\d v'\right)\left(\int_{B_R(v)}(1+t)^{-1-\delta}|v-v'|T_f(v,v')G'^2\d v'\right)\d v\right]^{\frac{1}{2}}\nonumber\\
&\lesssim (1+t)^{2+\delta}\norm{(1+t)^{-\frac{1}{2}-\frac{\delta}{2}}\jap{v}^{\frac{\gamma+2s}{2}}\jap{x-(t+1)v}^2 \der g}_{L^2_v}\nonumber\\
&\quad\times R^{1-s}\norm{\part_{v_i} f\jap{v}^{\gamma+2s}}^{\frac{1}{2}}_{L^1_v}\left[\int_{\R^3}\int_{B_R(v)}(1+t)^{-1-\delta}|v-v'|T_f(v,v')G'^2\d v'\d v\right]^{\frac{1}{2}}\nonumber\\
&\lesssim (1+t)^{2+\delta}\norm{(1+t)^{-\frac{1}{2}-\frac{\delta}{2}}\jap{v}^{\frac{\gamma+2s}{2}}\jap{x-(t+1)v}^2 \der g}_{L^2_v}\nonumber\\
&\quad\times R^{1-s}\norm{\part_{v_i} f\jap{v}^{\gamma+2s}}^{\frac{1}{2}}_{L^1_v}\left[\int_{\R^3}(1+t)^{-1-\delta}\jap{v'}^{\gamma+2s}G'^2\left(\int_{B_R(v')}|v-v'|\jap{v'}^{-\gamma-2s}T_f(v,v')\d v\right)\d v'\right]^{\frac{1}{2}}\nonumber\\
&\lesssim (1+t)^{2+\delta}\norm{(1+t)^{-\frac{1}{2}-\frac{\delta}{2}}\jap{v}^{\frac{\gamma+2s}{2}}\jap{x-(t+1)v}^2 \der g}_{L^2_v}\nonumber\\
&\quad\times R^{2-2s}\norm{\part_{v_i} f\jap{v}^{\gamma+2s}}_{L^1_v}\left[\int_{\R^3}(1+t)^{-1-\delta}\jap{v'}^{\gamma+2s}G'^2\d v'\right]^{\frac{1}{2}}\nonumber\\
&\lesssim (1+t)^{2+\delta}R^{2-2s}\norm{\part_{v_i} f\jap{v}^{\gamma+2s}}_{L^1_v}\norm{(1+t)^{-\frac{1}{2}-\frac{\delta}{2}}\jap{v}^{\frac{\gamma+2s}{2}}\jap{x-(t+1)v}^2 \der g}^2_{L^2_v}.\label{e.diff_f_1_sub_R}
\end{align}

For $|v-v'|\geq R$ we have
\begin{align*}
&\int_{\R^3}\int_{\R^3\backslash B_R(v)}(K_f(v,v')-K_f(v',v))\jap{x-(t+1)v}\der g\jap{x-(t+1)v'}(\der g)' \\
&\hspace{10em}\times [\jap{x-(t+1)v}^2-\jap{x-(t+1)v'}^2]\d v'\d v\\
&\lesssim (1+t)\left|\int_{\R^3}\int_{\R^3\backslash B_R(v)}|v-v'|K_f(v,v')GG'\d v'\d v\right|+\left|\int_{\R^3}\int_{\R^3\backslash B_R(v)}|v-v'|K_f(v',v)GG'\d v'\d v\right|\\
&\lesssim (1+t)\left|\int_{\R^3}\int_{\R^3\backslash B_R(v)}|v-v'|K_f(v,v')GG'\d v'\d v\right|.
\end{align*}
We used pre-post collision change of variables for the last inequality.

Since $s\geq \frac{1}{2}$, we get by Cauchy--Schwarz, \lref{K_bound} and Fubini-Tonelli,
\begin{align}
(1+t)&\left|\int_{\R^3}\int_{\R^3\backslash B_R(v)}|v-v'|K_f(v,v')GG'\d v'\d v\right|\nonumber\\
&\lesssim (1+t)^{2+\delta}\norm{(1+t)^{-\frac{1}{2}-\frac{\delta}{2}}\jap{v}^{\frac{\gamma+2s}{2}}\der g}_{L^2_v}\nonumber\\
&\quad \times\left[\int_{\R^3}\left(\int_{\R^3\backslash B_R(v)}|v-v'|\jap{v}^{-\gamma-2s}K_f(v,v')\d v'\right)\right.\\
&\qquad\times \left.\left(\int_{\R^3\backslash B_R(v)}(1+t)^{-1-\delta}|v-v'|K_f(v,v')G'^2\d v'\right)\d v\right]^{\frac{1}{2}}\nonumber\\
&\lesssim (1+t)^{2+\delta}R^{\frac{1}{2}-s}\norm{(1+t)^{-\frac{1}{2}-\frac{\delta}{2}}\jap{v}^{\frac{\gamma+2s}{2}}\der g}_{L^2_v}\norm{f\jap{v}^{\gamma+2s}}^{\frac{1}{2}}_{L^1_v}\nonumber\\
&\quad \times\left[\int_{\R^3}\int_{\R^3\backslash B_R(v)}(1+t)^{-1-\delta}|v-v'|K_f(v,v')G'^2\d v'\d v\right]^{\frac{1}{2}}\nonumber\\
&\lesssim (1+t)^{2+\delta}R^{1-2s}\norm{(1+t)^{-\frac{1}{2}-\frac{\delta}{2}}\jap{v}^{\frac{\gamma+2s}{2}}\der g}^2_{L^2_v}\norm{f\jap{v}^{\gamma+2s}}_{L^1_v}.\label{e.diff_f_2_sub_R}
\end{align}

Choosing $R=\frac{\norm{f\jap{v}^{\gamma+2s}}^{\frac{1}{2}}_{L^1_v}}{\norm{\part_{v_i}f\jap{v}^{\gamma+2s}}^{\frac{1}{2}}_{L^1_v}}$, we get that
$$\eref{diff_f_1_sub_R}\lesssim (1+t)^{2+\delta}\norm{f\jap{v}^{\gamma+2s}}^{1-s}_{L^1_v}\norm{\part_{v_i}f\jap{v}^{\gamma+2s}}^{s}_{L^1_v}\norm{(1+t)^{-\frac{1}{2}-\frac{\delta}{2}}\jap{v}^{\frac{\gamma+2s}{2}}\der g}^2_{L^2_v},$$
and
$$\eref{diff_f_2_sub_R}\lesssim (1+t)^{2+\delta}\norm{f\jap{v}^{\gamma+2s}}^{\frac{3}{2}-s}_{L^1_v}\norm{\part_{v_i}f\jap{v}^{\gamma+2s}}^{s-\frac{1}{2}}_{L^1_v}\norm{(1+t)^{-\frac{1}{2}-\frac{\delta}{2}}\jap{v}^{\frac{\gamma+2s}{2}}\der g}^2_{L^2_v}.$$

We thus get the required result after using \lref{exp_bound} and Cauchy--Schwarz in space.
\end{proof}
\begin{lemma}\label{l.top_order_s_less_than_half}
For $s<\frac{1}{2}$, $\gamma+2s\in(0,2]$ and $k\geq 0$ in the singularity decompostion for $B$, we have the following bounds
\begin{align*}
|\eref{top_order_s_less_than_half}|\lesssim2^{(2s-1)k} \int_0^{T} &(1+t)^{2+\delta}\norm{g\jap{x-(t+1)v}}_{L^\infty_xL^1_v}\\
&\qquad\times\norm{(1+t)^{-\frac{1+\delta}{2}}\der g\jap{v}^{\frac{\gamma+2s}{2}}\jap{x-(t+1)v}^2}^2_{L^2_xL^2_v}\d t,
\end{align*}
and
$$\sum_{k=0}^{k=\infty}|\eref{top_order_s_less_than_half}| \lesssim \int_0^{T} (1+t)^{2+\delta}\norm{g\jap{x-(t+1)v}}_{L^\infty_xL^1_v}\norm{(1+t)^{-\frac{1+\delta}{2}}\der g\jap{v}^{\frac{\gamma+2s}{2}}\jap{x-(t+1)v}^2}^2_{L^2_xL^2_v}\d t,$$
where $\delta$ is the same as in \eref{delta}.
\end{lemma}
\begin{proof}
As in \lref{top_order_comm_weights} we have that $$|\jap{x-(t+1)v'}^2-\jap{x-(t+1)v}^2|\lesssim (1+t)|v-v_*|\sin \frac{\theta}{2}\jap{x-(t+1)v}\jap{x-(t+1)v_*}.$$

Hence we have that
\begin{align*}
|B_kf_*G'\der g &[\jap{x-(t+1)v'}^2-\jap{x-(t+1)v}^2]|\\
&\qquad \lesssim (1+t)\sin \frac{\theta}{2}B_k|v-v_*| \jap{x-(t+1)v_*}f_*GG'.
\end{align*}

Now by Cauchy--Schwarz and noting that $s<\frac{1}{2}$ we get
\begin{align*}
(1+t)&\int_{\R^3}\int_{\R^3}\int_{\S^2} |\sin \frac{\theta}{2}B_k|v-v_*| \jap{x-(t+1)v_*}f_*GG'|\d \sigma\d v_*\d v\\
&\lesssim (1+t)^{2+\delta}\left(\int_{\R^3}\int_{\R^3}\int_{\S^2} |(1+t)^{-1-\delta}\sin \frac{\theta}{2}{\color{black}{B_k}}|v-v_*|\right.\\
&\hspace{10em}\times \left.\jap{x-(t+1)v_*}f_*G^2| \d \sigma\d v_*\d v\right)^\frac{1}{2}\\
&\quad \times \left(\int_{\R^3}\int_{\R^3}\int_{\S^2} |(1+t)^{-1-\delta}\sin \frac{\theta}{2}{\color{black}{B_k}}|v-v_*| \jap{x-(t+1)v_*}\right.\\
&\hspace{10em}\times\left.f_*G'^2| \d \sigma\d v_*\d v\right)^\frac{1}{2}.
\end{align*}

Next we use $$\int_{\S^2}B_k \sin \frac{\theta}{2}\d \sigma\lesssim\int_{2^{-k-1}|v-v_*|^{-1}}^{2^{-k}|v-v_*|^{-1}} |v-v_*|^\gamma \theta^{-2s}\d \theta\lesssim 2^{(2s-1)k}|v-v_*|^{\gamma+2s-1}.$$

Now proceeding in the same way as \lref{top_order_comm_weights} we get the desired result.
\end{proof}
\begin{lemma}\label{l.top_order_Gaussian_1}
Let $\gamma+2s\in(0,2]$ and $k\geq 0$ in the singularity decompostion for $B$, then we have the following bounds
\begin{align*}
|\eref{top_order_Gaussian_1}|\lesssim 2^{(2s-2)k}&\int_0^{T}(1+t)^{1+\delta}\norm{g\jap{x-(t+1)v}^2}_{L^\infty_xL^2_v}\\
&\qquad\times\norm{(1+t)^{-\frac{1+\delta}{2}}\jap{x-(t+1)v}^2\jap{v}\der g}_{L^2_xL^2_v}^2\d t,
\end{align*}
and 
\begin{align*}
\sum_{k=0}^{k=\infty}|\eref{top_order_Gaussian_1}|\lesssim& \int_0^{T}(1+t)^{1+\delta}\norm{g\jap{x-(t+1)v}^2}_{L^\infty_xL^2_v}\\
&\qquad\times\norm{(1+t)^{-\frac{1+\delta}{2}}\jap{x-(t+1)v}^2\jap{v}\der g}_{L^2_xL^2_v}^2\d t.
\end{align*}
\end{lemma}
\begin{proof}
We first apply pre-post collision change of variables to get 
\begin{align*}
\int_{\R^3}\int_{\R^3}\int_{\S^2}& B_k(\mu_*-\mu_*')g'_*(\der g)' G[\jap{x-(t+1)v}^2-\jap{x-(t+1)v'}^2]\\
&=-\int_{\R^3}\int_{\R^3}\int_{\S^2} B_k(\mu'_*-\mu_*)g_*(\der g) G'[\jap{x-(t+1)v}^2-\jap{x-(t+1)v'}^2].
\end{align*}

Now as in \lref{top_order_comm_weights} we have
$$|\jap{x-(t+1)v}^2-\jap{x-(t+1)v}^2|\lesssim (1+t)\sin \frac{\theta}{2}|v-v_*|(\jap{x-(t+1)v}+\jap{x-(t+1)v'}).$$
Next using $$\jap{x-(t+1)v'}\lesssim \jap{x-(t+1)v}+\jap{x-(t+1)v_*},$$ and that $$|v-v_*|\lesssim (1+t)^{-1}(\jap{x-(t+1)v}+\jap{x-(t+1)v_*}),$$
we get that 
\begin{align*}
|\jap{x-(t+1)v'}^2-\jap{x-(t+1)v}^2|&\lesssim \sin \frac{\theta}{2}[\jap{x-(t+1)v}^2+\jap{x-(t+1)v_*}^2]\\
&\lesssim \sin \frac{\theta}{2}\jap{x-(t+1)v}^2\jap{x-(t+1)v_*}^2.
\end{align*}

\textcolor{black}{By Taylor's theorem we have that $$|\mu_*'-\mu_*|=|v_*-v_*'|\part_{v_i}\mu(\eta v_*+(1-\eta)v_*'),$$
for $\eta\in (0,1)$. Since $|v_*'-v_*|=|v'-v|\leq 1$ as $k\geq 0$ (in $B_k$) we have that $\jap{v_*}\lesssim \jap{\eta v_*+(1-\eta)v_*'}\lesssim \jap{v'_*}$. Thus for $\tau>0$ small enough we have that
$$|v_*-v_*'|\part_{v_i}\mu(\eta v_*+(1-\eta)v_*')\lesssim |v-v_*|\sin \frac{\theta}{2}(\mu_*\mu_*')^\tau.$$}

\textcolor{black}{From above and using Cauchy--Schwarz we have that 
\begin{align*}&\int_{\R^3}\int_{\R^3}\int_{\S^2} |B_k(\mu_*-\mu_*')g_*(\der g) G'[\jap{x-(t+1)v}^2-\jap{x-(t+1)v'}^2]|\\
&\quad\lesssim \int_{\R^3}\int_{\R^3}\int_{\S^2} |v-v_*|^{1+\gamma}\sin^2 \frac{\theta}{2} B_k\jap{x-(t+1)v_*}^2g_* G G'(\mu_*\mu_*')^\tau\d \sigma\d v_*\d v\\
&\quad\lesssim \left(\int_{\R^3}\int_{\R^3}\int_{\S^2} |v-v_*|^{1+\gamma}\sin^2 \frac{\theta}{2}B_k \jap{x-(t+1)v_*}^2g_* G^2 (\mu_*\mu_*')^\tau\d \sigma\d v_*\d v\right)^\frac{1}{2}\\
&\qquad \times \left(\int_{\R^3}\int_{\R^3}\int_{\S^2} |v-v_*|^{1+\gamma}\sin^2 \frac{\theta}{2}B_k \jap{x-(t+1)v_*}^2g_* G'^2 (\mu_*\mu_*')^\tau\d \sigma\d v_*\d v\right)^\frac{1}{2}\\
&\quad\lesssim 2^{(2s-2)k}\left(\int_{\R^3}\int_{\R^3}|v-v_*|^{\gamma+2s-1}\jap{x-(t+1)v_*}^2g_* G^2 (\mu_*\mu_*')^\tau\d v_*\d v\right)^\frac{1}{2}\\
&\qquad \times \left(\int_{\R^3}\int_{\R^3}|v'-v_*|^{\gamma+2s-1} \jap{x-(t+1)v_*}^2g_* G'^2 (\mu_*\mu_*')^\tau\d v_*\d v'\right)^\frac{1}{2},
\end{align*}
where to get the last inequality we used the regular change of variables $v\to v'$ and the fact that $|v-v_*|\approx |v-v_*|$ (see \lref{top_order_comm_weights}). In addition, we also used 
$$\int_{\S^2}B_k \sin^2\frac{\theta}{2}\d \sigma\lesssim\int_{2^{-k-1}|v-v_*|^{-1}}^{2^{-k}|v-v_*|^{-1}} |v-v_*|^\gamma \theta^{-2s+1}\d \theta\lesssim 2^{(2s-2)k}|v-v_*|^{\gamma+2s-2}.$$}
Next we consider the cases $\gamma+2s-1\geq 0$ and $\gamma+2s-1<0$ separately.\\
\emph{Case 1:} $1\geq \gamma+2s-1\geq 0$. In this case we have $|v-v_*|^{\gamma+2s-1}\leq \jap{v_*}\jap{v}$. Thus we have the bound 
\begin{align*}
\int_{\R^3}\int_{\R^3}&|v-v_*|^{\gamma+2s-1}\jap{x-(t+1)v_*}^2g_* G^2 (\mu_*\mu_*')^\tau\d \sigma\d v_*\d v\\
&\leq \norm{g\mu^\tau\jap{v}\jap{x-(t+1)v}^2}_{L^1_v}\norm{\jap{x-(t+1)v}^2\jap{v}\der g}_{L^2_v}\\
&\leq \norm{g\jap{x-(t+1)v}^2}_{L^2_v}\norm{\jap{x-(t+1)v}^2\jap{v}\der g}_{L^2_v}.
\end{align*}
We get a similar bound for the other factor. 

Thus applying Cauchy--Schwarz in space followed by multiplying and dividing by $(1+t)^{1+\delta}$, we get the desired result.\\
\emph{Case 2:} $-1<\gamma+2s-1<0$. For this we will apply \lref{H-L-S_for_infinity_lp} to get
\begin{align*}
\int_{\R^3}\int_{\R^3}&|v-v_*|^{\gamma+2s-1}\jap{x-(t+1)v_*}^2g_* G^2 (\mu_*\mu_*')^\tau\d \sigma\d v_*\d v\\
&\leq [\norm{g\mu^\tau\jap{x-(t+1)v}^2}_{L^1_v}+\norm{ g\mu^\tau\jap{x-(t+1)v}^2}_{L^2_v}]\norm{\jap{x-(t+1)v}^2\der g}_{L^2_v}\\
&\leq \norm{ g\jap{x-(t+1)v}^2}_{L^2_v}\norm{\jap{x-(t+1)v}^2\jap{v}\der g}_{L^2_v}.
\end{align*}

Now we proceed in the same way as in the last case.
\end{proof}
\begin{lemma}\label{l.top_order_Gaussian_2}
Let $\gamma+2s\in(0,2]$ and $k\geq 0$ in the singularity decompostion for $B$, then we have the following bounds
\begin{align*}
|\eref{top_order_Gaussian_2}|\lesssim 2^{(2s-2)k}\int_0^{T}(1+t)^{1+\delta}\norm{\part_{v_i} g}_{L^\infty_xL^1_v}\norm{(1+t)^{-\frac{1+\delta}{2}}\jap{x-(t+1)v}^2\jap{v}\der g}_{L^2_xL^2_v}^2 \d t,
\end{align*}
and 
\begin{align*}
\sum_{k=0}^{k=\infty}|\eref{top_order_Gaussian_2}|\lesssim \int_0^{T}(1+t)^{1+\delta}\norm{\part_{v_i} g}_{L^\infty_xL^1_v}\norm{(1+t)^{-\frac{1+\delta}{2}}\jap{x-(t+1)v}^2\jap{v}\der g}_{L^2_xL^2_v}^2 \d t,
\end{align*}
\end{lemma}
\begin{proof}
As in \lref{top_order_Gaussian_1}, we have that 
$$|\mu_*'-\mu_*|=|v_*-v_*'|\part_{v_i}\mu(\eta v_*+(1-\eta)v_*')\lesssim |v-v_*|\sin \frac{\theta}{2}(\mu_*\mu_*')^\tau.$$

Thus, 
\begin{align*}
\int_{\R^3}\int_{\R^3}\int_{\S^2}&|B_k(\mu_*'-\mu_*)(g_*'-g_*)G' G|\d \sigma\d v_*\d v\\\
&\lesssim \int_{\R^3}\int_{\R^3}\int_{\S^2} B_k(\mu_*\mu_*')^\tau|v-v'|\cdot|(g_*'-g_*)|\cdot|G' |\cdot |G|\d \sigma\d v_*\d v.
\end{align*}

Now using integral form of Taylor's theorem we get that $$|g_*'-g_*|\lesssim |v_*-v_*'|\int_0^1 |\part_{v_i} g|(\eta v_*+(1-\eta)v_*')\d \eta.$$

Let $u=\eta v_*+(1-\eta)v_*'$. From above observations we get, 
\begin{align*}
\int_{\R^3}\int_{\R^3}\int_{\S^2}&|B_k(\mu_*'-\mu_*)(g_*'-g_*)G' G|\d \sigma\d v_*\d v\\
&\lesssim \int_0^1 \int_{\R^3}\int_{\R^3}\int_{\S^2} B_k |v-v_*|^2\sin^2 \frac{\theta}{2}(\mu_*\mu_*')^\tau|\part_{v_i} g|(u)|G'|\cdot|G|\d \sigma\d v_*\d v\d \eta\\
&\lesssim \left(\int_0^1 \int_{\R^3}\int_{\R^3}\int_{\S^2} B_k |v-v_*|^2\sin^2 \frac{\theta}{2}(\mu_*\mu_*')^\tau|\part_{v_i} g|(u)G^2 \d \sigma\d v_*\d v\d \eta\right)^{\frac{1}{2}}\\
&\quad \times \left(\int_0^1 \int_{\R^3}\int_{\R^3}\int_{\S^2} B_k |v-v_*|^2\sin^2 \frac{\theta}{2}(\mu_*\mu_*')^\tau|\part_{v_i} g|(u)G'^2 \d \sigma\d v_*\d v\d \eta\right)^{\frac{1}{2}}.
\end{align*}

For the first factor, we need to apply the change of variables $u=\eta v_*+(1-\eta)v_*'$. Due to the collisional variables \eref{col_var}, we see that $$\frac{\d u_i}{\d v_{*_j}}=\eta \delta_{ij}+(1-\eta)\frac{\d v_{*_i}'}{\d v_{*_j}}=\left(\frac{1+\eta}{2}\right)\delta_{ij}+\frac{1-\eta}{2}k_j\sigma_i,$$
where $k=(v-v_*)/ |v-v_*|$. Thus the jacobian is $$\left|\frac{\d u_i}{\d v_{*_j}}\right|=\left(\frac{1+\eta}{2}\right)^2\left\{\left(\frac{1+\eta}{2}\right)+\frac{1-\eta}{2}\jap{k,\sigma}\right\}.$$
Since $b(\jap{k,\sigma})=0$, when $\jap{k,\sigma}\leq 0$ from \eref{sym_b} and $\eta\in [0,1]$, it follows that the Jacobian is bounded from below on support of the integral of the factor.\\
Also note that 
\begin{align*}
|v-u|&=\left|\frac{1+\eta}{2}(v-v_*)+\frac{1-\eta}{2}|v-v_*|\sigma\right|\\
&=|v-v_*|\left|\left(\frac{1+\eta}{2}\right)^2+\left(\frac{1-\eta}{2}\right)^2+\frac{1-\eta^2}{2}k\cdot \sigma\right|^{\frac{1}{2}}\\
&=|v-v_*|\left|\eta^2+(1-\eta^2)\cos^2 \frac{\theta}{2}\right|^{\frac{1}{2}}\geq \frac{|v-v_*|}{\sqrt 2},
\end{align*}
and since $|v_*-v_*'|\leq 1$, $$(\mu_*\mu_*')^{\tau}\leq \mu(u)^{\tau'},$$ for small enough $\tau'$.

Applying this change of variables to the first factor and  using the above observations we get,
\begin{align*}
\int_0^1 \int_{\R^3}\int_{\R^3}\int_{\S^2}& B_k |v-v_*|^2\sin^2 \frac{\theta}{2}(\mu_*\mu_*')^\tau|\part_{v_i} g|(u)G^2 \d \sigma\d v_*\d v\d \eta\\
&\lesssim \int_0^1 \int_{\R^3}\int_{\R^3}\int_{\S^2} B_k |v-u|^2\sin^2 \frac{\theta}{2}(\mu(u))^{\tau'}|\part_{v_i} g|(u)G^2 \d \sigma\d u\d v\d \eta.
\end{align*}
Now we apply Cauchy--Schwarz in $u$ and absorb the $\jap{u}^{\gamma+2s}$ into $\mu^{\tau'}(u)$.

For the second factor we first apply pre-post collisional change of variables to get that
\begin{align*}
\int_0^1 \int_{\R^3}\int_{\R^3}\int_{\S^2}& B_k |v-v_*|^2\sin^2 \frac{\theta}{2}(\mu_*\mu_*')^\tau|\part_{v_i} g|(u)G'^2 \d \sigma\d v_*\d v\d \eta\\
&=\int_0^1 \int_{\R^3}\int_{\R^3}\int_{\S^2} B_k |v-v_*|^2\sin^2 \frac{\theta}{2}(\mu_*\mu_*')^\tau|\part_{v_i} g|(\tilde u)G^2 \d \sigma\d v_*\d v\d \eta,
\end{align*}
where $\tilde u=\eta v_*'+(1-\eta)v_*.$ In a similar way as for the change of variables for $u$, the jacobian for the change of variables for $\tilde u$ is also bounded from below. Hence we get the required inequality by proceeding in the same way as for the first half.
\end{proof}
\begin{lemma}\label{l.top_order_Gaussian_s_less_than_half}
Let $\gamma+2s\in(0,2]$, $s<\frac{1}{2}$ and $k\geq 0$ in the singularity decompostion for $B$, then we have the following bounds
\begin{align*}
|\eref{top_order_Gaussian_s_less_than_half}|\lesssim 2^{(2s-1)k}\int_0^{T}&(1+t)^{1+\delta}\norm{g\jap{x-(t+1)v}^2}_{L^\infty_xL^2_v}\\
&\qquad\times\norm{(1+t)^{-\frac{1+\delta}{2}}\jap{x-(t+1)v}^2\jap{v}\der g}_{L^2_xL^2_v}^2 \d t,
\end{align*}
and 
\begin{align*}
\sum_{k=0}^{k=\infty}|\eref{top_order_Gaussian_s_less_than_half}|\lesssim \int_0^{T}&(1+t)^{1+\delta}\norm{g\jap{x-(t+1)v}^2}_{L^\infty_xL^2_v}\\
&\qquad\times\norm{(1+t)^{-\frac{1+\delta}{2}}\jap{x-(t+1)v}^2\jap{v}\der g}_{L^2_xL^2_v}^2 \d t.
\end{align*}
\end{lemma}
\begin{proof}
We first apply pre-post collision change of variables to get
\begin{align*}
\int_{\R^3}\int_{\R^3}\int_{\S^2}& B_k(\mu_*'-\mu_*)g_*'(\der g)'\jap{x-(t+1)v}^2 G\d \sigma\d v_*\d v\\
&=-\int_{\R^3}\int_{\R^3}\int_{\S^2} B_k(\mu_*'-\mu_*)g_*\der g\jap{x-(t+1)v'}^2 G'\d \sigma\d v_*\d v.
\end{align*}

Next, by the collisional variables \eref{col_var}, we have 
\begin{align*}
\jap{x-(t+1)v'}^2&\lesssim \jap{x-(t+1)v}^2+\jap{x-(t+1)v_*}^2\\
&\lesssim \jap{x-(t+1)v}^2\jap{x-(t+1)v_*}^2.
\end{align*}
Moreover, as in \lref{top_order_Gaussian_1} and \lref{top_order_Gaussian_2}, we get
$$|\mu_*'-\mu_*|=|v_*-v_*'|\part_{v_i}\mu(\eta v_*+(1-\eta)v_*')\lesssim |v-v_*|\sin \frac{\theta}{2}(\mu_*\mu_*')^\tau.$$

\textcolor{black}{From above observations we have the bound
\begin{align*}
\int_{\R^3}\int_{\R^3}\int_{\S^2} &|B_k(\mu_*'-\mu_*)g_*\der g\jap{x-(t+1)v'}^2 G'|\d \sigma\d v_*\d v\\
&\lesssim \int_{\R^3}\int_{\R^3}\int_{\S^2} |v-v_*|\sin \frac{\theta}{2} B_k(\mu_*'\mu_*)^\tau \jap{x-(t+1)v_*}^2|g_*|\cdot |G|\cdot |G'|\d \sigma\d v_*\d v\\
&\lesssim \left(\int_{\R^3}\int_{\R^3}\int_{\S^2} |v-v_*|\sin \frac{\theta}{2} B_k(\mu_*'\mu_*)^\tau \jap{x-(t+1)v_*}^2|g_*|\cdot |G|^2\d \sigma\d v_*\d v\right)^{\frac{1}{2}}\\
&\qquad\times \left(\int_{\R^3}\int_{\R^3}\int_{\S^2} |v-v_*|\sin \frac{\theta}{2} B_k(\mu_*'\mu_*)^\tau \jap{x-(t+1)v_*}^2|g_*|\cdot |G'|^2\d \sigma\d v_*\d v\right)^{\frac{1}{2}}
\end{align*}
}
Now by our decomposition of singularity, we have
$$|v-v_*|\int_{\S^2} \sin \frac{\theta}{2}B_k\d \sigma\lesssim |v-v_*|^{\gamma+1}\int_{2^{-k-1}|v-v_*|^{-1}}^{2^{-k}|v-v_*|^{-1}} \theta^{-2s}\d \theta\lesssim 2^{(2s-1)k}|v-v_*|^{\gamma+2s}.$$
Now note that we have that $\mu_*\jap{v_*}^m\lesssim \sqrt{\mu_*}$ for some fixed integer $m$. With this observation and Cauchy--Schwarz, we have 
\begin{align*}
\int_{\R^3}\int_{\R^3}\int_{\S^2}& |v-v_*|\sin \frac{\theta}{2} B_k(\mu_*'\mu_*)^\tau \jap{x-(t+1)v_*}^2|g_*|\cdot |G|^2\d \sigma\d v_*\d v\\
&\lesssim 2^{(2s-1)k}\int_{\R^3}\int_{\R^3}(\mu_*\mu_*')^{\frac{\tau}{2}}\jap{v}^2 \jap{x-(t+1)v_*}^2|g_*|\cdot |G|^2\d v_*\d v\\
&\lesssim 2^{(2s-1)k}\norm{\jap{x-(t+1)v}^2 g}_{L^2_v}^2\norm{\jap{x-(t+1)v}^2\jap{v}\der g}_{L^2_v}^2
\end{align*}

To bound $$ \left(\int_{\R^3}\int_{\R^3}\int_{\S^2} |v-v_*|\sin \frac{\theta}{2} B_k(\mu_*'\mu_*)^\tau \jap{x-(t+1)v_*}^2|g_*|\cdot |G'|^2\d \sigma\d v_*\d v\right)^{\frac{1}{2}}$$
we use the regular change of variables $v\to v'$ and then proceed as in \lref{top_order_comm_weights}. Finally, we apply Cauchy--Schwarz in space and multiply and divide by $(1+t)^{1+\delta}$.
\end{proof}
\begin{lemma}\label{l.top_est}
For $\gamma+2s\in (0,2]$ and $|\alpha|+|\beta|+|\omega|=10$, we have the following bound,
\begin{align*}
\int_0^{T}\int_{\R^3}\int_{\R^3}\Gamma(g,\der g)\jap{x-(t+1)v}^4&\der g\d v\d x\d t\\
&\lesssim \mathbb{I}+\mathbb{II}+\mathbb{III}+\mathbb{IV}+\mathbb{V}+\mathbb{VI},
\end{align*}
where $\mathbb{I}-\mathbb{VI}$ are the same as defined in \lref{main_boltz_estimates}.
\end{lemma}
\begin{proof}
We use the singularity decompostion from \eref{sing_decom_top_order}. For $k<0$, we use the first estimate of \lref{trivial_est_1} and \lref{trivial_est_2} with $g$ as $\mf$, $\der g$ as $\mg$ and $\mh$ to get that 
\begin{align*}
\int_{0}^{T}\int_{\R^3}\int_{\R^3}\int_{\R^3}\int_{\S^2} &|B_k \mu_*(g_*'(\der g)'-g_*\der g)\\
&\qquad\times\jap{x-(t+1)v}^4\der g|\d \sigma\d v_*\d v\d x\d t\\
&\lesssim 2^{2sk}\int_0^{T}(1+t)^{1+\delta}\norm{\jap{x-(t+1)v}^2g}_{L^\infty_xL^2_v}\\
&\qquad \times \norm{(1+t)^{-\frac{1}{2}-\frac{\delta}{2}}\jap{x-(t+1)v}^2\jap{v}\der g}^2_{L^2_xL^2_v}\d t.
\end{align*}
Now summing over $k<0$, we get that 
\begin{align*}
\sum_{k=-\infty}^{k=0}\int_{0}^{T}\int_{\R^3}&\int_{\R^3}\int_{\R^3}\int_{\S^2} |B_k \mu_*(g_*'(\der g)'-g_*\der g)\\
&\hspace{6em}\times\jap{x-(t+1)v}^4\der g|\d \sigma\d v_*\d v\d x\d t\\
&\lesssim \int_0^{T}(1+t)^{1+\delta}\norm{\jap{x-(t+1)v}^2g}_{L^\infty_xL^2_v}\norm{(1+t)^{-\frac{1}{2}-\frac{\delta}{2}}\jap{x-(t+1)v}^2\jap{v}\der g}^2_{L^2_xL^2_v}\d t\\
&\lesssim \mathbb{I}.
\end{align*}

Next we treat the case $k\geq 0$. We bound \eref{top_order_symm} using \lref{sym_term} to get, 
\begin{align*}  
\int_{0}^{T}\int_{\R^3}\int_{R^3}\sum_{k=0}^{k=\infty}Q_k(f,G)G&\lesssim\mathbb{II}.
\end{align*}

Next we split into two cases depending on whether $s<\frac{1}{2}$ or $s\in \left[\frac{1}{2},1\right).$\\
\emph{Case 1:} $s<\frac{1}{2}$. For this case, we focus on the terms \eref{top_order_s_less_than_half} and \eref{top_order_Gaussian_s_less_than_half}.

Now using \lref{top_order_s_less_than_half}, we get
\begin{align*}   
\sum_{k=0}^{k=\infty}|\eref{top_order_s_less_than_half}|& \lesssim
      \int_0^{T} (1+t)^{2+\delta}\norm{g\jap{x-(t+1)v}}_{L^\infty_xL^1_v}\\
&\qquad\times\norm{(1+t)^{-\frac{1+\delta}{2}}\der g\jap{x-(t+1)v}^2\jap{v}}^2_{L^2_xL^2_v}\d t\\
&\lesssim \mathbb{III}.
\end{align*}

Finally, we use \lref{top_order_Gaussian_s_less_than_half} to get the bound,
\begin{align*}
\sum_{k=0}^{k=\infty}|\eref{top_order_Gaussian_s_less_than_half}|&\lesssim \int_0^{T}(1+t)^{1+\delta}\norm{g\jap{x-(t+1)v}^2}_{L^\infty_xL^2_v}\\
&\qquad\times\norm{(1+t)^{-\frac{1+\delta}{2}}\jap{x-(t+1)v}^2\jap{v}\der g}_{L^2_xL^2_v}^2 \d t\\
&\lesssim \mathbb{I}
\end{align*}
Combining the above bounds we get the desired result.\\
\emph{Case 2:} $s\geq 1$. For this case we have to bound the infinite sum of the terms \eref{top_order_diff_weights_mult}, \eref{top_order_diff_f}, \eref{top_order_Gaussian_1} and \eref{top_order_Gaussian_2}.

We use \lref{top_order_comm_weights} to get,
\begin{align*}
\sum_{k=0}^{k=\infty}|\eref{top_order_diff_weights_mult}|& \lesssim\int_0^{T}(1+t)^{3-\delta}\norm{g\jap{x-(t+1)v}^{2\delta}}_{L^\infty_xL^1_v}\\
&\qquad\times \norm{(1+t)^{-\frac{1}{2}-\frac{\delta}{2}}\der g\jap{x-(t+1)v}^2\jap{v}}^2_{L^2_xL^2_v}\d t\\
&\lesssim \mathbb{IV}.
\end{align*}
Next we use \lref{top_order_diff_f} to get,
\begin{align*}
\sum_{k=0}^{k=\infty}|\eref{top_order_diff_f}|&\lesssim \int_0^{T}(1+t)^{2+\delta}\norm{g}^{1-s}_{L^\infty_xL^1_v}\norm{\part_{v_i}g}^{s}_{L^\infty_xL^1_v}\norm{(1+t)^{-\frac{1}{2}-\frac{\delta}{2}}\jap{v}^{\frac{2s+\gamma}{2}}\der g}^2_{L^2_xL^2_v}\d t\\
&\quad + \int_0^{T}(1+t)^{2+\delta}\norm{g}^{\frac{3}{2}-s}_{L^\infty_xL^1_v}\norm{\part_{v_i}g}^{s-\frac{1}{2}}_{L^\infty_xL^1_v}\norm{(1+t)^{-\frac{1}{2}-\frac{\delta}{2}}\jap{v}^{\frac{2s+\gamma}{2}}\der g}^2_{L^2_xL^2_v}\d t\\
&\lesssim \mathbb{V}.
\end{align*}
Further, \lref{top_order_Gaussian_1} gives us,
\begin{align*}
\sum_{k=0}^{k=\infty}|\eref{top_order_Gaussian_1}|&\lesssim \int_0^{T}(1+t)^{1+\delta}\norm{g\jap{x-(t+1)v}^2}_{L^\infty_xL^2_v}\\
&\qquad\norm{(1+t)^{-\frac{1+\delta}{2}}\jap{x-(t+1)v}^2\jap{v}\der g}_{L^2_xL^2_v}^2\d t\\
&\lesssim \mathbb{I}.
\end{align*}
Finally, \lref{top_order_Gaussian_2} implies,
\begin{align*}
\sum_{k=0}^{k=\infty}|\eref{top_order_Gaussian_2}|&\lesssim \int_0^{T}(1+t)^{1+\delta}\norm{\part_{v_i} g}_{L^\infty_xL^1_v}\norm{(1+t)^{-\frac{1+\delta}{2}}\jap{x-(t+1)v}^2\jap{v}\der g}_{L^2_xL^2_v}^2 \d t\\
&\lesssim \mathbb{VI}.
\end{align*}
\end{proof}
\subsection{Penultimate order terms} This corresponds to $|\alpha''|+|\beta''|+|\omega''|=|\alpha|+|\beta|+|\omega|-1$. We again use the singularity decomposition and for $k<0$, we directly bound, 
\begin{equation}\label{e.pen_two_prime}
\begin{split}
\int_{0}^{T}\int_{\R^3}\int_{\R^3}\int_{\R^3}\int_{\S^2}&B_k \mu_{\beta''',\omega'''}(v_*)((\derv{'}{'}{'}g)_*'(\derv{''}{''}{''} g)'-(\derv{'}{'}{'}g)_*\derv{''}{''}{''} g)\\
&\quad \times\jap{x-(t+1)v}^4\der g\d \sigma\d v_*\d v\d x\d t 
\end{split}
\end{equation}
using \lref{trivial_est_1} and \lref{trivial_est_2}.\\
For $k\geq 0$, we need to make some changes to be able to treat the singularity
\textcolor{black}{
\begin{align}
B_k&\mu_{\beta''',\omega'''}(v_*)((\derv{'}{'}{'}g)_*'(\derv{''}{''}{''} g)'-(\derv{'}{'}{'}g)_*\derv{''}{''}{''} g)\nonumber\\
&\hspace{10em}\times\jap{x-(t+1)v}^4\der g\nonumber\\
&\equiv Q_k(\part_x^{\bar \alpha}\part_v^{\bar \beta}Y^{\bar \omega} (\mu g),\derv{''}{''}{''}g)\jap{x-(t+1)v}^4\der g\label{e.pen_term_main}\\
&\qquad +B_k(\mu_{\beta''',\omega'''}(v'_*)-\mu_{\beta''',\omega'''}(v_*))(\derv{'}{'}{'} g)_*'(\derv{''}{''}{''}g)'\jap{x-(t+1)v}^4\der g\label{e.Gaussian_diff_term},
\end{align}
where $\part_x^{\bar \alpha}\part_v^{\bar \beta}Y^{\bar \omega}\derv{''}{''}{''}=\der$. We used Leibnitz rule and the fact that $|\alpha'|+|\beta'|+|\beta'''|+|\omega'|+|\omega'''|= 1$ to absorb the Maxwellian into the derivative with $g$.}

 For the \eref{pen_term_main} we further specialize to three cases depending on $(|\bar \alpha|,|\bar\beta|,|\bar\omega|)$:
\begin{enumerate}
\item $(|\bar \alpha|,|\bar \beta|,|\bar\omega|)=(1,0,0)$. In this case $\der g=\part_{x_l}(\derv{''}{''}{''} g)$, where $\derv{''}{''}{''}$ is the same as in \eref{pen_two_prime}.\\
Let $\underline{\bar{G}_{\alpha'',\beta'',\omega''}=\jap{x-(t+1)v}^2\derv{''}{''}{''} g}$. From now on, we suppress its dependence on multi-indices.

By our definition we have,
$$\jap{x-(t+1)v}^2\part_{x_l}\derv{''}{''}{''} g=\part_{x_l}\bar G-2(x_l-(t+1)v_l) \derv{''}{''}{''} g.$$
Substituting the above equation we have
\begin{align*}
\int_0^{T}&\int_{\R^3}\int_{\R^3} Q_k(\part_{x_l} f,\derv{''}{''}{''} g)\jap{x-(t+1)v}^4\der g\d v \d x\d t\\
&\equiv\int_0^{T}\int_{\R^3}\int_{\R^3} Q_k(\part_{x_l} f,\bar G)\part_{x_l}\bar G\d v \d x\d t\\
&\quad - 2\int_0^{T}\int_{\R^3}\int_{\R^3} Q_k(\part_{x_l} f,\bar G)(x_l-(t+1)v_l)\derv{''}{''}{''} g\d v \d x\d t\\
&\quad+ \int_0^{T}\int_{\R^3}\int_{\R^3}[\jap{x-(t+1)v}^2Q_k(\part_{x_l} f,\derv{''}{''}{''} g)\\
&\hspace{10em}-Q_k(\part_{x_l} f,\jap{x-(t+1)v}^2\derv{''}{''}{''}g)]\\
&\hspace{14em}\times\jap{x-(t+1)v}^2\der g\d v \d x\d t.
\end{align*}
\item $(|\bar \alpha|,|\bar \beta|,|\bar\omega|)=(0,1,0)$. In this case $\der g=\part_{v_l}(\derv{''}{''}{''} g)$.\\
By our definition of $\bar G$ we have,
\begin{align*}
\jap{x-(t+1)v}^2\part_{v_l}\derv{''}{''}{''} g&=\part_{v_l}\bar G+2(t+1)(x_l-(t+1)v_l) \derv{''}{''}{''} g.
\end{align*}
Substituting the above equation we have
\begin{align*}
\int_0^{T}&\int_{\R^3}\int_{\R^3} Q_k(\part_{v_l} f,\derv{''}{''}{''} g)\jap{x-(t+1)v}^4\der f\d v \d x\d t\nonumber\\
&\equiv\int_0^{T}\int_{\R^3}\int_{\R^3} Q_k(\part_{v_l} f,\bar G)\part_{v_l}\bar G\d v \d x\d t\\
&\quad + 2(t+1)\int_0^{T}\int_{\R^3}\int_{\R^3} Q_k(\part_{v_l} f,\bar G)(x_l-(t+1)v_l)\derv{''}{''}{''} g\d v \d x\d t\\
&\quad+ \int_0^{T}\int_{\R^3}\int_{\R^3}[\jap{x-(t+1)v}^2Q_k(\part_{v_l} f,\derv{''}{''}{''} g)\\
&\hspace{10em}-Q_k(\part_{v_l} f,\jap{x-(t+1)v}^2\derv{''}{''}{''}g)]\\
&\hspace{14em}\times\jap{x-(t+1)v}^2\der g\d v \d x\d t.
\end{align*}
\item $(|\bar \alpha|,|\bar \beta|,|\bar\omega|)=(0,0,1)$. In this case $\der g=\part_{Y_l}(\derv{''}{''}{''} g)$.\\
First note that $Y_l$ commutes with $\jap{x-(t+1)v}$. Thus, by our definition of $\bar G$ we have,
\begin{align*}
\jap{x-(t+1)v}^2Y_l\derv{''}{''}{''} g={Y_l}\bar G.
\end{align*}
Substituting the above equation we have
\begin{align*}
\int_0^{T}&\int_{\R^3}\int_{\R^3}Q_k(Y_l f,\derv{''}{''}{''} g)\jap{x-(t+1)v}^4\der g\d v \d x\d t\\
&\equiv\int_0^{T}\int_{\R^3}\int_{\R^3} Q_k(Y_l f,\bar G)Y_l\bar G\d v \d x\d t\\
&\quad+ \int_0^{T}\int_{\R^3}\int_{\R^3} [\jap{x-(t+1)v}^2Q_k(Y f,\derv{''}{''}{''} g)\\
&\hspace{10em}-Q_k(Y f,\jap{x-(t+1)v}^2\derv{''}{''}{''}g)]\\
&\hspace{14em}\times\jap{x-(t+1)v}^2\der g\d v \d x\d t.
\end{align*}
\end{enumerate}
Hence we have in total that,
\begin{align}
|\eref{pen_term_main}|&\lesssim \left|\int_0^{T}\int_{\R^3} \int_{\R^3} Q_k(\part_x^{\bar \alpha}\part_v^{\bar \beta}Y^{\bar \omega} f,\bar G)\part_x^{\bar \alpha}\part_v^{\bar \beta}Y^{\bar \omega}\bar G\d v \d x\d t\right|\label{e.pen_main}\\
&\quad+(1+t)^{|\beta'|}\int_0^{T}\int_{\R^3}\int_{\R^3} |Q_k(\part_x^{\bar \alpha}\part_v^{\bar \beta}Y^{\bar \omega} f,\bar G)\bar G|\d v \d x\d t\label{e.pen_term_derv_weight}\\
&\quad+ \int_0^{T}\int_{\R^3}\int_{\R^3} |[\jap{x-(t+1)v}^2Q_k(\part_x^{\bar \alpha}\part_v^{\bar \beta}Y^{\bar \omega} f,\derv{''}{''}{''} g)\nonumber\\
&\hspace{3em}-Q_k(\part_x^{\bar \alpha}\part_v^{\bar \beta}Y^{\bar \omega} f,\jap{x-(t+1)v}^2\derv{''}{''}{''}g)]\jap{x-(t+1)v}^2\der g|\d v \d x\d t\label{e.pen_comm_term}.
\end{align}

We first look at the main term \eref{pen_main} which is of the form $$\int_{0}^{T}\int_{\R^3}\int_{\R^3}Q_k(\part f,\bar G)\part \bar G\d v\d x\d t,$$
where $\part=\{\part_{x_i},\part_{v_i},Y_i\}$ and $\bar G=\jap{x-(t+1)v}^2\derv{''}{''}{''} g.$

To take care of the singularity in the angle we need to make more changes to the collisional kernel. Decomposing the kernel we get $Q_k(\part f,\bar G)=Q_{1,k}(\part f,\bar G)+Q_{2,k}(\part f,\bar G)$. We can take care of $Q_{2,k}(\part f,\bar F)$ using \lref{canc_lemma}. 
\begin{proposition}\label{p.pen_term_main}
For the term involving the singular part of the Boltzmann kernel, we have the following equality
\begin{align}
\int_{0}^{T}\int_{\R^3}\int_{\R^3}\int_{\R^3}\int_{\S^2}&Q_{1,k}(\part f,\bar G)\part \bar G\d \sigma\d v_*\d v\d x\d t\nonumber\\ 
&\equiv\frac{1}{4}\int_0^{T}\int_{\R^3}\int_{\R^3}\int_{\R^3}\chi_{k}(|v-v'|)K_{\part^2 f}(\bar G'-\bar G)^2\d v'\d v\d x\d t\label{e.pen_term_main_1}\\
&\quad\frac{1}{2}\int_0^{T}\int_{\R^3}\int_{\R^3}\int_{\R^3}\chi_{k}(|v-v'|)(K_{\part f}-K_{\part f}')(\bar G'-\bar G)\part\bar G\d v'\d v\d x\d t \label{e.pen_term_main_2}.
\end{align}
Here $\chi_k(|v-v'|)$ is the same dyadic partition used in \sref{sing}.
\end{proposition}
\begin{proof}
\textbf{We begin by emphasizing that the $\part$ acting on $f$ and $\bar G$ is of the same form.}\\
Now to prove the claim above we need to perform integration by parts,\\ 
\emph{Case 1:} $\part=\part_{x_i}$. We perfom integration by parts in $x$ twice but for brevity we drop the integrals.
\begin{align*}
B_k(\part_{x_i}f)_*'&(\bar G'-\bar G)\part_{x_i}\bar G\\
&\equiv B_k(\part_{x_i}^2f)'_*(\bar G'-\bar G)^2-B_k(\part_{x_i}^2f)'_*(\bar G'-\bar G)\bar G'\\
&\quad-B_k(\part_{x_i}f)_*'((\part_{x_i} \bar G)'-\part_{x_i}\bar G)(\bar{G}-\bar{G}')-B_k(\part_{x_i}f)_*'((\part_{x_i} \bar G)'-\part_{x_i}\bar G)\bar{G}'\\
&\equiv B_k(\part^2_{x_i}f)'_*(\bar G'-\bar G)^2+\cancel{B_k(\part_{x_i}f)_*'((\part_{x_i} \bar G)'-\part_{x_i}\bar G)\bar G'}+B_k(\part_{x_i}f)_*'(\bar G'-\bar G)(\part_{x_i}\bar{G})'\\
&\quad-B_k(\part_{x_i}f)_*'((\part_{x_i} \bar G)'-\part_{x_i}\bar G)(\bar G-\bar G')-\cancel{B_k(\part_{x_i}f)_*'((\part_{x_i} \bar G)'-\part_{x_i}\bar G)\bar G'}\\
&\equiv B_k(\part^2_{x_i}f)'_*(\bar G'-\bar G)^2+2B_k((\part_{x_i}f)_*'-(\part_{x_i}f)_*)(\bar G'-\bar G)(\part_{x_i}\bar{G})'\\
&\quad+2B_k(\part_{x_i}f)_*(\bar G'-\bar G)(\part_{x_i}\bar G)'-B_k(\part_{x_i}f)_*'(\bar G'-\bar G)\part_{x_i}\bar G.
\end{align*}

Next note by pre-post collision change of variables, $$B_k(\part_{x_i}f)_*'(\bar G'-\bar G)\part_{x_i} \bar{G}\equiv -B_k(\part_{x_i}f)_*(\bar G'-\bar G)(\part_{x_i} \bar{G})',$$
and 
$$B_k((\part_{x_i}f)_*'-(\part_{x_i}f)_*)(\bar G'-\bar G)(\part_{x_i}\bar{G})'\equiv B_k((\part_{x_i}f)_*'-(\part_{x_i}f)_*)(\bar G'-\bar G)\part_{x_i}\bar{G}.$$
Using these observations and plugging them into the above equation we get
\begin{align*}
4B_k&(\part_{x_i}f)_*'(\bar G'-\bar G)\part_{x_i}\bar G\\
&\equiv B_k(\part_{x_i}^2f)'_*(\bar G'-\bar G)^2+2B_k((\part_{x_i}f)_*'-(\part_{x_i}f)_*)(\bar G'-\bar G)\part_{x_i} \bar G\\
&\equiv \color{black}{\int_0^{T}\int_{\R^3}\int_{\R^3}\int_{\R^3}\chi_{k}(|v-v'|)[K_{\part^2_{x_i}f}(\bar G'-\bar G)^2+2(K_{\part_{x_i}f}-K_{\part_{x_i}f}')(\bar G'-\bar G)\part_{x_i}\bar G]\d v'\d v\d x\d t.}
\end{align*}
For the last equality, we use the change of variables in \lref{change_of_variables_sil}.\\
\emph{Case 2:} $\part=\part_{v_i}.$ This case is a little more complicated but we still perform integration by parts twice in $v$. This time when we perform the first integration by parts we have the derivative falling on the kernel as well. We use the idea used in \cite{Vil98} to get over this issue,
\begin{align*}
\part_{v_i}[B_k&(v-v_*,\sigma)(\part_{v_i}f)_*'(\bar G'-\bar G)]\bar G\\
&\equiv \part_{v_i}(B_k(v-v_*,\sigma))(\part_{v_i}f)_*'(\bar G'-\bar G)\bar G+B_k(v-v_*,\sigma)\part_{v_i}[(\part_{v_i}f)_*'(\bar G'-\bar G)]\bar G\\
&\equiv - \part_{v_{*_i}}(B_k(v-v_*,\sigma))(\part_{v_i}f)_*'(\bar G'-\bar G)\bar G+B_k(v-v_*,\sigma)\part_{v_i}[(\part_{v_i}f)_*'(\bar G'-\bar G)]\bar G\\
&\equiv B_k(v-v_*,\sigma)(\part_{v_i}+\part_{v_{*_i}})[(\part_{v_i}f)_*'(\bar G'-\bar G)]\bar G\\
&\equiv B_k (\part_{v_i}^2 f)_*'(\bar G'-\bar G)\bar G+B_k(\part_{v_i}f)_*'((\part_{v_i} \bar G)'-\part_{v_i} \bar G)\bar{G}.
\end{align*}
We now proceed in the same way as in Case 1 and perform an additional integration by parts. This poses no issues as $(\part^2_{v_i} f)_*'=(\part_{v_i}+\part_{{v_*}_i})(\part_{v_i}f)_*'$, $(\part_{v_i}+\part_{{v_*}_i}) B=0$ and \\$(\part_{v_i}+\part_{{v_*}_i})\bar{G}'=(\part_{v_i}\bar{G})'$.\\
\emph{Case 3:} $\part=Y_i$. For this case we break up $Y_i$ as $(t+1)\part_{x_i}+\part_{v_i}$ and then apply Case 1 to the first term and Case 2 to the second term and then combine the appropriate terms. \textcolor{black}{Indeed, integrating by parts in $x$ and $v$ we get
\begin{align*}
B_k(Y_i f)_*'&(\bar G'-\bar G)Y_i\bar G\\
&\equiv B_k(Y_i f)_*'(\bar G'-\bar G)((t+1)\partial_{x_i}+\part_{v_i})\bar G\\
&\equiv B_k(Y_i^2f)'_*(\bar G'-\bar G)^2-B_k(Y_i^2f)'_*(\bar G'-\bar G)\bar G'\\
&\quad-B_k(Y_i f)_*'((Y_i \bar G)'-Y_i\bar G)(\bar{G}-\bar{G}')-B_k(Y_if)_*'((Y_i \bar G)'-Y_i\bar G)\bar{G}'.
\end{align*}
Next we integrate by parts again but since the details are the same as in the cases above, we skip them.}
\end{proof}
\begin{lemma}\label{l.sym_bound_weight}
Let $\mf$ and $\mg$ be any smooth functions and $K_{|\mf|}(v',v)$ be defined as in \lref{boltz_kernel} then for $\gamma+2s\geq 0$, we have the following bound,
$$\int_{\R^3}\int_{\R^3} K_{|\mf|}(\mg'-\mg)^2\d v'\d v \lesssim \norm{\mf\jap{v}^{2s+\gamma}}_{L^1_v}\norm{\mg\jap{v}^{(s+\frac{\gamma}{2})_+}}_{H^s_v}^2.$$
\end{lemma}
\begin{proof}
We first write 
\begin{align*}\int_{\R^3}\int_{\R^3} K_{|\mf|}(\mg'-\mg)^2\d v'\d v&=\int\int_{\{(v,v'):|v-v'|\leq 1\}} K_{|\mf|}(\mg'-\mg)^2\d v'\d v\\
&\quad+\int\int_{\{(v,v'):|v-v'|> 1\}} K_{|\mf|}(\mg'-\mg)^2\d v'\d v\\
&:=I_1+I_2.
\end{align*}

To estimate the first term, we follow \cite{HeSnTa19} and decompose the integral into a sum over compact sets and apply a cut-off. Let $\chi$ be a smooth function such that it is one on $B_{10}$ and zero outside $B_{20}.$

Then we have that $$I_1\lesssim \sum_{z\in \Z^3}\int_{B_{10}(z)}\int_{B_{10}(z)}K_{|\mf|}(\mg'-\mg)^2\d v'\d v\lesssim \sum_{z\in \Z^3}\int_{\R^3}\int_{\R^3} K_{|\mf|}(\chi'\mg'-\chi \mg)^2\d v'\d v.$$

Now by a simple adaptation of [Lemma 4.2, \cite{ImbSil16}] and noting that $\jap{v}\approx \jap{v'}\lesssim \jap{z}$ for each piece in the sum we get
\begin{align*}I_1&\lesssim  \sum_{z\in \Z^3} \jap{z}^{\gamma+2s}\norm{\mf\jap{v}^{\gamma+2s}}_{L^1_v}\norm{\chi \mg}_{H^s_v}^2\\
&\lesssim \sum_{z\in \Z^3} \frac{\norm{\mf\jap{v}^{\gamma+2s}}_{L^1_v}\norm{\mg\jap{v}^{(s+\frac{\gamma}{2})_+}}_{H^s_v(B_{20})}^2}{\jap{z}^{\eps}}\\
&\lesssim \norm{\mf\jap{v}^{\gamma+2s}}_{L^1_v}\norm{\mg\jap{v}^{(s+\frac{\gamma}{2})_+}}_{H^s_v}^2.
\end{align*}

For $I_2$, we use triangle inequality and \lref{K_bound} to get,
\begin{align*}
I_2&\lesssim \int\int_{\{(v,v'):|v-v'|> 1\}} \mg(v')^2 K_{|\mf|}(v,v') \d v' \d v+\int\int_{\{(v,v'):|v-v'|> 1\}} \mg(v)^2 K_{|\mf|}(v,v') \d v' \d v\\
&\lesssim \int \jap{v'}^{\gamma+2s}\mg(v')^2 \left(\jap{v'}^{-\gamma-2s}\int_{B_1(v')^c} K_{|\mf|}(v,v')\d v\right)\d v'\\
&\quad+\int \jap{v}^{\gamma+2s}\mg(v)^2 \left(\jap{v}^{-\gamma-2s}\int_{B_1(v)^c} K_{|\mf|}(v,v')\d v'\right)\d v\\
&\lesssim \norm{\mf\jap{v}^{2s+\gamma}}_{L^1_v}\norm{\mg\jap{v}^{s+\frac{\gamma}{2}}}_{L^2_v}^2.
\end{align*}
\end{proof}
\begin{lemma}\label{l.pen_term_main_1}
For $\gamma+2s\in(0,2)$, we have the following bound for $\sum_{k=0}^{k=\infty}|\eref{pen_term_main_1}|$,
$$\sum_{k=0}^{k=\infty}|\eref{pen_term_main_1}|\lesssim \int_{0}^{T}(1+t)^{1+\delta}\norm{\part^2 g}_{L^\infty_xL^1_v}\norm{(1+t)^{-\frac{1+\delta}{2}}\jap{v}\jap{x-(t+1)v}^2\derv{''}{''}{''} g}_{L^2_x H^s_v}^2.$$
\end{lemma}
\begin{proof}
First note that $$|\eref{pen_term_main_1}|\lesssim \int_0^{T}\int_{\R^3}\int_{\R^3}\int_{\R^3}\chi_{k}(|v-v'|)K_{|\part^2 f|}(\bar G'-\bar G)^2\d v'\d v\d x\d t.$$
Next using monotone convergence theorem, we get that
$$\sum_{k=0}^{k=\infty}|\eref{pen_term_main_1}|\lesssim \int_0^{T}\int_{\R^3}\int_{\R^3}\int_{\R^3}K_{|\part^2 f|}(\bar G'-\bar G)^2\d v'\d v\d x\d t.$$
Finally using \lref{sym_bound_weight} with $\mf=\part^2 f$, $\mg=\jap{x-(t+1)v}^2\derv{''}{''}{''} g$, we get the required bound
\end{proof}
\begin{lemma}\label{l.diff_to_gradient}
Let $\mf:\R^3\to \R$ be a differentiable function. The following inequality\footnote{see Lemma 4.5 in \cite{ImbSil16} for a similar result.} holds for any pair $v,v'\in \R^3$.
$$\left|\mf(v')-\mf(v)\right|\lesssim |v-v'|^{2}\int_{B_R(m)}\frac{|\part_{v_i}\mf| (w)}{|w-v|^{2}|w-v'|^{2}}\d w.$$

Here $R=\frac{|v-v'|}{2}$ and $m=\frac{v+v'}{2}$.
\end{lemma} 
\begin{proof}
For any $w\in B_R(m)$, we have $$|\mf(w)-\mf(v)|\leq \int_0^{|w-v|}|\part_{v_i}\mf|(v+z\widehat{w-v})\d z,$$
where $\widehat{w-v}=\frac{w-v}{|w-v|}$. Thus, using spherical coordinates with zero at $w=v$ we have 
\begin{align*}
\left|\fint_{B_R(m)} \mf(w)\d w-\mf(v)\right|&=\left|\fint_{B_R(m)} [\mf(w)-\mf(v)]\d w\right|\\
&\lesssim \fint_{B_R(m)} \left(\int_{0}^{|w-v|} |\part_{v_i} \mf|(v+z\widehat{w-v})\d z\right)\d w\\
&\lesssim \int_{B_R(m)}\frac{|\part_{v_i}\mf|(w)}{|w-v|^{2}}\d w.
\end{align*}

This implies that $$\left|\fint_{B_R(m)} \mf(w)\d w-\mf(v)\right|\lesssim R^{2}\int_{B_R(m)}\frac{|\part_{v_i}\mf| (w)}{|w-v|^{2}|w-v'|^{2}}\d w.$$
Exchanging the role of $v$ and $v'$ and substraction the resulting inequalities gives us the required result.
\end{proof}

We next define an auxiliary kernel in the spirit of \cite{ImbSil16}. Let $$\aux(v,w)=\int_{\{v'\in B_R(v):(v'-v)\cdot(w-v)\geq |w-v|^2\}}\frac{|v'-v|^{2}(K_{\mf}(v,v')-K_{\mf}(v',v))}{|w-v|^{2}|w-v'|^{2}}\d v'.$$
\begin{lemma}\label{l.diff_aux_est}
Let $\mf:\R^3\to \R$ be a differentiable function then for $|v-v'|<R\leq 1$ and $w$ such that $\jap{w}\approx \jap{v}\approx \jap{v'}$ we have that 
$$\int_{B_R(v)}\jap{v}^{-\gamma-2s}|\aux(v,w)|\d w\lesssim R^{2-2s}\norm{\part_{v_i} \mf\jap{v}^{2s+\gamma}}_{L^1_v},$$
and 
$$\int_{B_R(w)}\jap{w}^{-\gamma-2s}|\aux(v,w)|\d v\lesssim R^{2-2s}\norm{\part_{v_i} \mf\jap{v}^{2s+\gamma}}_{L^1_v}.$$
\end{lemma}
\begin{proof}
For the first inequality, we have
\begin{align*}
\int_{B_R(v)} &|\aux(v,w)|\d w\\
&\leq \int_{B_R} \int_{\{v'\in B_R(v):(v'-v)\cdot(w-v)\geq |w-v|^2\}}\frac{|v'-v|^{2}|K_{\mf}(v,v')-K_{\mf}(v',v)|}{|w-v|^{2}|w-v'|^{2}}\d v'\d w\\
&=\int_{B_R(v)} |v-v'|^{2}|K_{\mf}(v,v')-K_{\mf}(v',v)|\left(\int_{B_{\frac{|v-v'|}{2}}(\frac{v+v'}{2})} \frac{1}{|w-v|^{2}|w-v'|^{2}}\d w\right)\d v'.
\end{align*}

Next we perform the change of variables $w\mapsto w+v$. We use this change of variables and \cref{diff_f_cor} to get,
\begin{align*}
\int_{B_R(v)}& |v-v'|^{2}|K_{\mf}(v,v')-K_{\mf}(v',v)|\left(\int_{B_{\frac{|v-v'|}{2}}(\frac{v+v'}{2})} \frac{1}{|w-v|^{2}|w-v'|^{2}}\d w\right)\d v'\\
&=\int_{B_R(v)} |v-v'|^{2}|K_{\mf}(v,v')-K_{\mf}(v',v)|\left(\int_{B_{\frac{|v-v'|}{2}}(\frac{v'-v}{2})} \frac{1}{|w|^{2}|w-(v'-v)|^{2}}\d w\right)\d v'\\
&=C\int_{B_R(v)} |K_{\mf}(v,v')-K_{\mf}(v',v)||v-v'|\d v'\lesssim R^{2-2s}\jap{v}^{\gamma+2s}\int_{\R^3}|\part_{v_i}\mf|(z)\jap{z}^{\gamma+2s}\d z.
\end{align*}

For the second inequality we use triangle inequality $$|v-v'|^{2}\leq |w-v|^{2}+|w-v'|^{2}.$$ Thus, we have
\begin{align*}
\jap{w}^{-\gamma-2s}&\int_{B_R(w)}\aux(v,w)\d v\\
&\lesssim \int_{B_R(w)}\int_{\{v':|v-v'|<R, (v'-v)\cdot (w-v)\geq |w-v|^2\}}\jap{w}^{-2s-\gamma}\frac{|K_{\mf}(v,v')-K_{\mf}(v',v)|}{|w-v'|^{2}}\d v'\d v \\
&\quad+  \int_{B_R(w)}\int_{\{v':|v-v'|<R, (v'-v)\cdot (w-v)\geq |w-v|^2\}}\jap{w}^{-2s-\gamma}\frac{|K_{\mf}(v,v')-K_{\mf}(v',v)|}{|w-v|^{2}}\d v'\d v.
\end{align*}

Using \lref{diff_K_outside_ball}, Fubini's theorem and that $\jap{w}\approx \jap{v}\approx\jap{v'}$, we have
\begin{align*}
\int_{B_R(w)}&\int_{\{v':|v-v'|<R, (v'-v)\cdot (w-v)\geq |w-v|^2\}}\jap{w}^{-2s-\gamma}\frac{|K_{\mf}(v,v')-K_{\mf}(v',v)|}{|w-v'|^{2}}\d v'\d v\\
&=\int_{B_{2R}(w)}\int_{\{v':|v-v'|<R, (w-v')\cdot (w-v)\leq 0\}\cap B_R(w)}\jap{w}^{-2s-\gamma}\frac{|K_{\mf}(v,v')-K_{\mf}(v',v)|}{|w-v'|^{2}}\d v\d v'\\
&\lesssim \int_{B_{2R}(w)}\int_{\{v':|v-v'|<R, |v-v'|\geq |w-v'|\}\cap B_R(w)}\jap{w}^{-2s-\gamma}\frac{|K_{\mf}(v,v')-K_{\mf}(v',v)|}{|w-v'|^{2}}\d v\d v'\\
&\lesssim \norm{\part_{v_i} \mf\jap{v}^{\gamma+2s}}\int_{B_{2R}(w)} \frac{1}{|w-v'|^{1+2s}}\d v'\\
&\lesssim R^{2-2s}\norm{\part_{v_i} \mf\jap{v}^{\gamma+2s}}_{L^1_v}.
\end{align*}

In a similar way we can prove the required bound for the second term. In that case we do not need to use Fubini's theorem.
\end{proof}

\begin{lemma}\label{l.diff_f_diff_G} Let $\mf$, $\mg$ and $\mh$ be any smooth functions then for $\gamma+2s\in(0,2]$, we have the following bound for $$\sum_{k=0}^{k=\infty}\left|\int_0^{T}\int_{\R^3}\int_{\R^3}\int_{\R^3} \chi_{k}(|v-v'|)(K_{\mf}-K_{\mf'})(\mg'-\mg)\mh\d v'\d v\d x\d t\right|:=I,$$
\[   
I\lesssim\\
     \begin{cases}
\begin{split}
		\int_0^{T}(1+t)^{1+\delta}&\norm{\part_{v_i} \mf\jap{v}^{2}}_{L^\infty_xL^1_v}\norm{(1+t)^{-\frac{1+\delta}{2}}\jap{v}\mg}_{L^2_xL^2_v}\\
&\qquad\times\norm{(1+t)^{-\frac{1+\delta}{2}}\jap{v}\mh}_{L^2_xL^2_v}\d t
\end{split}&\text{if } s\in \left(0,\frac{1}{2}\right)\\ \\

\begin{split}
		&\int_0^{T}(1+t)^{1+2\delta}\norm{\part_{v_i} \mf\jap{v}^{2}}_{L^\infty_xL^1_v}\norm{(1+t)^{-\frac{1+2\delta}{2}}\jap{v}\mh}_{L^2_xL^2_v}\\
&\hspace{5em}\times\norm{(1+t)^{-\frac{1+2\delta}{2}}\jap{v}\mg}^{(2s-1)_+}_{L^2_xH^{1}_v}\norm{(1+t)^{-\frac{1+2\delta}{2}}\jap{v}\mg}^{(2-2s)_-}_{L^2_xL^2_v}\d t\\
&+\int_0^{T}(1+t)^{1+2\delta}\norm{\part_{v_i} \mf\jap{v}^{2}}_{L^\infty_xL^1_v}\norm{(1+t)^{-\frac{1+2\delta}{2}}\jap{v}\mh}_{L^2_xL^2_v}\\
&\qquad\times\norm{(1+t)^{-\frac{1+2\delta}{2}}\jap{v}\mg}^{2s-1}_{L^2_xH^{1}_v}\norm{(1+t)^{-\frac{1+2\delta}{2}}\jap{v}\mg}^{2-2s}_{L^2_xL^2_v}\d t
\end{split}&\text{if } s\in \left[\frac{1}{2},1\right).
     \end{cases}
\] 

\end{lemma}
\begin{proof}
\emph{Case 1:} $s\geq \frac{1}{2}$.\\
Fix a $0<R<1$ and let $k_1$ be such that $R\in [2^{-k_1-1},2^{-k_1}]$.\\
Then we have two cases,\\
\emph{Case 1a):} $k\leq k_1$. In this case we only exploit the cancellation from $\mf.$
\begin{align*}
\int_{\R^3}\int_{\R^3}& |\chi_{k}(|v-v'|)(K_{\mf}-K_{\mf}')(\mg'-\mg)|\mh\d v'\d v\\
&\lesssim \norm{\mh\jap{v}^{\frac{\gamma}{2}+s}}_{L^2_v}\left(\int_{\R^3}\jap{v}^{-\gamma-2s}\left(\int_{\R^3}\chi_{k}(|v-v'|)|K_{\mf}-K_{\mf}'||\mg'-\mg|\d v'\right)^2\d v\right)^{\frac{1}{2}}:=I_k.
\end{align*}

Let $T_{\mf}(v,v')=|K_{\mf}(v,v')-K_{\mf}(v',v)|$. 
\begin{align*}
&\left(\int_{\R^3}\jap{v}^{-\gamma-2s}\left(\int_{2^{-k-1}\leq\{|v-v'|\leq 2^{-k}\}}T_{\mf}|\mg'-\mg|\d v'\right)^2\d v\right)^{\frac{1}{2}}\\
&\quad \lesssim \left(\int_{\R^3}\jap{v}^{-\gamma-2s}\left(\int_{\{2^{-k-1}\leq|v-v'|\leq 2^{-k}\}}T_{\mf}|\mg'|\d v'\right)^2\d v\right)^{\frac{1}{2}}\\
&\qquad+\left(\int_{\R^3}\jap{v}^{-\gamma-2s}\left(\int_{\{2^{-k-1}\leq|v-v'|\leq 2^{-k}\}}T_{\mf}|\mg|\d v'\right)^2\d v\right)^{\frac{1}{2}}.
\end{align*}
Using Cauchy--Schwarz, Fubini and \cref{diff_f_cor} we get 
\begin{align*}
&\left(\int_{\R^3}\jap{v}^{-\gamma-2s}\left(\int_{\{2^{-k-1}\leq|v-v'|\leq 2^{-k}\}}T_{\mf}|\mg'|\d v'\right)^2\d v\right)^{\frac{1}{2}}\\
&\qquad\lesssim \left(\int_{\R^3} \left(\jap{v}^{-\gamma-2s}\int_{\{2^{-k-1}\leq|v-v'|\leq 2^{-k}\}} T_{\mf} \d v'\right)\left(\int_{\{2^{-k-1}\leq|v-v'|\leq 2^{-k}\}} T_{\mf} {\mg}^2\d v'\right)\d v\right)^{\frac{1}{2}}\\
&\qquad\lesssim 2^{\frac{-k(1-2s)}{2}}\norm{\part_{v_i}\mf\jap{v}^{\gamma+2s}}^{\frac{1}{2}}_{L^1_v}\left(\int_{\R^3}{\mg'}^2 \jap{v'}^{\gamma+2s}\jap{v'}^{-\gamma-2s}\int_{\{2^{-k-1}\leq|v-v'|\leq 2^{-k}\}} T_{\mf} \d v\d v'\right)^{\frac{1}{2}}\\
&\qquad \lesssim 2^{-k(1-2s)}\norm{\part_{v_i}\mf\jap{v}^{\gamma+2s}}_{L^1_v}\norm{\jap{v}^{\frac{\gamma}{2}+s} \mg}_{L^2_v}.
\end{align*}
The other term is taken care of in the same way but this time we don't use Fubini.\\

Thus,
\begin{align*}
\int_{\R^3}\int_{\R^3} &|\chi_{k}(|v-v'|)(K_{\mf}-K_{\mf'})(\mg'-\mg)\mh|\d v'\d v\\
&\lesssim 2^{-k(1-2s)}\norm{\part_{v_i}\mf\jap{v}^{\gamma+2s}}_{L^1_v}\norm{\jap{v}^{\gamma+2s} \mg}_{L^2_v} \norm{\mh\jap{v}^{\frac{\gamma}{2}+s}}_{L^2_v}.
\end{align*}
\emph{Case 1b):} $k> k_1$. For this case we use use \lref{diff_to_gradient} to get
\begin{align*}
&\int_{\R^3}\jap{v}^{-\gamma-2s}\left(\int_{\{2^{-k-1}\leq|v-v'|\leq 2^{-k}\}}T_{\mf}|\mg'-\mg|\d v'\right)^2\d v\\
&\qquad \lesssim \int_{\R^3}\jap{v}^{-\gamma-2s}\left(\int_{B_{2^{-k}}(v)}\int_{B_r(m)}|\part_{v_i} \mg|(w)\frac{|v-v'|^2T_{\mf}}{|w-v|^2|w-v|^2}\d w\d v'\right)^2\d v,
\end{align*}
where $r=\frac{|v-v'|}{2}$ and $m=\frac{v+v'}{2}$. Since $|v-v'|\leq 1$, we have for all $w\in B_r(m)$, $\jap{w}\approx \jap{v}\approx\jap{v'}.$

Using Fubini's theorem we have
\begin{align*}
\int_{\R^3}&\jap{v}^{-\gamma-2s}\left(\int_{B_{2^{-k}}(v)}\int_{B_r(m)}|\part_{v_i} \mg|(w)\frac{|v-v'|^2T_{\mf}}{|w-v|^2|w-v|^2}\d w\right)\d v'\d v\\
&\qquad \lesssim \int_{\R^3}\jap{v}^{-\gamma-2s}\left(\int_{B_{2^{-k}}(v)}|\part_{v_i} \mg|(w)\left(\int_{\{v':(v'-v)\cdot(w-v)\geq |w-v|^2\}}\frac{|v-v'|^2T_{\mf}}{|w-v|^2|w-v|^2}\d v'\right)\d w\right)^2\d v.
\end{align*}

In view of definition of $\aux(v,w)$ and the fact that $\jap{w}\approx \jap{v}\approx\jap{v'}$, we use \lref{diff_aux_est} and Cauchy--Schwarz to get
\begin{align*}
\int_{\R^3}&\jap{v}^{-\gamma-2s}\left(\int_{\{2^{-k-1}\leq|v-v'|\leq 2^{-k}\}}T_{\mf}|\mg'-\mg|\d v'\right)^2\d v\\
&\lesssim \int_{\R^3}\jap{v}^{-\gamma-2s}\left(\int_{B_{2^{-k}}(v)}|\part_{v_i} \mg|(w)\aux(v,w)\d w\right)^2\d v\\
&\lesssim \int_{\R^3}\left(\int_{B_{2^{-k}}(v)}\jap{v}^{-2s-\gamma}\aux(v,w)\d w\right)\left(\int_{B_{2^{-k}}(v)}|\part_{v_i} \mg|^2(w)\aux(v,w)\d w\right)\d v\\
&\lesssim 2^{-k(2-2s)}\norm{\part_{v_i} \mf\jap{v}^{2s+\gamma}}_{L^1_v} \int_{\R^3}\int_{B_{2^{-k}}(v)} |\part_{v_i} \mg|^2(w)\aux(v,w)\d w\d v\\
&\lesssim 2^{-k(2-2s)}\norm{\part_{v_i} \mf\jap{v}^{2s+\gamma}}_{L^1_v} \int_{\R^3}\jap{w}^{2s+\gamma}|\part_{v_i} \mg|(w)^2\left(\int_{B_{2^{-k}}(w)} \jap{w}^{-2s-\gamma}\aux(v,w)\d v\right)\d w\\&\lesssim 2^{-2k(2-2s)}\norm{\part_{v_i} \mf\jap{v}^{2s+\gamma}}^2_{L^1_v} \norm{\part_{v_i}\mg\jap{v}^{\frac{\gamma}{2}+s}}^2_{L^2_v}.
\end{align*}
Thus taking square roots for the case $k\geq k_1$ we get 
\begin{align*}
\int_{\R^3}\int_{\R^3}& |\chi_{k}(|v-v'|)(K_{\mf}-K_{\mf'})(\mg'-\mg)\mh|\d v'\d v\\
&\lesssim 2^{-k(2-2s)}\norm{\part_{v_i} \mf\jap{v}^{2s+\gamma}}_{L^1_v}\norm{\jap{v}^{\frac{\gamma}{2}+s} \mg}_{H^1_v} \norm{\mh\jap{v}^{2s+\gamma}}_{L^2_v}.
\end{align*}

Using Case 1a) as above we have the following bound, 
\begin{align*}
\sum_{k=0}^{k=k_1}|I_k|&\lesssim \sum_{k=0}^{k=k_1} 2^{-k(1-2s)}\norm{\part_{v_i}\mf\jap{v}^{\gamma+2s}}_{L^1_v}\norm{\jap{v}^{\frac{\gamma}{2}+s} \mg}_{L^2_v} \norm{\mh\jap{v}^{\frac{\gamma}{2}+s}}_{L^2_v}\\
&\lesssim 2^{k_1(2s-1)}\norm{\part_{v_i}\mf\jap{v}^{\gamma+2s}}_{L^1_v}\norm{\jap{v}^{\frac{\gamma}{2}+s} \mg}_{L^2_v} \norm{\mh\jap{v}^{\frac{\gamma}{2}+s}}_{L^2_v}\\
&\lesssim R^{1-2s}\norm{\part_{v_i}\mf\jap{v}^{\gamma+2s}}_{L^1_v}\norm{\jap{v}^{\frac{\gamma}{2}+s} \mg}_{L^2_v} \norm{\mh\jap{v}^{\frac{\gamma}{2}+s}}_{L^2_v}.
\end{align*}
Here we used that $s>\frac{1}{2}$. For $s=\frac{1}{2}$ we can replace $R^{1-2s}$ by $R^{-\eta}$ for any $\eta>0$.

Using Case 1b) we get
\begin{align*}
\sum_{k=k_1+1}^{k=\infty}|I_k|&\lesssim \sum_{k=k_1+1}^{k=\infty} 2^{-k(2-2s)}\norm{\part_{v_i}\mf\jap{v}^{\gamma+2s}}_{L^1_v}\norm{\jap{v}^{\frac{\gamma}{2}+s} \mg}_{H^1_v} \norm{\mh\jap{v}^{\frac{\gamma}{2}+s}}_{L^2_v}\\
&\lesssim 2^{-k_1(2-2s)}\norm{\part_{v_i}\mf\jap{v}^{\gamma+2s}}_{L^1_v}\norm{\jap{v}^{\frac{\gamma}{2}+s} \mg}_{H^1_v} \norm{\mh\jap{v}^{\frac{\gamma}{2}+s}}_{L^2_v}\\
&\lesssim R^{2-2s}\norm{\part_{v_i}\mf\jap{v}^{\gamma+2s}}_{L^1_v}\norm{\jap{v}^{\frac{\gamma}{2}+s} \mg}_{H^1_v} \norm{\mh\jap{v}^{\frac{\gamma}{2}+s}}_{L^2_v}.
\end{align*}

Now let $$R=\frac{\norm{\jap{v}^{\frac{\gamma}{2}+s} \mg}_{L^2_v}}{\norm{\jap{v}^{\frac{\gamma}{2}+s} \mg}_{H^1_v}}.$$
This choice implies, 
\begin{align*}
\sum_{k=0}^{k=k_1}|I_k|&\lesssim \norm{\part_{v_i}\mf\jap{v}^{\gamma+2s}}_{L^1_v}\norm{\jap{v}^{\frac{\gamma}{2}+s} \mg}^{\max\{2s-1,\eta\}}_{H^{1}_v}\norm{\jap{v}^{\frac{\gamma}{2}+s} \mg}^{\min\{2-2s,1-\eta\}}_{L^2_v} \norm{\mh\jap{v}^{\frac{\gamma}{2}+s}}_{L^2_v}.
\end{align*}
and 
$$\sum_{k=k_1+1}^{k=\infty}|I_k|\lesssim \norm{\part_{v_i}\mf\jap{v}^{\gamma+2s}}_{L^1_v}\norm{\jap{v}^{\frac{\gamma}{2}+s} \mg}^{2s-1}_{H^{1}_v}\norm{\jap{v}^{\frac{\gamma}{2}+s} \mg}^{2-2s}_{L^2_v} \norm{\mh\jap{v}^{\frac{\gamma}{2}+s}}_{L^2_v}.$$
\emph{Case 2:} $s<\frac{1}{2}$.\\
We proceed in the same way as in Case 1a) and then sum from $k=0$ to $k=\infty$ to get the required estimate.
\end{proof}
\begin{lemma}[Commutator estimate]\label{l.comm_estimates}
Let $\mf$, $\mg$ and $\mh$ be any smooth functions, then for $\gamma+2s\in(0,2]$, we have the following bounds for 
\begin{align*}
I:=\sum_{k=0}^{k=\infty} \left|\int_0^T\int_{\R^3} \int_{\R^3} \int_{\S^2}\right.&\left. [\jap{x-(t+1)v}^2Q_k(\mf,\mg)\right.\\
&-\left.Q_k(\mf,\jap{x-(t+1)v}^2\mg]\jap{x-tv}^2\mh\d \sigma\d v_*\d v\d t\right|,
\end{align*}
\begin{align*}
I&\lesssim   \int_0^{T}(1+t)^{3-\delta}\norm{\mf\jap{x-(t+1)v}^{2\delta}\jap{v}^{\gamma+2s}}_{L^\infty_xL^1_v}\norm{(1+t)^{-\frac{1+\delta}{2}}\mg\jap{x-(t+1)v}^{2}\jap{v}^{\frac{\gamma}{2}+s}}_{L^2_xL^2_v}\\
&\qquad \times\norm{(1+t)^{-\frac{1+\delta}{2}}\mh\jap{x-(t+1)v}^{2}\jap{v}^{\frac{\gamma}{2}+s}}_{L^2_xL^2_v}\d t\\
&\quad +\int_0^{T}(1+t)^{2+2\delta}\norm{\mf\jap{v}^2}_{L^\infty_xL^1_v}\norm{(1+t)^{-\frac{1+2\delta}{2}}\mh\jap{x-(t+1)v}^2\jap{v}}_{L^2_xL^2_v}\d t\\
&\qquad \times\norm{(1+t)^{-\frac{1+2\delta}{2}}\mg\jap{x-(t+1)v}\jap{v}}_{L^2_xH^{1}_v}^{2s-1}\norm{(1+t)^{-\frac{1+2\delta}{2}}\mg\jap{x-(t+1)v}\jap{v}}_{L^2_xL^2_v}^{2-2s}\\
&\quad +\int_0^{T}(1+t)^{2+2\delta}\norm{\mf\jap{v}^2}_{L^\infty_xL^1_v}\norm{(1+t)^{-\frac{1+2\delta}{2}}\mh\jap{x-(t+1)v}^2\jap{v}}_{L^2_xL^2_v}\d t\\
&\qquad \times\norm{(1+t)^{-\frac{1+2\delta}{2}}\mg\jap{x-(t+1)v}\jap{v}}_{L^2_xH^{1}_v}^{(2s-1)_+}\norm{(1+t)^{-\frac{1+2\delta}{2}}\mg\jap{x-(t+1)v}\jap{v}}_{L^2_xL^2_v}^{(2-2s)_-}\\
&\quad +\int_0^{T}(1+t)^{2+\delta}\norm{\mf\jap{x-(t+1)v}\jap{v}^{\gamma+2s}}_{L^\infty_xL^1_v}\norm{(1+t)^{-\frac{1+\delta}{2}}\mg\jap{x-(t+1)v}^2\jap{v}^{\frac{\gamma}{2}+s}}_{L^2_xL^2_v}\\
&\qquad \times\norm{(1+t)^{-\frac{1+\delta}{2}}\mh\jap{x-(t+1)v}^2\jap{v}^{\frac{\gamma}{2}+s}}_{L^2_xL^2_v}\d t,
\end{align*}
and
\begin{align*}
I&\lesssim   \int_0^{T}(1+t)^{3-\delta}\norm{\mf\jap{x-(t+1)v}^{2\delta}\jap{v}^{\gamma+2s}}_{L^2_xL^1_v}\norm{(1+t)^{-\frac{1+\delta}{2}}\mg\jap{x-(t+1)v}^{2}\jap{v}^{\frac{\gamma}{2}+s}}_{L^\infty_xL^2_v}\\
&\qquad \times\norm{(1+t)^{-\frac{1+\delta}{2}}\mh\jap{x-(t+1)v}^{2}\jap{v}^{\frac{\gamma}{2}+s}}_{L^2_xL^2_v}\d t\\
&\quad +\int_0^{T}(1+t)^{2+2\delta}\norm{\mf\jap{v}^2}_{L^2_xL^1_v}\norm{(1+t)^{-\frac{1+2\delta}{2}}\mh\jap{x-(t+1)v}^2\jap{v}}_{L^2_xL^2_v}\d t\\
&\qquad \times\norm{(1+t)^{-\frac{1+2\delta}{2}}\mg\jap{x-(t+1)v}\jap{v}}_{L^\infty_xH^{1}_v}^{2s-1}\norm{(1+t)^{-\frac{1+2\delta}{2}}\mg\jap{x-(t+1)v}\jap{v}}_{L^\infty_xL^2_v}^{2-2s}\\
&\quad +\int_0^{T}(1+t)^{2+2\delta}\norm{\mf\jap{v}^2}_{L^2_xL^1_v}\norm{(1+t)^{-\frac{1+2\delta}{2}}\mh\jap{x-(t+1)v}^2\jap{v}}_{L^2_xL^2_v}\d t\\
&\qquad \times\norm{(1+t)^{-\frac{1+2\delta}{2}}\mg\jap{x-(t+1)v}\jap{v}}_{L^\infty_xH^{1}_v}^{(2s-1)_+}\norm{(1+t)^{-\frac{1+2\delta}{2}}\mg\jap{x-(t+1)v}\jap{v}}_{L^\infty_xL^2_v}^{(2-2s)_-}\\
&\quad +\int_0^{T}(1+t)^{2+\delta}\norm{\mf\jap{x-(t+1)v}\jap{v}^{\gamma+2s}}_{L^2_xL^1_v}\norm{(1+t)^{-\frac{1+\delta}{2}}\mg\jap{x-(t+1)v}^2\jap{v}^{\frac{\gamma}{2}+s}}_{L^\infty_xL^2_v}\\
&\qquad \times\norm{(1+t)^{-\frac{1+\delta}{2}}\mh\jap{x-(t+1)v}^2\jap{v}^{\frac{\gamma}{2}+s}}_{L^2_xL^2_v}\d t,
\end{align*}
\end{lemma}
\begin{proof}
\emph{Case 1:} $s\geq \frac{1}{2}.$ Using pre-post collision change of variables we have
\begin{align*}
\int_{\R^3} \int_{\R^3} \int_{\S^2}& [\jap{x-(t+1)v}^2Q_k(\mf,\mg)-Q_k(\mf,\jap{x-(t+1)v}^2\mg]\jap{x-(t+1)v}^2\mh\d \sigma\d v_*\d v\\
&=\int_{\R^3} \int_{\R^3} \int_{\S^2} B_k\mf_*'\mg'[\jap{x-(t+1)v}^2-\jap{x-(t+1)v'}^2]\jap{x-(t+1)v}^2\mh\d \sigma\d v_*\d v\\
&=\int_{\R^3} \int_{\R^3} \int_{\S^2}B_k \mf_*\mg[\jap{x-(t+1)v'}^2-\jap{x-(t+1)v}^2]\jap{x-(t+1)v'}^2\mh'\d \sigma\d v_*\d v.
\end{align*}

Before we start estimating, we need to make some changes to the equation above. For brevity, we drop the integration.
\begin{align}
&B_k\mf_*\mg[\jap{x-(t+1)v'}^2-[\jap{x-(t+1)v}^2]\jap{x-(t+1)v'}^2\mh'\nonumber\\
&\quad\equiv B_k\mf_*\mg[\jap{x-(t+1)v'}-\jap{x-(t+1)v}][\jap{x-(t+1)v'}^2-\jap{x-(t+1)v}^2]\jap{x-(t+1)v'}\mh' \label{e.comm_est_weight_prod_diff}\\
&\qquad+B_k\mf_*\mg\jap{x-(t+1)v}[\jap{x-(t+1)v'}^2-\jap{x-(t+1)v}^2-(v-v')_i\part_{v_i}\jap{x-(t+1)v}^2(v')]\nonumber\\
&\hspace{10em} \times\jap{x-(t+1)v'}\mh' \label{e.comm_est_weight_diff_g_1}\\
&\qquad+B_kf_*(\mg\jap{x-(t+1)v}-\mg'\jap{x-(t+1)v'})(v-v')_i\part_{v_i}\jap{x-(t+1)v}^2(v')\jap{x-(t+1)v'}\mh' \label{e.comm_est_weight_diff_g_2}\\
&\qquad+B_k \mf_*\mg'\jap{x-(t+1)v'}\part_{v_i}\jap{x-(t+1)v}^2(v')\cdot(v-v')_i\jap{x-(t+1)v'}\mh'.\label{e.comm_est_cancel}
\end{align}

We begin by estimating \eref{comm_est_weight_diff_g_1}. We have by integral form of Taylor's theorem, 
\begin{align*}
\jap{x-(t+1)v'}^2-&\jap{x-(t+1)v}^2-\part_{v_i}\jap{x-(t+1)v}^2(v')\cdot(v-v')_i\\
&=(v-v')_i(v-v')_j\int_0^1 \part^2_{v_iv_j} \jap{x-(t+1)v}^2(\eta v+(1-\eta)v')\d \eta.
\end{align*}
Since $|\part^2_{v_iv_j}\jap{x-(t+1)v}^2|\leq (1+t)^2$ and $|v-v'|=|v-v_*|\sin \left(\frac{\theta}{2}\right)$, we get
$$|(v-v')_i(v-v')_j\part^2_{v_iv_j} \jap{x-(t+1)v}^2(\eta v+(1-\eta)v')|\leq |v-v_*|^2\theta^2(1+t)^2.$$

Now, using \eref{decay_from_diff_velocity} from \lref{top_order_comm_weights} we get $$|v-v_*|^{2\delta}\lesssim (1+t)^{-2\delta} \jap{x-(t+1)v}^{2\delta}\jap{x-(t+1)v_*}^{2\delta}.$$

The above observations imply the following bound $$|v-v_*|^2(1+t)^2\lesssim (1+t)^{2-2\delta}|v-v_*|^{2-2\delta}\jap{x-(t+1)v}^{2\delta}\jap{x-(t+1)v_*}^{2\delta}.$$
Thus we get,
\begin{align*}
|B_k\mf_*\jap{x-(t+1)v}\mg&[\jap{x-(t+1)v'}^2-\jap{x-(t+1)v}^2\\
&\qquad-\part_{v_i}\jap{x-(t+1)v}^2(v')\cdot(v-v')_i]\jap{x-(t+1)v'}\mh'|\\
&\lesssim |(1+t)^{2-2\delta}|v-v_*|^{2-2\delta}\sin^2\frac{\theta}{2}B_k \mf_*\jap{x-(t+1)v_*}^{2\delta}\\
&\qquad \times \mg\jap{x-(t+1)v}^{2} \jap{x-(t+1)v'}^2\mh'|.
\end{align*}
Now we proceed in the same way as in \lref{top_order_comm_weights} to get the result,
\begin{align*}
\sum_{k=0}^{k=\infty}|\eref{comm_est_weight_diff_g_1}|\lesssim \int_0^{T_*}&(1+t)^{3-\delta}\norm{\mf\jap{x-(t+1)v}^{2\delta}\jap{v}^{\gamma+2s}}_{L^\infty_xL^1_v}\\
&\qquad\times\norm{(1+t)^{-\frac{1+\delta}{2}}\mg\jap{x-(t+1)v}^{2}\jap{v}^{\frac{\gamma}{2}+s}}_{L^2_xL^2_v}\\
&\hspace{5em} \times\norm{(1+t)^{-\frac{1+\delta}{2}}\mh\jap{x-(t+1)v}^{2}\jap{v}^{\frac{\gamma}{2}+s}}_{L^2_xL^2_v}\d t.
\end{align*}

Due to symmetry, the contribution coming from \eref{comm_est_cancel} amounts to zero.\\
Indeed, to see this we first apply pre-post collision change of variables to get,
\begin{align*}
B_k\mf_*\mg'\jap{x-(t+1)v'}&\part_{v_i}\jap{x-(t+1)v}^2(v')\cdot(v-v')_i\jap{x-(t+1)v'}\mh'\\
&\equiv -B_k \mf'_*\mg\jap{x-(t+1)v}\part_{v_i}\jap{x-(t+1)v}^2(v)\cdot(v-v')_i\jap{x-(t+1)v}\mh.
\end{align*}
For notational convenience, let $\mg\jap{x-(t+1)v}\part_{v_i}\jap{x-(t+1)v}^2(v)\jap{x-(t+1)v}\mh=J_i$. Next using the change of variables from \lref{change_of_variables_sil}, we have that
\begin{align*}
\int_{\R^3}\int_{\R^3}\int_{\S^2} B_k \mf'_* J_i (v-v')_i \d\sigma\d v_*\d v=\int_{\R^3}J_i\int_{\R^3} \chi_{k}(|v-v'|)(v-v')_iK_{\mf}(v,v')\d v'\d v,
\end{align*}
which by \lref{subt_can} is zero.

For \eref{comm_est_weight_prod_diff}, we use $$|\jap{x-(t+1)v}-\jap{x-(t+1)v'}|\lesssim (1+t)|v-v_*|\sin \frac{\theta}{2},$$ and that $$|\jap{x-(t+1)v}^2-\jap{x-(t+1)v'}^2|\lesssim (1+t)|v-v_*|\sin \frac{\theta}{2}\jap{x-(t+1)v}\jap{x-(t+1)v'}.$$
Further we also use,
$$|v-v_*|^{2\delta}\lesssim (1+t)^{-2\delta}\jap{x-(t+1)v}^{2\delta}\jap{x-(t+1)v_*}^{2\delta}.$$
Thus we have in total,
\begin{align*}
|\jap{x-(t+1)v}&-\jap{x-(t+1)v'}|\cdot |\jap{x-(t+1)v}^2-\jap{x-(t+1)v'}^2|\\
&\lesssim (1+t)^{2-2\delta}|v-v_*|^{2-2\delta}\sin^2 \frac{\theta}{2}\jap{x-(t+1)v}^{1+2\delta}\jap{x-(t+1)v'}\jap{x-(t+1)v_*}^{2\delta}.
\end{align*}
Using above, we get the following bound
\begin{align*}
 |B_k\mf_*\mg&[\jap{x-(t+1)v'}-\jap{x-(t+1)v}][\jap{x-(t+1)v'}^2-\jap{x-(t+1)v}^2]\jap{x-(t+1)v'}\mh'|\\
& \lesssim(1+t)^{2-2\delta}|v-v_*|^{2-2\delta}\sin^2\frac{\theta}{2}B_k \mf_*\jap{x-(t+1)v_*}^{2\delta} \mg\jap{x-(t+1)v}^{1+2\delta} \jap{x-(t+1)v'}^2\mh'.
\end{align*}
As above, we get,
\begin{align*}
\sum_{k=0}^{k=\infty}|\eref{comm_est_weight_prod_diff}|\lesssim \int_0^{T_*}&(1+t)^{3-\delta}\norm{\mf\jap{x-(t+1)v}^{2\delta}\jap{v}^{\gamma+2s}}_{L^\infty_xL^1_v}\\
&\qquad\times\norm{(1+t)^{-\frac{1+\delta}{2}}\mg\jap{x-(t+1)v}^{2}\jap{v}^{\frac{\gamma}{2}+s}}_{L^2_xL^2_v}\\
&\hspace{5em} \times\norm{(1+t)^{-\frac{1+\delta}{2}}\mh\jap{x-(t+1)v}^{2}\jap{v}^{\frac{\gamma}{2}+s}}_{L^2_xL^2_v}\d t.
\end{align*}

Finally we treat \eref{comm_est_weight_diff_g_2}. For aiding with the notation we let $\mathfrak G=\mathfrak g\jap{x-(t+1)v}$.\\
We also have the following bound, $$|\part_{v_i}\jap{x-(t+1)v}^2(v')\cdot(v-v')_i|\lesssim (1+t)|v-v_*|\sin\frac{\theta}{2}\jap{x-(t+1)v'}.$$
Fix a $0<R<1$ and let $k_1$ be such that $R\in [2^{-k_1-1},2^{-k_1}]$.\\
Then we have two cases,\\
\emph{Case 1a):} $k\leq k_1$. In this case, we use the bound, 
\begin{align*}
\int_{\R^3}\int_{\R^3}\int_{\S^2}&|B_k\mf_*(\mG-\mG')(v-v')_i\part_{v_i}\jap{x-(t+1)v}^2(v')\jap{x-(t+1)v'}\mh'|\d \sigma\d v_*\d v\\
&\lesssim \int_{\R^3}\int_{\R^3}\int_{\S^2}(1+t)|v-v_*||B_k\sin\frac{\theta}{2}\mf_*\mG\jap{x-(t+1)v'}^2\mh'|\d \sigma\d v_*\d v\\
&\quad+ \int_{\R^3}\int_{\R^3}\int_{\S^2}(1+t)|v-v_*||B_k\sin\frac{\theta}{2}\mf_*\mG'\jap{x-(t+1)v'}^2\mh'|\d \sigma\d v_*\d v
\end{align*}
Now using Cauchy-Schwarz and the regular change of variables $v\to v'$ in a similar way as in \lref{top_order_comm_weights}, we get
\begin{align*}
\int_{\R^3}\int_{\R^3}\int_{\S^2}&(1+t)|v-v_*||B_k\sin\frac{\theta}{2}\mf_*\mG\jap{x-(t+1)v'}^2\mh'|\d \sigma\d v_*\d v\\
&\lesssim 2^{k(2s-1)}(1+t)\norm{\mf\jap{v}^2}_{L^1_v}\norm{\mG\jap{v}}_{L^2_v}\norm{\mh\jap{x-(t+1)v}^2\jap{v}}_{L^2_v}.
\end{align*}
We get a similar bound for the other term. Thus, in total we have,
\begin{align}
\int_0^{T_*}\int_{\R^3}\int_{\R^3}\int_{\R^3}\int_{\S^2}&|B_k\mf_*(\mG-\mG)(v-v')_i\part_{v_i}\jap{x-(t+1)v}^2(v')\\
&\qquad\times\jap{x-(t+1)v'}\mh'|\d \sigma\d v_*\d v\d x\d t\nonumber\\
&\lesssim 2^{k(2s-1)}\int_0^{T_*}(1+t)^{2+\delta}\norm{\mf\jap{v}^2}_{L^\infty_xL^1_v}\norm{(1+t)^{-\frac{1+\delta}{2}}\mG\jap{v}}_{L^2_xL^2_v}\label{e.diff_g_bound_1}\\
&\hspace{10em}\times \norm{(1+t)^{-\frac{1+\delta}{2}}\mh\jap{x-(t+1)v}^2\jap{v}}_{L^2_xL^2_v}\d t\nonumber.
\end{align}
\emph{Case 2b):} $k> k_1$. In this case in addition to the ingredients we already used above, we also use, $$|\mG-\mG|\lesssim |v-v'|\int_0^1 |\part_{v_j}\mG|(\eta v+(1-\eta)v')\d \eta,$$
which follows by the integral form of Taylor's theorem. 

Thus we have the following bound,
\begin{align*}
\int_{\R^3}\int_{\R^3}\int_{\S^2}&|B_k\mf_*(\mG-\mG)(v-v')_i\part_{v_i}\jap{x-(t+1)v}^2(v')\jap{x-(t+1)v'}\mh'|\d \sigma\d v_*\d v\\
&\lesssim (1+t)\int_0^1\int_{\R^3}\int_{\R^3}\int_{\S^2}(1+t)|v-v_*|^2\sin^2\frac{\theta}{2}|B_k\mf_*\\
&\hspace{5em}\times\part_{v_j} \mG(z)\jap{x-(t+1)v'}^2\mh'|\d \sigma\d v_*\d v\d \eta,
\end{align*}
where $z=\eta v+(1-\eta)v'$.

Using the above abound and Cauchy--Schwarz, we get
\begin{align*}
&\int_0^1\int_{\R^3}\int_{\R^3}\int_{\S^2}(1+t)B_k|\mf_*||v-v_*|^2\sin^2\frac{\theta}{2}|\part_{v_i} \mG(z)|\cdot |\jap{x-(t+1)v'}^2\mh'|\d \sigma\d v_*\d v\d \eta\\
&\lesssim (1+t)\left(\int_0^1\int_{\R^3}\int_{\R^3}\int_{\S^2}B_k|\mf_*||v-v_*|^2\sin^2\frac{\theta}{2}|\part_{v_i}  \mG(z)|^2\d \sigma\d v_*\d v\d \eta\right)^{\frac{1}{2}}\\
&\qquad \times \left(\int_0^1\int_{\R^3}\int_{\R^3}\int_{\S^2}B_k|\mf_*||v-v_*|^{2}\sin^2\frac{\theta}{2}\jap{x-(t+1)v'}^4{\mh'}^2\d \sigma\d v_*\d v\d \eta\right)^{\frac{1}{2}}.
\end{align*}
First note that,
$$\int_{\S^2}B_k|v-v_*|^2\sin^2\frac{\theta}{2}\lesssim 2^{k(2s-2)}|v-v_*|^{\gamma+2s}.$$
For the first factor in the bound above, we perform the change of variables $v\to \eta v+(1-\eta)v'$. We can prove that the jacobian for $v\to \eta v+(1-\eta)v'$ has a lower bound in the same way as we proved for $u\to \eta v_*+(1-\eta)v'_*$ in \lref{top_order_Gaussian_2} (in particular the lower bound is independent of $v_*$, $\sigma$ and $\eta\in[0,1]$).\\
Next we get a lower bound for $|z-v_*|$. Since $\theta\leq \frac{\pi}{2}$, we have
\begin{align}
|z-v_*|&= \left|\frac{1+\eta}{2}(v-v_*)+\frac{1-\eta}{2}|v-v_*|\sigma\right|\notag\\
&=|v-v_*|\left|\left(\frac{1+\eta}{2}\right)^2+\left(\frac{1-\eta}{2}\right)^2+\frac{1-\eta^2}{2}\kappa\cdot \sigma\right|\notag\\
&=|v-v_*|\left|\eta^2+(1-\eta^2)\cos^2 \frac{\theta}{2}\right|\notag\\
&\geq c|v-v_*|\label{e.z_lower_bound}.
\end{align}
Putting these observations together, we get that
\begin{align*}
\int_0^1\int_{\R^3}\int_{\R^3}\int_{\S^2}&B_k|\mf_*||v-v_*|^2\sin^2\frac{\theta}{2}|\part_{v_i}  \mG(z)|^2\d \sigma\d v_*\d v\d \eta\\
&\lesssim 2^{k(2s-2)}\norm{\mf\jap{v}^{\gamma+2s}}_{L^1_v}\norm{\part_{v_i} (\mG)\jap{v}^{\frac{\gamma}{2}+s}}_{L^2_v}^2.
\end{align*}

Similarly, we can prove the following bound for the second factor using the change of variables, $v'\to v$,
\begin{align*}
\int_0^1\int_{\R^3}\int_{\R^3}\int_{\S^2}&B_k|\mf_*||v-v_*|^2\sin^2\frac{\theta}{2}|\jap{x-(t+1)v'}^4{\mh'}^2\d \sigma\d v_*\d v\d \eta\\
&\lesssim 2^{k(2s-2)}\norm{\mf\jap{v}^{\gamma+2s}}_{L^1_v}\norm{\mh\jap{x-(t+1)v}^2 \jap{v}^{\frac{\gamma}{2}+s}}_{L^2_v}^2.
\end{align*}
Hence we have the bound,
\begin{align}
\int_0^{T}\int_{\R^3}\int_{\R^3}\int_{\R^3}&\int_{\S^2}|B_k\mf_*(\mG-\mG)(v-v')_i\part_{v_i}\jap{x-(t+1)v}^2(v')\\
&\qquad\times\jap{x-(t+1)v'}\mh'|\d \sigma\d v_*\d v\d x\d t\nonumber\\
&\lesssim 2^{k(2s-2)}\int_0^{T}(1+t)^{2+2\delta}\norm{\mf\jap{v}^2}_{L^\infty_xL^1_v}\norm{(1+t)^{-\frac{1+2\delta}{2}}\part_{v_i}\mG\jap{v}}_{L^2_xL^2_v}\label{e.diff_g_bound_2}\\
&\hspace{10em}\times\norm{(1+t)^{-\frac{1+2\delta}{2}}\mh\jap{x-(t+1)v}^2\jap{v}}_{L^2_xL^2_v}\d t\nonumber.
\end{align}

With these bounds in hand, we now estimate,
 $$\sum_{k=0}^{k=\infty}|\eref{comm_est_weight_diff_g_2}|=\sum_{k=0}^{k=k_1}|\eref{comm_est_weight_diff_g_2}|+\sum_{k=k_1+1}^{k=\infty}|\eref{comm_est_weight_diff_g_2}|.$$
For the first term, we use \eref{diff_g_bound_1} to get
\begin{align*}
\sum_{k=0}^{k=k_1}|\eref{comm_est_weight_diff_g_2}|&\lesssim \sum_{k=0}^{k=k_1}2^{k(2s-1)}\int_0^{T}(1+t)^{2+2\delta}\norm{\mf\jap{v}^2}_{L^\infty_xL^1_v}\norm{(1+t)^{-\frac{1+2\delta}{2}}\mG\jap{v}}_{L^2_xL^2_v}\\
&\qquad\times \norm{(1+t)^{-\frac{1+2\delta}{2}}\mh\jap{x-(t+1)v}^2\jap{v}}_{L^2_xL^2_v}\d t\\
&\lesssim 2^{k_1(2s-1)}\int_0^{T}(1+t)^{2+2\delta}\norm{\mf\jap{v}^2}_{L^\infty_xL^1_v}\norm{(1+t)^{-\frac{1+2\delta}{2}}\mG\jap{v}}_{L^2_xL^2_v}\\
&\qquad\times \norm{(1+t)^{-\frac{1+2\delta}{2}}\mh\jap{x-(t+1)v}^2\jap{v}}_{L^2_xL^2_v}\d t\\
&\lesssim R^{1-2s}\int_0^{T}(1+t)^{2+2\delta}\norm{\mf\jap{v}^2}_{L^\infty_xL^1_v}\norm{(1+t)^{-\frac{1+2\delta}{2}}\mG\jap{v}}_{L^2_xL^2_v}\\
&\qquad\times \norm{(1+t)^{-\frac{1+2\delta}{2}}\mh\jap{x-(t+1)v}^2\jap{v}}_{L^2_xL^2_v}\d t,
\end{align*}
Here we used that for $s>\frac{1}{2}$. For $s=\frac{1}{2}$ we can replace $R^{1-2s}$ by $R^{-\eta}$ for any $\eta>0$.

For the second term, we use \eref{diff_g_bound_2} to get
\begin{align*}
\sum_{k=k_1+1}^{k=\infty}|\eref{comm_est_weight_diff_g_2}|&\lesssim \sum_{k=k_1+1}^{k=\infty}2^{k(2s-2)}\int_0^{T}(1+t)^{2+2\delta}\norm{\mf\jap{v}^2}_{L^\infty_xL^1_v}\norm{(1+t)^{-\frac{1+2\delta}{2}}\part_{v_i}\mG\jap{v}}_{L^2_xL^2_v}\\
&\qquad\times \norm{(1+t)^{-\frac{1+2\delta}{2}}\mh\jap{x-(t+1)v}^2\jap{v}}_{L^2_xL^2_v}\d t\\
&\lesssim 2^{k_1(2s-2)}\int_0^{T}(1+t)^{2+2\delta}\norm{\mf\jap{v}^2}_{L^\infty_xL^1_v}\norm{(1+t)^{-\frac{1+2\delta}{2}}\part_{v_i}\mG\jap{v}}_{L^2_xL^2_v}\\
&\qquad\times \norm{(1+t)^{-\frac{1+2\delta}{2}}\mh\jap{x-(t+1)v}^2\jap{v}}_{L^2_xL^2_v}\d t\\
&\lesssim R^{2-2s}\int_0^{T}(1+t)^{2+2\delta}\norm{\mf\jap{v}^2}_{L^\infty_xL^1_v}\norm{(1+t)^{-\frac{1+2\delta}{2}}\part_{v_i}\mG\jap{v}}_{L^2_xL^2_v}\\
&\qquad\times \norm{(1+t)^{-\frac{1+2\delta}{2}}\mh\jap{x-(t+1)v}^2\jap{v}}_{L^2_xL^2_v}\d t,
\end{align*}

Now let $$R=\frac{\norm{(1+t)^{-\frac{1+2\delta}{2}}\jap{v} \mG}_{L^2_v}}{\norm{(1+t)^{-\frac{1+2\delta}{2}}\jap{v} \mG}_{H^1_v}}.$$
Thus,
\begin{align*}
\sum_{k=0}^{k=k_1}|\eref{comm_est_weight_diff_g_2}|&\lesssim \int_0^{T}(1+t)^{2+2\delta}\norm{\mf\jap{v}^2}_{L^\infty_xL^1_v}\norm{(1+t)^{-\frac{1+2\delta}{2}}\mh\jap{x-(t+1)v}^2\jap{v}}_{L^2_xL^2_v}\\
&\qquad\times \norm{(1+t)^{-\frac{1+2\delta}{2}}\mG\jap{v}}^{\max\{2s-1,\eta\}}_{L^2_xH^1_v}\norm{(1+t)^{-\frac{1+2\delta}{2}}\mG\jap{v}}^{\min\{2-2s,1-\eta\}}_{L^2_xL^2_v}\d t
\end{align*}
for any $\eta>0$ and
\begin{align*}
\sum_{k=k_1}^{k=\infty}|\eref{comm_est_weight_diff_g_2}|&\lesssim \int_0^{T}(1+t)^{2+2\delta}\norm{\mf\jap{v}^2}_{L^\infty_xL^1_v}\norm{(1+t)^{-\frac{1+2\delta}{2}}\mh\jap{x-(t+1)v}^2\jap{v}}_{L^2_xL^2_v}\\
&\qquad\times \norm{(1+t)^{-\frac{1+2\delta}{2}}\mG\jap{v}}^{2s-1}_{L^2_xH^1_v}\norm{(1+t)^{-\frac{1+2\delta}{2}}\mG\jap{v}}^{2-2s}_{L^2_xL^2_v}\d t.
\end{align*}
Remembering that $G=\jap{x-(t+1)v}g$, we have the bound, 
\begin{align*}
\sum_{k=0}^{k=\infty}|\eref{comm_est_weight_diff_g_2}|&\lesssim \int_0^{T}(1+t)^{2+2\delta}\norm{\mf\jap{v}^2}_{L^\infty_xL^1_v}\norm{(1+t)^{-\frac{1+2\delta}{2}}\mh\jap{x-(t+1)v}^2\jap{v}}_{L^2_xL^2_v}\\
&\qquad \times\norm{(1+t)^{-\frac{1+2\delta}{2}}\mg\jap{x-(t+1)v}\jap{v}}^{2s-1}_{L^2_xH^1_v}\norm{(1+t)^{-\frac{1+2\delta}{2}}\mg\jap{x-(t+1)v}\jap{v}}^{2-2s}_{L^2_xL^2_v}\d t\\
&+\int_0^{T}(1+t)^{2+2\delta}\norm{\mf\jap{v}^2}_{L^\infty_xL^1_v}\norm{(1+t)^{-\frac{1+2\delta}{2}}\mh\jap{x-(t+1)v}^2\jap{v}}_{L^2_xL^2_v}\\
&\qquad \times\norm{(1+t)^{-\frac{1+2\delta}{2}}\mg\jap{x-(t+1)v}\jap{v}}^{(2s-1)_+}_{L^2_xH^1_v}\norm{(1+t)^{-\frac{1+2\delta}{2}}\mg\jap{x-(t+1)v}\jap{v}}^{(2-2s)_-}_{L^2_xL^2_v}\d t.
\end{align*}

\emph{Case 2:} $s<\frac{1}{2}$. For this case we work directly with 
$$\int_{\R^3} \int_{\R^3} \int_{\S^2}B_k \mf_*\mg[\jap{x-(t+1)v'}^2-[\jap{x-(t+1)v}^2]\jap{x-(t+1)v'}^2\mh'\d \sigma\d v_*\d v.$$

We again use the bound, $$|\jap{x-(t+1)v'}^2-\jap{x-(t+1)v}^2|\lesssim (1+t)|v-v'|(\jap{x-(t+1)v}+\jap{x-(t+1)v'}),$$
and the bound $$\jap{x-(t+1)v'}\lesssim \jap{x-(t+1)v}+\jap{x-(t+1)v_*}.$$

Putting this together we get the bound, 
\begin{align*}
\int_{\R^3} \int_{\R^3}& \int_{\S^2}|B_k \mf_*\mg[\jap{x-(t+1)v'}^2-[\jap{x-(t+1)v}^2]\jap{x-(t+1)v'}^2\mh'|\d \sigma\d v_*\d v\\
&\lesssim 2^{k(2s-1)}(1+t)^{2+\delta}\norm{\mf\jap{x-(t+1)v}\jap{v}^{\gamma+2s}}_{L^1_v}\norm{(1+t)^{-\frac{1+\delta}{2}}\mg\jap{x-(t+1)v}^2\jap{v}^{\frac{\gamma}{2}+s}}_{L^2_v}\\
&\qquad \times\norm{(1+t)^{-\frac{1+\delta}{2}}\mh\jap{x-(t+1)v}^2\jap{v}^{\frac{\gamma}{2}+s}}_{L^2_v}.
\end{align*}
Summing over positive $k$ followed by a Cauchy--Schwarz in space implies the required bound.
\end{proof}
\begin{lemma}\label{l.comm_gaussian}
Let $\mf$, $\mg$ and $\mh$ from $[0,T)\times\R^3\times\R^3\to \R$ be any smooth functions then for $\gamma+2s\in(0,2]$ we have the following estimates,
\begin{align*}
\sum_{k=0}^{k=\infty}&\left|\int_0^{T}\int_{\R^3}\int_{\R^3}\int_{\R^3}\int_{\S^2}B_k(\mu_{\beta,\omega}(v'_*)-\mu_{\beta,\omega}(v_*))\mf_*'\mg'\jap{x-(t+1)v}^4\mh\d\sigma\d v_*\d v\d x\d t\right|\\
&\lesssim \int_0^{T}(1+t)^{1+\delta}\norm{\mf\jap{x-(t+1)v}^2}_{L^\infty_xL^2_v}\norm{(1+t)^{-\frac{1+\delta}{2}}\jap{x-(t+1)v}^2\jap{v}\mg}_{L^2_xL^2_v}\\
&\quad\hspace{7em}\times\norm{(1+t)^{-\frac{1+\delta}{2}}\jap{x-(t+1)v}^2\jap{v}\mh}_{L^2_xL^2_v}\d t\\
&\quad+\int_0^{T}(1+t)^{1+\delta}\norm{\mf}_{L^\infty_xL^1_v}\norm{(1+t)^{-\frac{1+\delta}{2}}\jap{x-(t+1)v}^2\jap{v}\part_{v_i} \mg}_{L^2_xL^2_v}\\
&\quad \hspace{7em}\times\norm{(1+t)^{-\frac{1+\delta}{2}}\jap{x-(t+1)v}^2\jap{v}\mh}_{L^2_xL^2_v}\d t\\
&\quad+ \int_0^{T}(1+t)^{2+\delta}\norm{\mf}_{L^\infty_xL^1_v}\norm{(1+t)^{-\frac{1+\delta}{2}}\jap{x-(t+1)v}^2\jap{v} \mg}_{L^2_xL^2_v}\\
&\quad \hspace{7em}\times\norm{(1+t)^{-\frac{1+\delta}{2}}\jap{x-(t+1)v}^2\jap{v}\mh}_{L^2_xL^2_v}\d t,
\end{align*}
and
\begin{align*}
\sum_{k=0}^{k=\infty}&\left|\int_0^{T}\int_{\R^3}\int_{\R^3}\int_{\R^3}\int_{\S^2}B_k(\mu_{\beta,\omega}(v'_*)-\mu_{\beta,\omega}(v_*))\mf_*'\mg'\jap{x-(t+1)v}^4\mh\d\sigma\d v_*\d v\d x\d t\right|\\
&\lesssim \int_0^{T}(1+t)^{1+\delta}\norm{\mf\jap{x-(t+1)v}^2}_{L^2_xL^2_v}\norm{(1+t)^{-\frac{1+\delta}{2}}\jap{x-(t+1)v}^2\jap{v}\mg}_{L^\infty_xL^2_v}\\
&\quad\hspace{7em}\times\norm{(1+t)^{-\frac{1+\delta}{2}}\jap{x-(t+1)v}^2\jap{v}\mh}_{L^2_xL^2_v}\d t\\
&\quad+\int_0^{T}(1+t)^{1+\delta}\norm{\mf}_{L^2_xL^1_v}\norm{(1+t)^{-\frac{1+\delta}{2}}\jap{x-(t+1)v}^2\jap{v}\part_{v_i} \mg}_{L^\infty_xL^2_v}\\
&\quad \hspace{7em}\times\norm{(1+t)^{-\frac{1+\delta}{2}}\jap{x-(t+1)v}^2\jap{v}\mh}_{L^2_xL^2_v}\d t\\
&\quad+ \int_0^{T}(1+t)^{2+\delta}\norm{\mf}_{L^2_xL^1_v}\norm{(1+t)^{-\frac{1+\delta}{2}}\jap{x-(t+1)v}^2\jap{v} \mg}_{L^\infty_xL^2_v}\\
&\quad \hspace{7em}\times\norm{(1+t)^{-\frac{1+\delta}{2}}\jap{x-(t+1)v}^2\jap{v}\mh}_{L^2_xL^2_v}\d t.
\end{align*}
\end{lemma}
\begin{proof}
We treat the two cases, $s\geq \frac{1}{2}$ and $s<\frac{1}{2}$, differently.\\
\emph{Case 1:} $s\geq \frac{1}{2}$. To be able to get the required estimate, we make further changes to the term at hand,
\begin{align}
B_k(\mu_{\beta,\omega}&(v'_*)-\mu_{\beta,\omega}(v_*))\mf_*'\mg'\jap{x-(t+1)v}^4\mh\label{e.comm_gaussian_less_than_half}\\
&=B_k[\mu_{\beta,\omega}(v'_*)-\mu_{\beta,\omega}(v_*)]\mf_*'\mg'[\jap{x-(t+1)v}^2-\jap{x-(t+1)v'}^2]\jap{x-(t+1)v}^2\mh\label{e.comm_gaussian_1}\\
&\quad+B_k[\mu_{\beta'\omega}(v'_*)-\mu_{\beta,\omega}(v_*)]\mf_*'[\jap{x-(t+1)v'}^2\mg'-\jap{x-(t+1)v}^2\mg]\jap{x-(t+1)v}^2\mh\label{e.comm_gaussian_2}\\
&\quad+ B_k[\mu_{\beta,\omega}(v'_*)-\mu_{\beta,\omega}(v_*)-(v-v')_i\part_{v_i}(\mu_{\beta,\omega})(v_*)]\mf_*'\jap{x-(t+1)v}^2\mg\jap{x-(t+1)v}^2\mh\label{e.comm_gaussian_3}\\
&\quad+  B_k(v-v')_i\part_{v_i}(\mu_{\beta,\omega})(v_*)\mf_*'\jap{x-(t+1)v}^2\mg\jap{x-(t+1)v}^2\mh\label{e.comm_gaussian_4}.
\end{align}

First note that $$|\mu_{\beta,\omega}(v'_*)-\mu_{\beta,\omega}(v_*)|\lesssim |v-v'|\part_{v_i}\mu_{\beta,\omega}(\eta v_*+(1-\eta)v_*')\lesssim |v-v_*|\sin\frac{\theta}{2}(\mu_*\mu_*')^\tau,$$
where we used the fact that $\part_{v_i}\mu_{\beta,\omega}=\part_{v_i}\part^{|\beta|+|\omega|}_{v} \mu\lesssim_{|\beta|,|\omega|} \mu^{\frac{1}{2}}$.

Now we can bound \eref{comm_gaussian_1} in the same way as in \lref{top_order_Gaussian_1} to get,
\begin{align*}
\sum_{k=0}^{k=\infty}|\eref{comm_gaussian_1}|\lesssim \int_0^{T}(1+t)^{1+\delta}&\norm{\mf\jap{x-(t+1)v}^2}_{L^\infty_xL^2_v}\norm{(1+t)^{-\frac{1+\delta}{2}}\jap{x-(t+1)v}^2\jap{v}\mg}_{L^2_xL^2_v}\\
&\quad\times\norm{(1+t)^{-\frac{1+\delta}{2}}\jap{x-(t+1)v}^2\jap{v}\mh}_{L^2_xL^2_v}\d t,
\end{align*}
where we also absorbed the extra $\jap{v_*}^{\gamma+2s}$ by $\mu_*^\tau$ in the norm for $\mf$.

For \eref{comm_gaussian_2}, we first let $\mG=\jap{x-(t+1)v}^2\mg$ and then use Taylor's integral formula to get,
$$|\mG'-\mG|\lesssim |v-v'|\int_0^1 \mG(\eta v+(1-\eta)v')\d \eta.$$
Next we proceed in a similar way as \lref{top_order_Gaussian_2} but use the change of variables\\ $z\to \eta v+(1-\eta)v'$ from \lref{comm_estimates} to get
\begin{align*}
\sum_{k=0}^{k=\infty}|\eref{comm_gaussian_1}|\lesssim &\int_0^{T}(1+t)^{1+2\delta}\norm{\mf}_{L^\infty_xL^1_v}\norm{(1+t)^{-\frac{1+2\delta}{2}}\jap{x-(t+1)v}^2\jap{v}\part_{v_i} \mg}_{L^2_xL^2_v}\\
&\hspace{5em}\times\norm{(1+t)^{-\frac{1+2\delta}{2}}\jap{x-(t+1)v}^2\jap{v}\mh}_{L^2_xL^2_v}\d t\\
&\quad+ \int_0^{T}(1+t)^{2+\delta}\norm{\mf}_{L^\infty_xL^1_v}\norm{(1+t)^{-\frac{1+\delta}{2}}\jap{x-(t+1)v}^2\jap{v} \mg}_{L^2_xL^2_v}\\
&\hspace{5em}\times\norm{(1+t)^{-\frac{1+\delta}{2}}\jap{x-(t+1)v}^2\jap{v}\mh}_{L^2_xL^2_v}\d t.
\end{align*}

For \eref{comm_gaussian_3} we use Taylor's theorem and the fact that $|v-v'|\lesssim 1$ to get 
$$|\mu_{\beta,\omega}(v'_*)-\mu_{\beta,\omega}(v_*)-(v-v')_i\part_{v_i}(\mu_{\beta,\omega})(v_*)|\lesssim|v-v'|^2 (\mu_*\mu_*')^{\tau'}.$$
Now we can easily get the bound,
\begin{align*}
\sum_{k=0}^{k=\infty}|\eref{comm_gaussian_3}|\lesssim \int_0^{T}(1+t)^{1+\delta}&\norm{\mf}_{L^\infty_xL^1_v}\norm{(1+t)^{-\frac{1+\delta}{2}}\jap{x-(t+1)v}^2\jap{v}\mg}_{L^2_xL^2_v}\\
&\quad\times\norm{(1+t)^{-\frac{1+\delta}{2}}\jap{x-(t+1)v}^2\jap{v}\mh}_{L^2_xL^2_v}\d t.
\end{align*}

Finally in the same way as in \lref{comm_estimates}, we have that $\eref{comm_gaussian_4}=0$.\\
\emph{Case 2:} $s<\frac{1}{2}$. In this case we just use, 
$$|\mu_{\beta,\omega}(v'_*)-\mu_{\beta,\omega}(v_*)|\lesssim |v-v'|\part_{v_i}\mu_{\beta,\omega}(\eta v_*+(1-\eta)v_*')\lesssim |v-v_*|\sin\frac{\theta}{2}(\mu_*\mu_*')^\tau,$$
which is enough to overcome the singularity in $\theta$ as $s<\frac{1}{2}$.\\
Now proceeding as in \lref{top_order_Gaussian_s_less_than_half}, we can get the bound,
\begin{align*}
\sum_{k=0}^{k=\infty}|\eref{comm_gaussian_less_than_half}|\lesssim \int_0^{T}(1+t)^{1+\delta}&\norm{\mf\jap{x-(t+1)v}^2}_{L^\infty_xL^2_v}\norm{(1+t)^{-\frac{1+\delta}{2}}\jap{x-(t+1)v}^2\jap{v}\mg}_{L^2_xL^2_v}\\
&\quad\times\norm{(1+t)^{-\frac{1+\delta}{2}}\jap{x-(t+1)v}^2\jap{v}\mh}_{L^2_xL^2_v}\d t.
\end{align*}
\end{proof}

The final estimate we need is stated in the following lemma, which has been proved in various forms in \cite{GrSt11}, \cite{ImbSil16} and \cite{AMUXY10}. The proof for our setting can be inferred directly from \cite{GrSt11} or by making necessary changes as in \lref{diff_f_diff_G} to the proof in \cite{ImbSil16}. Thus we state the lemma without proof\footnote{The function $\mf$ must crucially be estimated in $L^1_v$  to get appropriate time decay. This is in contrast to \cite{GrSt11} where they estimate it in $L^2_v$.}.
\begin{lemma}\label{l.derv_dist}
For $\gamma+2s\in(0,2]$ and for $\mf$, $\mg$ and $\mh$ smooth functions, we have the following estimates,
\begin{equation}\label{e.H2s_1} 
\begin{split}
\sum_{k=0}^{k=\infty}\left|\int_{\R^3}Q_{k}(\mf,\mg)\mh\d v\right|&\lesssim\norm{\mf\jap{v}^{\gamma+2s}}_{L^1_v}\norm{\mg\jap{v}^{\frac{\gamma}{2}+s}}_{H^{2s}_v}\norm{\mh\jap{v}^{\frac{\gamma}{2}+s}}_{L^2_v}
\end{split}
\end{equation}
and 
\begin{equation}\label{e.Hs} 
\begin{split}
\sum_{k=0}^{k=\infty}\left|\int_{\R^3}Q_{k}(\mf,\mg)\mh\d v\right|&\lesssim\norm{\mf\jap{v}^{\gamma+2s}}_{L^1_v}\norm{\mg\jap{v}^{\frac{\gamma}{2}+s}}_{H^{s}_v}\norm{\mh\jap{v}^{\frac{\gamma}{2}+s}}_{H^{s}_v}.
\end{split}
\end{equation}
\end{lemma}
\begin{lemma}\label{l.pen_est}
For $\gamma+2s\in(0,2]$ and $|\alpha''|+|\beta''|+|\omega''|=|\alpha|+|\beta|+|\omega|-1=9$ (and hence $|\alpha'|+|\beta'|+|\omega'|+|\alpha'''|+|\beta'''|+|\omega'''|= 1$), we have the following estimate,
\begin{align*}
\int_0^{T}\int_{\R^3}\int_{\R^3}&\Gamma_{\beta''',\omega'''}(\derv{'}{'}{'} g,\derv{'}{'}{'} g)\jap{x-(t+1)v}^4\der g\d v\d x\d t\\
&\lesssim \mathbb{I}+\mathbb{II}+\mathbb{III}+\mathbb{IV}+\color{black}{\mathbb{VI}}+\mathbb{VII}+\mathbb{VIII}+\mathbb{IX}+\mathbb{X}+\mathbb{XI},
\end{align*}
where the terms on the right side are the same as in \lref{main_boltz_estimates}.
\end{lemma}
\begin{proof}
As before we use the singularity decomposition from \sref{sing}. For $k<0$, we use the first bound from \lref{trivial_est_1} and \lref{trivial_est_2} with $\mf=\derv{'}{'}{'} g$, $\mg=\derv{''}{''}{''} g$ and $\mh=\der g$. Thus we get,
\begin{align*}
\int_0^{T}&\int_{\R^3}\int_{\R^3}|\Gamma^k_{\beta''',\omega'''}(\derv{'}{'}{'} g,\derv{'}{'}{'} g)\jap{x-(t+1)v}^4\der g|\d v\d x\d t\\
&\lesssim 2^{2sk}\int_0^{T_*}(1+t)^{1+\delta}\norm{\jap{x-(t+1)v}^2\derv{'}{'}{'}g}_{L^\infty_xL^2_v}\\
&\qquad\times\norm{(1+t)^{-\frac{1}{2}-\frac{\delta}{2}}\jap{x-(t+1)v}^2\jap{v}\derv{''}{''}{''} g}_{L^2_xL^2_v}\\
&\hspace{5em} \times \norm{(1+t)^{-\frac{1}{2}-\frac{\delta}{2}}\jap{x-(t+1)v}^2\jap{v}\der g}_{L^2_xL^2_v}\d t.
\end{align*}

Now we can sum over $k<0$ and bound it by $\mathbb{I}$.

For $k\geq 0$ we have to bound the infinite sums of \eref{pen_term_main} and \eref{Gaussian_diff_term}.\\
First, for $\sum_{k=0}^{k=\infty}|\eref{Gaussian_diff_term}|$, we use the first bound from \lref{comm_gaussian} with $\mf=\derv{'}{'}{'} g$, $\mg=\derv{''}{''}{''} g$ and $\mh=\der g$. We thus get,
$$\sum_{k=0}^{k=\infty}|\eref{Gaussian_diff_term}|\lesssim \mathbb{I}+\mathbb{III}+\mathbb{XI}.$$

For \eref{pen_term_main}, we use that, $|\eref{pen_term_main}|\lesssim |\eref{pen_main}|+|\eref{pen_term_derv_weight}|+|\eref{pen_comm_term}|$.\\
For $\sum_{k=0}^{k=\infty}|\eref{pen_term_derv_weight}|$, we use \eref{Hs} from \lref{derv_dist} with $\mf=\part_x^{\bar \alpha}\part_v^{\bar \beta}Y^{\bar \omega} f$ (where $\part_x^{\bar \alpha}\part_v^{\bar \beta}Y^{\bar \omega}=\derv{'}{'}{'}\part_v^{\beta'''}Y^{\omega'''}$), $\mg=\mh=\jap{x-(t+1)v}^2\derv{''}{''}{''} g$ and then use \lref{exp_bound} to get the bound,
$$\sum_{k=0}^{k=\infty} |\eref{pen_term_derv_weight}|\lesssim \mathbb{VIII}.$$

Next we bound, $\sum_{k=0}^{k=\infty} |\eref{pen_comm_term}|$. We use \lref{comm_estimates} with $\mf=\part_x^{\bar \alpha}\part_v^{\bar \beta}Y^{\bar \omega} f$, $\mg=\derv{''}{''}{''} g$ and $\mh=\der g$ to get the bound,
$$\sum_{k=0}^{k=\infty} |\eref{pen_comm_term}|\lesssim \mathbb{IV}+\mathbb{X}+\mathbb{III}.$$

Finally we have, 
\begin{align*}
|\eref{pen_main}|&\lesssim \left|\int_{0}^{T} \int_{\R^3}\int_{\R^3}Q_{1,k}(\part_x^{\bar \alpha}\part_v^{\bar \beta}Y^{\bar \omega} f, \bar G)\derv{'}{'}{'} \bar G\d v\d x\d t\right|\\
&\quad+\left|\int_{0}^{T} \int_{\R^3}\int_{\R^3}Q_{2,k}(\part_x^{\bar \alpha}\part_v^{\bar \beta}Y^{\bar \omega} f, \bar G)\derv{'}{'}{'} \bar G\d v\d x\d t\right|
\end{align*}
For the second term we use \lref{canc_lemma} with $\mf=\part_x^{\bar \alpha}\part_v^{\bar \beta}Y^{\bar \omega}f$,\\
 $\mg=\jap{x-(t+1)v}^2\derv{''}{''}{''}g$ and $\mh=\derv{'}{'}{'}(\jap{x-(t+1)v}^2\derv{''}{''}{''} g)$ and use \lref{exp_bound} to get,
$$\left|\int_{0}^{T} \int_{\R^3}\int_{\R^3}Q_{2,k}(\derv{'}{'}{'} f, \bar G)\derv{'}{'}{'} \bar G\d v\d x\d t\right|\lesssim \mathbb{II}+\mathbb{IX}.$$
For the first term we use \pref{pen_term_main} and thus it suffices to bound the infinite sums of \eref{pen_term_main_1} and \eref{pen_term_main_2}.

To bound $\sum_{k=0}^{k=\infty} |\eref{pen_term_main_1}|$, we use \lref{sym_bound_weight} with $\mf=\part_x^{\bar \alpha}\part_v^{\bar \beta}Y^{\bar \omega} f$ and $\mg=\jap{x-(t+1)v}^2\derv{''}{''}{''} g$. We thus have,
$$\sum_{k=0}^{k=\infty} |\eref{pen_term_main_1}|\lesssim \mathbb{VIII}.$$

Finally, to bound $\sum_{k=0}^{k=\infty} |\eref{pen_term_main_2}|$, we use \lref{diff_f_diff_G} with $f=\part_x^{\bar \alpha}\part_v^{\bar \beta}Y^{\bar \omega} f$ and\\
 $g=\jap{x-(t+1)v}^2\derv{''}{''}{''} g$ and $h=\jap{x-(t+1)v}^2\der g$ to get the following estimate,
$$\sum_{k=0}^{k=\infty} |\eref{pen_term_main_2}|\lesssim \color{black}{\mathbb{VI}}+\mathbb{VII}.$$
\end{proof}
\subsection{Lower orders terms} This corresponds to $|\alpha''|+|\beta''|+|\omega''|\leq|\alpha|+|\beta|+|\omega|-2$.

We again use the singularity decomposition and for $k<0$, we directly bound, 
\begin{align*}
\int_{0}^{T}\int_{\R^3}\int_{\R^3}\int_{\R^3}\int_{\S^2}&B_k \mu_{\beta''',\omega'''}(v_*)((\derv{'}{'}{'}g)_*'(\derv{''}{''}{''} g)'-(\derv{'}{'}{'}g)_*\derv{''}{''}{''} g)\\
&\quad \times\jap{x-(t+1)v}^4\der g\d \sigma\d v_*\d v\d x\d t 
\end{align*}
using \lref{trivial_est_1} and \lref{trivial_est_2}.\\
For $k\geq 0$, we need to make some changes to be able to treat the singularity
\begin{align}
B_k&\mu_{\beta''',\omega'''}(v_*)((\derv{'}{'}{'}g)_*'(\derv{''}{''}{''} g)'\nonumber\\
&\hspace{6em}-(\derv{'}{'}{'}g)_*\derv{''}{''}{''} g)\jap{x-(t+1)v}^4\der g\nonumber\\
&\equiv Q_k(\mu_{\beta''',\omega'''}\derv{'}{'}{'} g,\derv{''}{''}{''}g)\jap{x-(t+1)v}^4\der g\nonumber\\
&\quad +B_k(\mu_{\beta''',\omega'''}(v'_*)-\mu_{\beta''',\omega'''}(v_*))(\derv{'}{'}{'} g)_*'(\derv{''}{''}{''}g)'\jap{x-(t+1)v}^4\der g\label{e.low_Gaussian_diff_term}\\
&\equiv Q_k(\mu_{\beta''',\omega'''}\derv{'}{'}{'} g,\jap{x-(t+1)v}^2\derv{''}{''}{''}g)\jap{x-(t+1)v}^2\der g\label{e.low_main}\\
&\quad+ [Q_k(\mu_{\beta''',\omega'''}\derv{'}{'}{'} g,\derv{''}{''}{''}g)\jap{x-(t+1)v}^4\der g\nonumber\\
&\qquad-Q_k(\mu_{\beta''',\omega'''}\derv{'}{'}{'} g,\jap{x-(t+1)v}^2\derv{''}{''}{''}g)\jap{x-(t+1)v}^2\der g]\label{e.low_comm}\\
&\quad +B_k(\mu_{\beta''',\omega'''}(v'_*)-\mu_{\beta''',\omega'''}(v_*))(\derv{'}{'}{'} g)_*'(\derv{''}{''}{''}g)'\jap{x-(t+1)v}^4\der g\nonumber.
\end{align}
\begin{lemma}\label{l.low_est}
For $\gamma+2s\in(0,2]$ and $|\alpha''|+|\beta''|+|\omega''|\leq 8$ (and hence $|\alpha'|+|\beta'|+|\omega'|+|\alpha'''|+|\beta'''|+|\omega'''|\geq 2$), we have the following estimate,
\begin{align*}
\int_0^{T}\int_{\R^3}\int_{\R^3}&\Gamma_{\beta''',\omega'''}(\derv{'}{'}{'} g,\derv{'}{'}{'} g)\jap{x-(t+1)v}^4\der g\d v\d x\d t\\
&\lesssim \mathbb{I}+\mathbb{III}+\mathbb{IV}+\mathbb{X}+\mathbb{XI}+\mathbb{XII},
\end{align*}
where the terms on the right side are the same as in \lref{main_boltz_estimates}.
\end{lemma}
\begin{proof}
As in \lref{pen_est}, we have the bound,
\begin{align*}
\sum_{k=0}^{k=\infty}\int_0^{T}\int_{\R^3}\int_{\R^3}&|\Gamma^k_{\beta''',\omega'''}(\derv{'}{'}{'} g,\derv{'}{'}{'} g)\jap{x-(t+1)v}^4\der g|\d v\d x\d t\\
&\lesssim \mathbb{I}.
\end{align*}
The only difference is that for $|\alpha'|+|\beta'|+|\omega'|\geq 8$, we use the second estimate from \lref{trivial_est_1} and \lref{trivial_est_2}.

For $k\geq 0$, we bound infinite sums of \eref{low_main}, \eref{low_comm} and \eref{low_Gaussian_diff_term}.\\
We use \eref{H2s_1} from \lref{derv_dist} with $\mf=\mu_{\beta''',\omega'''}\derv{'}{'}{'} g$,\\ $\mg=\jap{x-(t+1)v}^2\derv{''}{''}{''} g$ and $\mh=\jap{x-(t+1)v}^2\der g$ and the fact that $|\mu_{\beta''',\omega'''}\jap{v}^m\lesssim 1|$ for any $m\in \R$ to get,
$$\sum_{k=0}^{k=\infty} |\eref{low_main}|\lesssim\mathbb{XII}.$$

Next we use \lref{comm_estimates} with $\mf=\mu_{\beta''',\omega'''}\derv{'}{'}{'} g$, $\mg=\derv{''}{''}{''} g$ and $\mh=\der g$ to get,
$$\sum_{k=0}^{k=\infty} |\eref{low_comm}|\lesssim \mathbb{III}+\mathbb{IV}+\mathbb{X}.$$

Finally we use \lref{comm_gaussian} with $\mf=\mu_{\beta''',\omega'''}\derv{'}{'}{'} g$, $\mg=\derv{''}{''}{''} g$ and $\mh=\der g$ to get,
$$\sum_{k=0}^{k=\infty} |\eref{low_Gaussian_diff_term}|\lesssim \mathbb{I}+\mathbb{III}+\mathbb{XI}.$$
\end{proof}
\begin{proof}[Proof of \lref{main_boltz_estimates}]
Combining \lref{top_est}, \lref{pen_est} and \lref{low_est} gives us \lref{main_boltz_estimates}.
\end{proof}
\section{Local existence}\label{s.local}
With the estimates in hand from last sections, local existence can be proved by following \cite{AMUXY10}, in which the authors introduce a cut-off in angle and then take the limit or by following \cite{HeSnTa19} via a method of continuity. Thus, we give priori estimates in this section which can be easily turned into an existence theorem (\lref{local_ext}) using the methods of \cite{AMUXY10} or \cite{HeSnTa19}.
\begin{lemma}\label{l.local}
Fix $d_0>0$, and let $\delta$ be the same as in \eref{delta}. Let $f$ be a sufficiently regular and non-negative solution to \eref{boltz} on $[0,T_*)\times\R^3\times \R^3$ such that the initial data satisfies
\begin{equation*}
\sum\limits_{|\alpha|+|\beta|\leq 10}\norm{\jap{x-(t+1)v}^2\part^\alpha_x\part_v^\beta(e^{2d_0\jap{v}^2}f_{\ini})}_{L^2_xL^2_v}^2\leq \eps^2,
\end{equation*}
for some $\eps$ small enough . In addition to this, assume that $f$ also satisfies the a priori bound
$$\sum_{|\alpha|+|\beta|+|\omega|\leq 10}\norm{\jap{x-(t+1)v}^2\der (e^{d(t)\jap{v}^2}f)}_{L^\infty([0,T];L^2_xL^2_v)}\leq \bar{\eps}^2,$$
for $\bar{\eps}\geq \eps$ but small enough depending on $d_0$ and $\delta$. Then $f$ actually satisfies the following bound for all $T\in (0,T_*]$
\begin{equation}\label{e.local}
\begin{split}
&\sum_{|\alpha|+|\beta|+|\omega|\leq 10}\norm{\jap{x-(t+1)v}^2\der (e^{d(t)\jap{v}^2}f)}^2_{L^\infty([0,T];L^2_xL^2_v)}\\
&\qquad+\sum_{|\alpha|+|\beta|+|\omega|\leq 10}\norm{\jap{v}\jap{x-(t+1)v}^2\der (e^{d(t)\jap{v}^2}f)}^2_{L^2([0,T];L^2_xL^2_v)}\\
&\hspace{8em}\leq \eps^2\exp(C(d_0,\delta,\gamma,s)T).
\end{split}
\end{equation}
\end{lemma}
\begin{proof}
We first recall some notations that will be helpful throughout the proof:
\begin{itemize}
\item As in \sref{set_up}, we let $g=e^{d(t)\jap{v}^2}f$.\\
\item $\norm{h}^2_{Y^m}(T)=\sum \limits_{|\alpha|+|\beta|+|\omega|\leq m}\norm{\jap{x-(t+1)v}^2 h}^2_{L^2_xL^2_v}(T)$.\\
For $m=10$, we will drop the superscript.\\
\item $\norm{h}^2_{X^m}(T)=\sum \limits_{|\alpha|+|\beta|+|\omega|\leq m}=\norm{\jap{v}\jap{x-(t+1)v}^2 h}^2_{L^2([0,T];L^2_xL^2_v)}$.\\
For $m=10$, we again drop the superscript.
\end{itemize}
Now, using \lref{eng_set_up} and the bound on initial data, we get that
\begin{align*}
&\norm{g}^2_{Y}(T)+\norm{g}^2_{X}(T)\leq\eps^2+\text{Comm}_1+\text{Comm}_2\\
&\quad +\mathlarger{\sum}_{\substack{|\alpha'|+|\alpha''|=|\alpha|\\|\beta'|+|\beta''|+|\beta'''|=|\beta|\\|\omega'|+|\omega''|+|\omega'''|=|\omega|}}\int_0^{T}\int_{\R^3} \int_{\R^3} \Gamma_{\beta''',\omega'''}(\derv{'}{'}{'} g,\derv{''}{''}{''} g)\jap{x-(t+1)v}^4\der g\d v\d x\d t.
\end{align*}
By assuming $T_*\leq 1$, we can ignore all the space-time weight issues.\\
\textbf{Estimate for $\text{Comm}_1$.} Using H\"{o}lder's inequality followed by Young's inequality we have 
\begin{align*}
\text{Comm}_1&\lesssim \int_{0}^{T}\norm{\jap{x-(t+1)v}^2\der g}^2_{L^2_xL^2_v}\d t\\
&\quad+ \sum_{\substack{|\alpha'|\leq |\alpha|+1\\ |\beta'|\leq |\beta|-1}}\int_0^T\norm{\jap{x-(t+1)v}^2\derv{'}{'}{'} g}^2_{L^2_xL^2_v}\d t\\
&\lesssim \int_0^T\norm{g}^2_Y(t)\d t.
\end{align*}
\textbf{Estimate for $\text{Comm}_2$.} Using Young's inequality we have,
\textcolor{black}{\begin{align*}
\jap{v}&\jap{x-(t+1)v}^4|\der g||\derv{'}{'}{} g|\\
&\leq \eta \jap{v}^2\jap{x-(t+1)v}^4|\der g|^2+C_\eta\jap{x-(t+1)v}^4|\derv{'}{'}{'} g|^2.
\end{align*}}
Using this and H\"{o}lder's we get that 
\begin{align*}
\text{Comm}_2&\leq  C_\eta \int_0^T\norm{g}^2_Y(t)\d t+\eta\norm{g}^2_X(T).
\end{align*}
\textbf{Estimates for the integral of the kernel.} Thanks to \lref{main_boltz_estimates}, it suffices to bound the terms $\mathbb{I}-\mathbb{XII}$.\\
Since we do not care about time decay, the estimates are straight forward as we just need to make sure that we can close in terms of derivatives. But the estimates in \sref{main_est} are such that we do not have a derivative loss problem. To give an example, we can bound $\mathbb{I}$ by using Sobolev embedding and the a priori bound as follows,
\begin{align*}
\int_0^{T}&(1+t)^{1+\delta}\norm{(1+t)^{-\frac{1}{2}-\frac{\delta}{2}}\jap{x-(t+1)v}^2\jap{v}\derv{''}{''}{''} g}_{L^2_xL^2_v} \\
&\qquad \times \norm{\jap{x-(t+1)v}^2\derv{'}{'}{'}g}_{L^\infty_xL^2_v}\norm{(1+t)^{-\frac{1}{2}-\frac{\delta}{2}}\jap{x-(t+1)v}^2\jap{v}\der g}_{L^2_xL^2_v}\d t\\
&\lesssim \norm{(1+t)^{-\frac{1}{2}-\frac{\delta}{2}}\jap{x-(t+1)v}^2\jap{v}\derv{''}{''}{''} g}_{L^2([0,T];L^2_xL^2_v)} \\
&\qquad \times \norm{\jap{x-(t+1)v}^2\derv{'}{'}{'}g}_{L^\infty([0,T];L^\infty_xL^2_v)}\\
&\hspace{5em}\times\norm{(1+t)^{-\frac{1}{2}-\frac{\delta}{2}}\jap{x-(t+1)v}^2\jap{v}\der g}_{L^2([0,T];L^2_xL^2_v)}\\
&\lesssim \norm{(1+t)^{-\frac{1}{2}-\frac{\delta}{2}}\jap{x-(t+1)v}^2\jap{v}\derv{''}{''}{''} g}_{L^2([0,T];L^2_xL^2_v)} \\
&\qquad \times {\color{black}{\sum_{|\alpha'''|\leq |\alpha'|+2}\norm{\jap{x-(t+1)v}^2\derv{'''}{'}{'}g}_{L^\infty([0,T];L^\infty_xL^2_v)}}}\\
&\hspace{5em}\times\norm{(1+t)^{-\frac{1}{2}-\frac{\delta}{2}}\jap{x-(t+1)v}^2\jap{v}\der g}_{L^2([0,T];L^2_xL^2_v)}\\
&\lesssim \bar \eps\norm{g}_X^2(T).
\end{align*}

In a very similar way we can prove by using Sobolev embedding, that,
\begin{align*}
\sum_{\substack{|\alpha'|+|\alpha''|=|\alpha|\\|\beta'|+|\beta''|+|\beta'''|=|\beta|\\ |\omega'|+|\omega''|+|\omega'''|=|\omega|}}&\int_{0}^{T}\int_{\R^3}\int_{\R^3}\Gamma_{\beta''',\omega'''}(\derv{'}{'}{'} g,\derv{''}{''}{''} g)\jap{x-(t+1)v}^4 \der g\\
&\lesssim \bar \eps \norm{g}_X^2(T).
\end{align*}

We also prove these estimates in more detail in \sref{errors} in \lref{kernel_boot_est} and thus we skip the details here.

By combining the above estimates we have that,
\begin{align*}
&\norm{g}^2_{Y}(T)+\norm{g}^2_{X}(T)\leq \eps^2+ C(d_0,\delta,\gamma,s)\int_0^T \norm{g}^2_{Y}(t)\d t+C(d_0,\delta,\gamma,s)\bar \eps \norm{g}^2_X(T).
\end{align*}
If $\bar \eps C\leq \frac{1}{2}$, we can absorb it on the right hand side to get,
$$\norm{g}^2_Y(T)+\norm{g}^2_{X}(T)\lesssim \eps^2+ C\int_0^T \norm{g}^2_{Y}(t)\d t.$$
Now \eref{local} follows by Gr\"{o}nwall's inequality.
\end{proof}

We now state the local existence theorem that we use without proof. Using the apriori estimate in \lref{local}, one can prove the following lemma employing techniques from \cite{AMUXY10} or \cite{HeSnTa19}.
\begin{lemma}\label{l.local_ext}
Fix $s\in (0,1)$ and $\gamma$ such that $\gamma+2s\in(0,2)$. Further let $M_0>0$, $d_0>0$ and the initial data $f_{\ini}$ be non-negative and such that
$$\sum_{|\alpha|+|\beta|\leq 10}\norm{(1+|x-v|^2)\part^\alpha_x\part^\beta_v (e^{2d_0(1+|v|^2)} f_{\ini})}_{L^2_xL^2_v}^2\leq M_0^2.$$ 
Then there exists a time $T>0$ (depending only on $M_0,$ $d_0,$ $\gamma$ and $s$) such that \eref{boltz} admits a solution $f(t,x,v)$ on $[0,T]\times\R^3\times\R^3$ satisfying the initial condition, i.e. $f(0,x,v)=f_{\ini}(x,v)$. Further, $f(t,x,v)$ remains non-negative for all $t\in [0,T].$

In addition, if we let $\delta$ such that $$\delta=\min\left\{\frac{1-s}{4},\frac{1}{10},\frac{\gamma+2s}{8}\right\},$$ we have that $fe^{d_0(1+(1+t)^{-\delta})(1+|v|^2)} $ is unique in the energy space $E^4_T\cap C^0([0,T);Y^4_{x,v})$. Moreover, $fe^{d_0(1+(1+t)^{-\delta})(1+|v|^2)}\in E_T\cap C^0([0,T);Y_{x,v})$.
\end{lemma}
See \sref{notation} for the definition of the energy spaces.
\section{Bootstrap assumptions and the bootstrap theorem}\label{s.bootstrap}
Now we introduce the bootstrap assumption for $E$ norm
\begin{equation}\label{e.boot_assumption_1}
E_T(g)\leq \eps^{\frac{3}{4}}
\end{equation}
Our goal from now until \sref{errors} would be to improve the bootstrap assumption \eref{boot_assumption_1} with $\eps^{\frac{3}{4}}$ replaced by $C\eps$ for some constant $C$ depending only on $d_0$, $\gamma$ and $s$. This is presented as a theorem below,
\begin{theorem}\label{t.boot}
Let $\gamma$, $s$, $d_0$ and $f_{\ini}$ be as in \tref{global} and let $\delta>0$ be as in \eref{delta}. There exists $\eps_0=\eps_0(d_0,\gamma,s)>0$ and $C_0=C_0(d_0,\gamma,s)>0$ with $C_0\eps_0\leq \frac{1}{2}\eps^{\frac{3}{4}}_0$ such that the following holds for $\eps \in [0,\eps_0]$:\\
Suppose there exists $T_{\boot}>0$ and a solution $f:[0,T_{\boot})\times\R^3\times\R^3$ with $f(t,x,v)\geq 0$ and $f(0,x,v)=f_{\ini}(x,v)$ such that the estimate \eref{boot_assumption_1} holds for $T\in[0,T_{\boot})$, then the estimate in fact holds for all $T\in[0,T_{\boot})$ with $\eps^{\frac{3}{4}}$ replaced by $C\eps.$
\end{theorem}
\textbf{In the next sections, we always work under the assumptions of \tref{boot}}.
\section{Error estimates for global problem}\label{s.errors}
We first begin by a few lemmas that will help us get the required time decay.
\begin{lemma}\label{l.L2_time_in_norm} 
For $T\in (0,T_{\boot}]$ and $|\alpha|+|\beta|+|\omega|\leq 10$,
$$\norm{(1+t)^{-\frac{1}{2}-\delta-|\beta|}\jap{v}\jap{x-(t+1)v}^{2}\der g}_{L^2([0,T];L^2_xL^2_v)}\lesssim \eps^{\frac{3}{4}}.$$
\end{lemma}
\begin{proof}
We can assume that $T>1$, otherwise the inequality is an immediate consequence of the bootstrap assumtion \eref{boot_assumption_1}.\\
We split the integration in time into dyadic intervals. More precisely, let $k=\lceil\log_2 T\rceil+1$. Define $\{T_i\}_{i=0}^k$ with $T_0<T_1<\dots<T_k$, where $T_0=0$ and $T_i=2^{i-1}$ where $i=1,\dots, k-1$ and $T_k=T$. Now by the bootstrap assumption \eref{boot_assumption_1}, for any $i=1,\dots,k$, 
$$\norm{(1+t)^{-\frac{1}{2}-\frac{\delta}{2}}\jap{v}\jap{x-(t+1)v}^{2}\der g}_{L^2([T_{i-1},T_i];L^2_xL^2_v)}\lesssim \eps^{\frac{3}{4}}2^{|\beta|i}.$$
Thus,
\begin{align*}
&\norm{(1+t)^{-\frac{1}{2}-\delta-|\beta|}\jap{v}\jap{x-(t+1)v}^{2}\der g}_{L^2([0,T];L^2_xL^2_v)}\\
&\lesssim (\sum_{i=1}^k\norm{(1+t)^{-\frac{1}{2}-\delta-|\beta|}\jap{v}\jap{x-(t+1)v}^{2}\der g}^2_{L^2([T_{i-1},T_I];L^2_xL^2_v)})^{\frac{1}{2}}\\
&\lesssim (\sum_{i=1}^k2^{-2|\beta|i-\delta i}\norm{(1+t)^{-\frac{1}{2}-\frac{\delta}{2}}\jap{v}\jap{x-(t+1)v}^{2}\der g}^2_{L^2([T_{i-1},T_I];L^2_xL^2_v)})^{\frac{1}{2}}\\
&\lesssim \eps^{\frac{3}{4}}(\sum_{i=1}^k 2^{-2|\beta|i-\delta i}\cdot 2^{2|\beta|i})^{\frac{1}{2}}=\eps^{\frac{3}{4}}(\sum_{i=1}^k 2^{-\delta i})^{\frac{1}{2}}\lesssim \eps^{\frac{3}{4}}.
\end{align*}
\end{proof}
\begin{lemma}[Vector Field Trick] \label{l.vect_trick}
For a smooth $h$ we have the following embedding estimate
$$\norm{h}_{L^\infty_xL^2_v} \lesssim\sum_{|\omega|\leq 3}(1+t)^{\frac{-3}{2}}  \norm{Y^\omega h}_{L^2_xL^2_v}.$$
\end{lemma}
\begin{proof}

From fundamental theorem of calculus we have that $$ \sup_{x} \int_{\R^3} h^2 \d v  \leq \vert\int_{\R^3} \part_{x_1}  \part_{x_2} \part_{x_3} \int_{\R^3} h^2 \d v \d x\vert.$$

Now by an application of dominated convergence theorem, we can pull the $x$ derivatives inside the integral for $v$. 

We write $\part_{x_i}=(1+t)^{-1} (1+t)\part_{x_i}$, we get\\
$$\int_{\R^3} \part_{x_1}  \part_{x_2} \part_{x_3} \int_{\R^3} h^2 \d v \d x =(1+t)^{-3}\int_{\R^3} \int_{\R^3}  ((1+t)\part_{x_1})  ((1+t)\part_{x_2}) ((1+t)\part_{x_3})h^2 \d v \d x.$$

Now noting that \textcolor{black}{$\int_{\R^3}\part_{v_i} h^2 \d v_i=0$}, we get that 
$$(1+t)^{-3}\int_{\R^3} \int_{\R^3}  ((1+t)\part_{x_1})  ((1+t)\part_{x_2}) ((1+t)\part_{x_3})h^2 \d v \d x=(1+t)^{-3}\int_{\R^3} \int_{\R^3} Y_1Y_2Y_3 h^2 \d v \d x.$$

This implies the required lemma.\\
\end{proof}
\begin{lemma}\label{l.L1_to_L2}
Let $h$ be a smooth function. Then for any $n>\frac{3}{2}$, we have that
$$\norm{h}_{L^1_v}(t,x)\lesssim (1+t)^{-\frac{3}{2}}\norm{\jap{x-(t+1)v}^{n}h(t,x,v)}_{L^2_v}.$$ 
\end{lemma}
\begin{proof}
We break the region of integration to obtain decay from the $\jap{x-(t+1)v}$ weights.
\begin{align*}
\int_{\R^3}|h|(t,x,v)\d v &\lesssim \int_{\R^3\cap\{|v-\frac{x}{t+1}|\leq (1+t)^{-1}\}}|h| \d v+\int_{\R^3\cap\{|v-\frac{x}{t+1}|\geq (1+t)^{-1}\}}|h| \d v\\
&\lesssim (\int_{\R^3\cap\{|v-\frac{x}{t+1}|\leq (1+t)^{-1}\}}h^2 \d v)^{\frac{1}{2}}(\int_{\R^3\cap\{|v-\frac{x}{t+1}|\leq (1+t)^{-1}\}}1 \d v)^{\frac{1}{2}}\\
&\quad+ (\int_{\R^3\cap\{|v-\frac{x}{t+1}|\geq (1+t)^{-1}\}}|v-\frac{x}{t+1}|^{2n}h^2 \d v)^{\frac{1}{2}}(\int_{\R^3\cap\{|v-\frac{x}{t+1}|\geq (1+t)^{-1}\}}|v-\frac{x}{t+1}|^{-2n} \d v)^{\frac{1}{2}}\\
&\lesssim (1+t)^{-\frac{3}{2}}(\int_{\R^3}h^2\d v)^{\frac{1}{2}}+(1+t)^{-n}(\int_{\R^3}\jap{x-(t+1)v}^{2n}h^2\d v)^{\frac{1}{2}}(1+t)^{\frac{2n-3}{2}}\\
&\lesssim (1+t)^{-\frac{3}{2}}(\int_{\R^3}\jap{x-(t+1)v}^{2n}h^2\d v)^{\frac{1}{2}}.
\end{align*}
\end{proof}

\begin{lemma}\label{l.L1_to_L2_less_decay}
Let $h$ be a smooth function. Then we have that
$$\norm{h}_{L^1_v}(t,x)\lesssim (1+t)^{-1+}\norm{\jap{x-(t+1)v}h(t,x,v)\jap{v}}_{L^2_v}.$$ 
\end{lemma}
\begin{proof}
We again break the region of integration to obtain decay from the $\jap{x-(t+1)v}$ weights. The for every $\nu>0$ and for $\frac{1}{q}+\frac{1}{3+\nu}=\frac{1}{2}$ we have that
\begin{align*}
\int_{\R^3}|h|(t,x,v)\d v &\lesssim \int_{\R^3\cap\{|v-\frac{x}{t+1}|\leq (1+t)^{-1}\}}|h| \d v+\int_{\R^3\cap\{|v-\frac{x}{t+1}|\geq (1+t)^{-1}\}}|h| \d v\\
&\lesssim (\int_{\R^3\cap\{|v-\frac{x}{t+1}|\leq (1+t)^{-1}\}}h^2 \d v)^{\frac{1}{2}}(\int_{\R^3\cap\{|v-\frac{x}{t+1}|\leq (1+t)^{-1}\}}1 \d v)^{\frac{1}{2}}\\
&\quad+ (\int_{\R^3\cap\{|v-\frac{x}{t+1}|\geq (1+t)^{-1}\}}|v-\frac{x}{t+1}|^2h^2\jap{v}^2 \d v)^{\frac{1}{2}}\\
&\qquad\times(\int_{\R^3\cap\{|v-\frac{x}{t+1}|\geq (1+t)^{-1}\}}|v-\frac{x}{t+1}|^{-3-\nu} \d v)^{\frac{1}{3+\nu}} (\int_{\R^3\cap\{|v-\frac{x}{t+1}|\geq (1+t)^{-1}\}}\jap{v}^{-q} \d v)^{\frac{1}{q}}\\
&\lesssim (1+t)^{-\frac{3}{2}}(\int_{\R^3}h^2\d v)^{\frac{1}{2}}+(1+t)^{-1}(\int_{\R^3}\jap{x-(t+1)v}^{2}h^2\jap{v}^2\d v)^{\frac{1}{2}}(1+t)^{\frac{\nu}{3+\nu}}\\
&\lesssim (1+t)^{-1^+}(\int_{\R^3}\jap{x-(t+1)v}^2h^2\jap{v}^2\d v)^{\frac{1}{2}}.
\end{align*}
\end{proof}
\begin{corollary}\label{c.combining_decay_1}
Let $h$ be a smooth function, then we have
$$\norm{h}_{L^\infty_xL^1_v}(t)\lesssim \sum_{|\omega|\leq 3}(1+t)^{-3}\norm{Y^\omega h\jap{x-(t+1)v}^2}_{L^2_xL^2_v}(t).$$
\end{corollary}
\begin{proof}
Using \lref{L1_to_L2} with $n=2$, we get for every $x\in \R^3$, $$\norm{h(t,x,v)}_{L^1_v}\lesssim (1+t)^{-\frac{3}{2}}\norm{h(t,x,v)\jap{x-(t+1)v}^2}_{L^2_v}.$$

Next using \lref{vect_trick}, we get that $$\norm{h\jap{x-(t+1)v}^2}_{L^\infty_xL^2_v}\lesssim (1+t)^{-\frac{3}{2}} \sum_{|\omega|\leq 3}\norm{Y^\omega(h\jap{x-(t+1)v}^2)}_{L^2_xL^2_v}.$$

Finally noting that $Y(\jap{x-(t+1)v})=0$, we get the required result.
\end{proof}

\begin{corollary}\label{c.combining_decay_2}
Let $h$ be a smooth function, then we have
$$\norm{\jap{x-(t+1)v}h}_{L^\infty_xL^1_v}(t)\lesssim (1+t)^{-\frac{5}{2}+}\sum_{|\omega|\leq 3}\norm{Y^\omega h\jap{x-(t+1)v}^2\jap{v}}_{L^2_xL^2_v}(t).$$
\end{corollary}
\begin{proof}
Using \lref{L1_to_L2_less_decay}, we get for every $x\in \R^3$, $$\norm{h(t,x,v)}_{L^1_v}\lesssim (1+t)^{-1^+}\norm{h(t,x,v)\jap{x-(t+1)v}^2\jap{v}}_{L^2_v}.$$

Next using \lref{vect_trick}, we get that $$\norm{h\jap{x-(t+1)v}^2}_{L^\infty_xL^2_v}\lesssim (1+t)^{-\frac{3}{2}} \norm{Y^3(h\jap{x-(t+1)v}^2\jap{v})}_{L^2_xL^2_v}.$$

Finally noting that $Y(\jap{x-(t+1)v})=0$ and that $Y\jap{v}\lesssim 1$, we get the required result.
\end{proof}
\begin{lemma}\label{l.partial_derv_interpol}
For any $\eta\in (0,1)$ and $n\in \mathbb{N}\cup \{0\}$, we have the following interpolation result\\
$$\norm{f}_{H^{n+\eta}_v}\lesssim \norm{ f}_{H^n_v}+\norm{\part^\alpha_v f}_{ L^2_v}^{1-\eta}\norm{\part_{v_i}\part^\alpha_{v} f}^{\eta}_{L^2_v},$$
where $|\alpha|=n$.
\end{lemma}

\begin{lemma}\label{l.kernel_boot_est}
For $\gamma+2s\in(0,2]$ and $|\alpha|+|\beta|+|\omega|\leq10$, then for every $T\in [0,T_{\boot})$ we have the following bound,
\begin{align*}
\sum_{\substack{|\alpha'|+|\alpha''|=|\alpha|\\|\beta'|+|\beta''|+|\beta'''|=|\beta|\\ |\omega'|+|\omega''|+|\omega'''|=|\omega|}}\int_{0}^{T_*}\int_{\R^3}\int_{\R^3}&\Gamma_{\beta''',\omega'''}(\derv{'}{'}{'} g,\derv{''}{''}{''} g)\jap{x-(t+1)v}^4 \der g\\
&\lesssim \eps^2(1+T)^{2|\beta|}.
\end{align*}
\end{lemma}
\begin{proof}
Using \lref{main_boltz_estimates}, it suffices to get a bound for the terms $\mathbb{I}$ to $\mathbb{XII}.$

We begin by bounding $\mathbb{I}$ from \eref{main_bound_no_sing}. For the first half we have that $|\alpha'|+|\beta'|+|\omega'|\leq 7$, thus using \lref{vect_trick} and \eref{boot_assumption_1}, we have
\begin{align*}
&\int_0^{T_*}(1+t)^{1+\delta}\norm{(1+t)^{-\frac{1}{2}-\frac{\delta}{2}}\jap{x-(t+1)v}^2\jap{v}\derv{''}{''}{''} g}_{L^2_xL^2_v} \\
&\qquad \times \norm{\jap{x-(t+1)v}^2\derv{'}{'}{'}g}_{L^\infty_xL^2_v}\norm{(1+t)^{-\frac{1}{2}-\frac{\delta}{2}}\jap{x-(t+1)v}^2\jap{v}\der g}_{L^2_xL^2_v}\d t\\
&\lesssim (1+T)^{2|\beta|}\left((1+T)^{-|\beta''|}\norm{(1+t)^{-\frac{1}{2}-\frac{\delta}{2}}\jap{x-(t+1)v}^2\jap{v}\derv{''}{''}{''} g}_{L^2([0,T];L^2_xL^2_v}\right)\\
&\qquad\times (1+T)^{-|\beta'|}\left(\sup_{t\in [0,T]}((1+t)^{1+\delta}\norm{\jap{x-(t+1)v}^2\derv{'}{'}{'} g}_{L^\infty_xL^2_v}\right)\\
&\qquad \times \left((1+T)^{-|\beta|}\norm{(1+t)^{-\frac{1}{2}-\frac{\delta}{2}}\jap{x-(t+1)v}^2\jap{v}\der g}_{L^2([0,T];L^2_xL^2_v}\right)\\
&\lesssim (1+T)^{2|\beta|}\eps^{\frac{3}{4}}\times\eps^{\frac{3}{4}}\\
&\qquad\times (1+T)^{-|\beta'|}\left(\sup_{t\in [0,T]}((1+t)^{1+\delta}(1+t)^{-\frac{3}{2}}\sum_{|\omega'''|\leq|\omega'|+3}\norm{\jap{x-(t+1)v}^2\derv{'}{'}{'''} g}_{L^2_xL^2_v}\right)\\
&\lesssim (1+T)^{2|\beta|}\eps^{\frac{3}{2}}\times \eps^{\frac{3}{2}}=(1+T)^{2|\beta|}\eps^{\frac{9}{4}}.
\end{align*}
For the second half we proceed in the same way as before, except that we apply \lref{vect_trick} to $\norm{(1+t)^{-\frac{1+\delta}{2}}\jap{v}\jap{x-(t+1)v}^2\derv{''}{''}{''} g}_{L^\infty_xL^2_v}$. More precisely,
\begin{align*}
&\int_0^{T_*}(1+t)^{1+\delta}\norm{(1+t)^{-\frac{1}{2}-\frac{\delta}{2}}\jap{x-(t+1)v}^2\jap{v}\derv{''}{''}{''} g}_{L^\infty_xL^2_v} \\
&\qquad \times \norm{\jap{x-(t+1)v}^2\derv{'}{'}{'}g}_{L^2_xL^2_v}\norm{(1+t)^{-\frac{1}{2}-\frac{\delta}{2}}\jap{x-(t+1)v}^2\jap{v}\der g}_{L^2_xL^2_v}\d t\\
&\lesssim \int_{0}^{T}(1+t)^{-\frac{3}{2}}(1+t)^{1+\delta}\left(\sum_{|\omega'''|\leq|\omega''|+3}\norm{(1+t)^{-\frac{1}{2}-\frac{\delta}{2}}\jap{x-(t+1)v}^2\jap{v}\derv{''}{''}{'''} g}_{L^2_xL^2_v}\right)\\
&\qquad\times \norm{\jap{x-(t+1)v}^2\derv{'}{'}{'} g}_{L^2_xL^2_v}\\
&\qquad \times\norm{(1+t)^{-\frac{1}{2}-\frac{\delta}{2}}\jap{x-(t+1)v}^2\jap{v}\der g}_{L^2_xL^2_v}\d t\\
&\lesssim(1+T)^{2|\beta|}\left((1+T)^{-|\beta''|}\norm{(1+t)^{-\frac{1}{2}-\frac{\delta}{2}}\jap{x-(t+1)v}^2\jap{v}Y^3\derv{''}{''}{''} g}_{L^2([0,T];L^2_xL^2_v)}\right)\\
&\qquad\times (1+T)^{-|\beta'|}\sup_{t\in [0,T]}\left((1+t)^{-\frac{3}{2}}(1+t)^{1+\delta}\norm{\jap{x-(t+1)v}^2\derv{'}{'}{'} g}_{L^2_xL^2_v}\right)\\
&\qquad \times \left((1+T)^{-|\beta|}\norm{(1+t)^{-\frac{1}{2}-\frac{\delta}{2}}\jap{x-(t+1)v}^2\jap{v}\der g}_{L^2([0,T];L^2_xL^2_v)}\right)\\
&\lesssim (1+T)^{2|\beta|}\eps^{\frac{3}{4}}\times\eps^{\frac{3}{4}}\times\eps^{\frac{3}{4}}. 
\end{align*}
 Thus, after making $\eps$ smaller if needed, we get that,
$$\mathbb{I}\lesssim (1+T)^{2|\beta|}\eps^2.$$

For $\mathbb{II}$, we first use Cauchy--Schwarz in time to get,
\begin{align*}
\mathbb{II}&\lesssim (1+T)^{2|\beta|}\norm{(1+t)^{-\frac{1+2\delta}{2}-|\beta''|}\jap{x-(t+1)v}^2 \jap{v}\derv{''}{''}{''} g}_{L^2([0,T];L^2_xL^2_v)}\\
&\qquad\times\norm{(1+t)^{-\frac{1+2\delta}{2}-|\beta|}\jap{x-(t+1)v}^2\jap{v} \der g}_{L^2([0,T];L^2_xL^2_v)}\\
&\qquad\times \sup_{t\in[0,T]}\left((1+t)^{1+2\delta}(1+t)^{-|\beta'|}\norm{\jap{v}^{\gamma+2s}\derv{'}{'}{'} f}_{L^\infty_xW^{2,1}_v}^s\norm{\jap{v}^{\gamma+2s}\derv{'}{'}{'} f}_{L^\infty_xL^1_v}^{1-s}\right).
\end{align*}
Next using \cref{combining_decay_1} and \lref{exp_bound} followed by bootstrap assumption \eref{boot_assumption_1} we get,
\begin{align*}
(1+t)^{1+2\delta-|\beta'|}\norm{\jap{v}^{\gamma+2s}&\derv{'}{'}{'} f}_{L^\infty_xW^{2,1}_v}^s\norm{\jap{v}^{\gamma+2s}\derv{'}{'}{'} f}_{L^\infty_xL^1_v}^{1-s}\\
&\lesssim \sum_{|\omega'''|\leq |\omega'|+3}(1+t)^{1+2\delta-|\beta'|}(1+t)^{-3}\norm{\jap{v}^{\gamma+2s}\jap{x-(t+1)v}^2\derv{'}{'}{'''} f}_{L^2_xH^2_v}^s\\
&\hspace{5em}\times\norm{\jap{v}^{\gamma+2s}\jap{x-(t+1)v}^2\derv{'}{'}{'} f}_{L^2_xL^2_v}^{1-s}\\
&\lesssim \sum_{|\omega'''|\leq |\omega'|+3}(1+t)^{1+2\delta-|\beta'|}(1+t)^{-3}\norm{\jap{x-(t+1)v}^2\derv{'}{'}{'''} g}_{L^2_xH^2_v}^s\\
&\hspace{5em}\times\norm{\jap{v}^{\gamma+2s}\jap{x-(t+1)v}^2\derv{'}{'}{'} g}_{L^2_xL^2_v}^{1-s}\\
&\lesssim \eps^{\frac{3}{4}}(1+t)^{-3+2s+1+2\delta}\\
&\lesssim \eps^{\frac{3}{4}},
\end{align*}
where we used the fact that $\delta<1-s$.\\
Finally using the bootstrap assumption \eref{boot_assumption_1}, we get 
$$\norm{(1+t)^{-\frac{1+2\delta}{2}-|\beta''|}\jap{x-(t+1)v}^2\jap{v} \derv{''}{''}{"} g}_{L^2([0,T];L^2_xL^2_v)}\lesssim \eps^{\frac{3}{4}}$$
and 
$$\norm{(1+t)^{-\frac{1+2\delta}{2}-|\beta|}\jap{x-(t+1)v}^2\jap{v} \der g}_{L^2([0,T];L^2_xL^2_v)}\lesssim \eps^{\frac{3}{4}}.$$
Combining all of this we get the desired result,
$$\mathbb{II}\lesssim \eps^{2}(1+T)^{2|\beta|}.$$

\textbf{Henceforth we will only concern ourselves with the first half of every remaining term (i.e. for $|\alpha'|+|\beta'|+|\omega'|\leq 7$). For the remaining half, we can just proceed in the same way as for the first half with the only change that we apply \lref{vect_trick} to the term with $\derv{''}{''}{''} g$.}

Next for $\mathbb{III}$, we use \cref{combining_decay_2} and \eref{boot_assumption_1} to get,
\begin{align*}
\int_0^{T}&(1+t)^{2+\delta}\norm{(1+t)^{-\frac{1+\delta}{2}}\jap{x-(t+1)v}^2\jap{v}\derv{''}{''}{''} g}_{L^2_xL^2_v}\\
&\qquad\times \norm{\jap{x-(t+1)v}\derv{'}{'}{'} g}_{L^\infty_xL^1_v}\norm{(1+t)^{-\frac{1+\delta}{2}}\jap{x-(t+1)v}^2\jap{v}\der g}_{L^2_xL^2_v}\d t\\
&\lesssim (1+T)^{2|\beta|}\left((1+T)^{-|\beta''|}\norm{(1+t)^{-\frac{1+\delta}{2}}\jap{x-(t+1)v}^2\jap{v}\derv{''}{''}{''} g}_{L^2([0,T];L^2_xL^2_v)}\right)\\
&\qquad\times (1+T)^{-|\beta'|}\left(\sup_{t\in[0,T]} (1+t)^{2+\delta}\norm{\jap{x-(t+1)v}\derv{'}{'}{'} g}_{L^\infty_xL^1_v}\right)\\
&\times \left((1+T)^{-|\beta|}\norm{(1+t)^{-\frac{1+\delta}{2}}\jap{x-(t+1)v}^2\jap{v}\der g}_{L^2([0,T];L^2_xL^2_v)}\right)\\
&\lesssim (1+T)^{2|\beta|}\eps^{\frac{3}{4}}\times \eps^{\frac{3}{4}}\\
&\qquad\times (1+T)^{-|\beta'|}\sup_{t\in[0,T]}\left((1+t)^{2+\delta}(1+t)^{-\frac{5}{2}+}\sum_{|\omega'''|\leq|\omega'|+3} \norm{\jap{x-(t+1)v}\derv{'}{'}{'''} g}_{L^2_xL^2_v}\right)\\
&\lesssim (1+T)^{2|\beta|}\eps^{\frac{9}{4}}.
\end{align*}

For $\mathbb{IV}$, we use \cref{combining_decay_1}, \eref{boot_assumption_1} and that $\delta<\frac{1}{4}$ to get,
\begin{align*}
\int_0^{T}&(1+t)^{3-\delta}\norm{(1+t)^{-\frac{1+\delta}{2}}\jap{x-(t+1)v}^2\jap{v}\derv{''}{''}{''} g}_{L^2_xL^2_v}\\
&\qquad\times \norm{\jap{x-(t+1)v}^{2\delta}\derv{'}{'}{'} g}_{L^\infty_xL^1_v}\norm{(1+t)^{-\frac{1+\delta}{2}}\jap{x-(t+1)v}^2\jap{v}\der g}_{L^2_xL^2_v}\\
&\lesssim (1+T)^{2|\beta|}\left((1+T)^{-|\beta''|}\norm{(1+t)^{-\frac{1+\delta}{2}}\jap{x-(t+1)v}^2\jap{v}\derv{''}{''}{''} g}_{L^2([0,T];L^2_xL^2_v)}\right)\\
&\qquad\times (1+T)^{-|\beta'|}\left(\sup_{t\in[0,T]} (1+t)^{3-\delta}\norm{\jap{x-(t+1)v}^{2\delta}\derv{'}{'}{'} g}_{L^\infty_xL^2_v}\right)\\
&\times \left((1+T)^{-|\beta|}\norm{(1+t)^{-\frac{1+\delta}{2}}\jap{x-(t+1)v}^2\jap{v}\der g}_{L^2([0,T];L^2_xL^2_v)}\right)\\
&\lesssim (1+T)^{2|\beta|}\eps^{\frac{3}{4}}\times \left((1+t)^{3-\delta}(1+t)^{-3}\sup_{t\in[0,T]}\sum_{|\omega'''|\leq|\omega'|+3} \norm{\jap{x-(t+1)v}^2\derv{'}{'}{'''} g}_{L^2_xL^2_v}\right)\times \eps^{\frac{3}{4}}\\
&\lesssim (1+T)^{2|\beta|}\eps^{\frac{9}{4}}.
\end{align*}

For $\mathbb{V}$, we use \cref{combining_decay_1}, \eref{boot_assumption_1}, that $\delta<1-s$ and the fact that $s\geq \frac{1}{2}$.
\begin{align*}
\int_0^{T}&(1+t)^{2+\delta}\norm{g}^{1-s}_{L^\infty_xL^1_v}\norm{\part_{v_i}g}^{s}_{L^\infty_xL^1_v}\norm{(1+t)^{-\frac{1}{2}-\frac{\delta}{2}}\jap{v}^{\frac{2s+\gamma}{2}}\der g}^2_{L^2_xL^2_v}\d t\\
&\lesssim\sup_{t\in[0,T]}\left((1+t)^{2+\delta}(1+t)^{-3}\sum_{|\omega'|\leq 3}\norm{\jap{x-(t+1)v}^2Y^{\omega'} g}^{1-s}_{L^2_xL^2_v}\norm{\jap{x-(t+1)v}^2 Y^{\omega'}\part_{v_i}g}^{s}_{L^2_xL^2_v}\right)\\
&\qquad \times\norm{(1+t)^{-\frac{1}{2}-\frac{\delta}{2}}\jap{v}^{\frac{2s+\gamma}{2}}\der g}^2_{L^2([0,T];L^2_xL^2_v)}\\
&\lesssim \color{black}{\sup_{t\in[0,T]}((1+t)^{-1+\delta}(1+t)^s\eps^{\frac{3}{4}})\times (1+T)^{2|\beta|}\eps^{\frac{3}{2}}}\\
&\lesssim (1+T)^{2|\beta|}\eps^{\frac{9}{4}}.
\end{align*}
Similarly, we can get the bound 
\begin{align*}
\int_0^{T}&(1+t)^{2+\delta}\norm{g}^{\frac{3}{2}-s}_{L^\infty_xL^1_v}\norm{\part_{v_i}g}^{s-\frac{1}{2}}_{L^\infty_xL^1_v}\norm{(1+t)^{-\frac{1}{2}-\frac{\delta}{2}}\jap{v}^{\frac{2s+\gamma}{2}}\der g}^2_{L^2_xL^2_v}\d t\\
&\lesssim (1+T)^{2|\beta|}\eps^{\frac{9}{4}}.
\end{align*}
Hence,
$$\mathbb{V}\lesssim (1+T)^{2|\beta|}\eps^2.$$

\textcolor{black}{For $\mathbb{VI}$, we use \cref{combining_decay_1} and \eref{boot_assumption_1} to get,
\begin{align*}
\int_0^{T}&(1+t)^{1+\delta}\norm{\part_{v_i} \derv{'}{'}{'}g}_{L^\infty_xL^1_v}\norm{(1+t)^{-\frac{1+\delta}{2}}\jap{x-(t+1)v}^2\jap{v}\derv{''}{''}{''} g}_{L^2_xL^2_v}\\
&\qquad\times\norm{(1+t)^{-\frac{1+\delta}{2}}\jap{x-(t+1)v}^2\jap{v}\der g}_{L^2_xL^2_v} \d t\\
&\lesssim \sup_{t\in[0,T]}\left((1+t)^{1+\delta}(1+t)^{-3}\sum_{|\omega'''|\leq |\omega|+3}\norm{\jap{x-(t+1)v}^2\part_{v_i} \derv{'}{'}{'''} g}_{L^2_xL^2_v}\right)\\
&\qquad\times \norm{(1+t)^{-\frac{1+\delta}{2}}\jap{x-(t+1)v}^2\jap{v}\der g}_{L^2([0,T];L^2_xL^2_v)}\\
&\qquad\times \norm{(1+t)^{-\frac{1+\delta}{2}}\jap{x-(t+1)v}^2\jap{v}\derv{''}{''}{''} g}_{L^2([0,T];L^2_xL^2_v)}\\
&\lesssim \sup_{t\in[0,T]}\left((1+t)^{1+\delta}(1+t)^{-3}(1+t)^{1+|\beta'|}\eps^{\frac{3}{4}}\right)\times (1+T)^{|\beta|+|\beta''|}\eps^{\frac{3}{2}}\\
&\lesssim (1+T)^{2|\beta|}\eps^{\frac{9}{4}}.
\end{align*}}
Thus,
$$\mathbb{VI}\lesssim (1+T)^{2|\beta|}\eps^2.$$

For $\mathbb{VII}$, recall that $|\alpha'|+|\beta'|+|\omega'|=1$.  First note that by  \lref{partial_derv_interpol} followed by H{\"o}lder's and \lref{L2_time_in_norm} we have,
\begin{align*}
&\norm{(1+t)^{-\frac{1+2\delta}{2}}(1+t)^{-|\beta''|-(2s-1)+}\jap{v}\jap{x-(t+1)v}^2\derv{''}{''}{''} g}_{L^2([0,T];L^2_xH^{(2s-1)^+}_v)}\\
&\lesssim [\norm{(1+t)^{-\frac{1+2\delta}{2}}(1+t)^{-|\beta''|-(2s-1)+}\jap{v}\jap{x-(t+1)v}^2\derv{''}{''}{''} g}_{L^2([0,T];L^2_xL^2_v)}\\
&\quad+\norm{(1+t)^{-\frac{1+2\delta}{2}}(1+t)^{-|\beta''|-1}\jap{v}\jap{x-(t+1)v}^2\part_{v_i}\derv{''}{''}{''} g}^{2s-1+}_{L^2([0,T];L^2_xL^2_v)}\\
&\qquad\times\norm{(1+t)^{-\frac{1+2\delta}{2}}(1+t)^{-|\beta''|}\jap{v}\jap{x-(t+1)v}^2\derv{''}{''}{''} g}^{2-2s-}_{L^2([0,T];L^2_xL^2_v)}]\\
&\lesssim \eps^{\frac{3}{4}}.
\end{align*}
Now using \cref{combining_decay_1}, \eref{boot_assumption_1}, $\delta<1-s$ and that $s\geq \frac{1}{2}$ we get
\begin{align*}
&\int_0^{T}(1+t)^{1+2\delta}\norm{(1+t)^{-\frac{1+2\delta}{2}}\jap{v}\jap{x-(t+1)v}^2\derv{''}{''}{''} g}_{L^2_xH^{(2s-1)^+}_v}\\
&\hspace{5em}\times\norm{\part_{v_i} \derv{'}{'}{'} g}_{L^\infty_xL^1_v}\norm{(1+t)^{-\frac{1+2\delta}{2}}\jap{v}\jap{x-(t+1)v}^2\der g}_{L^2_xL^2_v}\d t\\
&\lesssim (1+T)^{2|\beta|}\left(\norm{(1+t)^{-|\beta''|-(2s-1)+}(1+t)^{-\frac{1+\delta}{2}}\jap{v}\jap{x-(t+1)v}^2\derv{''}{''}{''} g}_{L^2([0,T];L^2_xH^{(2s-1)^+}_v)}\right)\\
&\qquad\times \sup_{t\in [0,T]}\left((1+t)^{-|\beta'|+(2s-1)+}(1+t)^{1+2\delta}(1+t)^{-3}\sum_{|\omega'''|\leq|\omega|+3}\norm{\jap{x-(t+1)v}^2\part_{v_i} \derv{'}{'}{'''} g}_{L^2_xL^2_v}\right)\\
&\qquad\times \left((1+T)^{-|\beta|}\norm{(1+t)^{-\frac{1+2\delta}{2}}\jap{v}\jap{x-(t+1)v}^2\der  g}_{L^2([0,T];L^2_xL^2_v)}\right)\\
&\lesssim (1+T)^{2|\beta|}\eps^{\frac{3}{4}}\times \sup_{t\in [0,T]}\left((1+t)^{-|\beta'|+(2s-1)+}(1+t)^{-3}(1+t)^{|\beta'|+1}(1+t)^{1+2\delta}\eps^{\frac{3}{4}}\right)\times \eps^{\frac{3}{4}}\\
&\lesssim (1+T)^{2|\beta|}\eps^{\frac{3}{4}}\times \sup_{t\in [0,T]}\left((1+t)^{[(2s-1)+]+2\delta-1}\eps^{\frac{3}{4}}\right)\times \eps^{\frac{3}{4}}\\
&\lesssim (1+T)^{2|\beta|}\eps^{\frac{9}{4}}.
\end{align*}
The second half can be estimated in the same manner.
As a result,
$$\mathbb{VII}\lesssim (1+T)^{2|\beta|}\eps^2.$$

For $\mathbb{VIII}$, we have that $|\alpha'|+|\beta'|+|\omega'|=1$. By \lref{partial_derv_interpol} followed by H{\"o}lder's and \lref{L2_time_in_norm} we have,
\begin{align*}
&\norm{(1+t)^{-\frac{1+2\delta}{2}}(1+t)^{-|\beta''|-s}\jap{v}\jap{x-(t+1)v}^2\derv{''}{''}{''} g}_{L^2([0,T];L^2_xH^s_v)}\\
&\lesssim [\norm{(1+t)^{-\frac{1+2\delta}{2}}(1+t)^{-|\beta''|-s}\jap{v}\jap{x-(t+1)v}^2\derv{''}{''}{''} g}_{L^2([0,T];L^2_xL^2_v)}\\
&\quad+\norm{(1+t)^{-\frac{1+2\delta}{2}}(1+t)^{-|\beta''|-1}\jap{v}\jap{x-(t+1)v}^2\part_{v_i}\derv{''}{''}{''} g}^{s}_{L^2([0,T];L^2_xL^2_v)}\\
&\qquad\times\norm{(1+t)^{-\frac{1+2\delta}{2}}(1+t)^{-|\beta''|}\jap{v}\jap{x-(t+1)v}^2\derv{''}{''}{''} g}^{1-s}_{L^2([0,T];L^2_xL^2_v)}]\\
&\lesssim \eps^{\frac{3}{4}}.
\end{align*}
Next, using \cref{combining_decay_1}, \eref{boot_assumption_1} and that $\delta<1-s$, we have
\begin{align*}
&\int_0^{T}(1+t)^{|\beta'|+1+2\delta}\norm{(1+t)^{-\frac{1+2\delta}{2}}\jap{v}\jap{x-(t+1)v}^2\derv{''}{''}{''} g}^2_{L^2_xH^s_v}\\
&\qquad\times\norm{\derv{'}{'}{'} g}_{L^\infty_xL^1_v}\d t\\
&\lesssim (1+T)^{2|\beta|}\left(\norm{(1+t)^{-|\beta''|-s}(1+t)^{-\frac{1+2\delta}{2}}\jap{v}\jap{x-(t+1)v}^2\derv{''}{''}{''} g}^2_{L^2([0,T];L^2_xH^{s}_v)}\right)\\
&\qquad\times \sup_{t\in [0,T]}\left((1+t)^{-2|\beta'|}(1+t)^{2s}(1+t)^{|\beta'|+1+2\delta}(1+t)^{-3}\right.\\
&\hspace{8em}\times\left.\sum_{|\omega'''|\leq|\omega|+3}\norm{\jap{x-(t+1)v}^2 \derv{'}{'}{'''} g}_{L^2_xL^2_v}\right)\\
&\lesssim (1+T)^{2|\beta|}\eps^{\frac{6}{4}}\times \sup_{t\in [0,T]}\left((1+t)^{-|\beta'|+2s}(1+t)^{-3}(1+t)^{|\beta'|}(1+t)^{1+2\delta}\eps^{\frac{3}{4}}\right) \\
&\lesssim (1+T)^{2|\beta|}\eps^{\frac{3}{4}}\times \sup_{t\in [0,T]}\left((1+t)^{2s-2+2\delta}\eps^{\frac{3}{4}}\right)\\
&\lesssim (1+T)^{2|\beta|}\eps^{\frac{9}{4}}.
\end{align*}
Thus we get that,
$$\mathbb{VIII}\lesssim (1+T)^{2|\beta|}\eps^2.$$

For $\mathbb{IX}$, we again have that $|\alpha'|+|\beta'|+|\omega'|=1$. As in $\mathbb{II}$, we have the bound, 
\begin{align*}
 \norm{\jap{v}^{\gamma+2s}&\derv{'}{'}{'} f}_{L^\infty_xW^{2,1}_v}^s\norm{\jap{v}^{\gamma+2s}\derv{'}{'}{'} f}_{L^\infty_xL^1_v}^{1-s}\lesssim \eps^{\frac{3}{4}}(1+t)^{|\beta'|}(1+t)^{-3+2s}.
\end{align*}
Now using \eref{boot_assumption_1} we get,
\begin{align*}
 \int_0^{T}&(1+t)^{|\beta'|+1+\delta} \norm{\jap{v}^{\gamma+2s}\derv{'}{'}{'} f}_{L^\infty_xW^{2,1}_v}^s\norm{\jap{v}^{\gamma+2s}\derv{'}{'}{'} f}_{L^\infty_xL^1_v}^{1-s}\\
&\quad\times\norm{(1+t)^{-\frac{1+\delta}{2}}\jap{x-(t+1)v}^2\jap{v}\derv{''}{''}{''} g}^2_{L^2_xL^2_v}\\
&\lesssim (1+T)^{2|\beta|}\left((1+T)^{-2|\beta''|}\norm{(1+t)^{-\frac{1+\delta}{2}}\jap{x-(t+1)v}^2\jap{v}\derv{''}{''}{''} g}^2_{L^2([0,T];L^2_xL^2_v)}\right)\\
&\quad\times (1+T)^{-2|\beta'|}\left(\sup_{t\in[0,T]} (1+t)^{|\beta'|+1+\delta}\norm{\jap{v}^{\gamma+2s}\derv{'}{'}{'} f}_{L^\infty_xW^{2,1}_v}^s\norm{\jap{v}^{\gamma+2s}\derv{'}{'}{'} f}_{L^\infty_xL^1_v}^{1-s}\right)\\
&\lesssim (1+T)^{2|\beta|}\eps^{\frac{3}{4}}\times (1+T)^{-2|\beta'|}\left(\sup_{t\in[0,T]} (1+t)^{|\beta'|+1+\delta}(1+t)^{-3+2s}(1+t)^{|\beta'|}\eps^{\frac{3}{4}}\right)\times \eps^{\frac{3}{4}}\\
&\lesssim (1+T)^{2|\beta|}\eps^{\frac{9}{4}},
\end{align*}
where we used the fact that $\delta<2-2s.$

Putting this together we get,
$$\mathbb{IX}\lesssim (1+T)^{2|\beta|}\eps^2.$$

For $\mathbb{X}$, we first use the bound established for $\mathbb{VII}$, i.e.
\begin{align*}
\norm{(1+t)^{-\frac{1+2\delta}{2}}(1+t)^{-|\beta''|-(2s-1)+}\jap{v}\jap{x-(t+1)v}^2\derv{''}{''}{''} g}_{L^2([0,T];L^2_xH^{(2s-1)^+}_v)}\lesssim \eps^{\frac{3}{4}}.
\end{align*}
Using this and \eref{boot_assumption_1}, we get,
\begin{align*}
&\int_0^{T}(1+t)^{2+2\delta}\norm{(1+t)^{-\frac{1+2\delta}{2}}\jap{v}\jap{x-(t+1)v}^2\derv{''}{''}{''} g}_{L^2_xH^{(2s-1)^+}_v}\\
&\qquad\times\norm{\derv{'}{'}{'} g}_{L^\infty_xL^1_v}\norm{(1+t)^{-\frac{1+2\delta}{2}}\jap{v}\jap{x-(t+1)v}^2\der g}_{L^2_xL^2_v}\d t\\
&\lesssim (1+T)^{2|\beta|}\left(\norm{(1+t)^{-|\beta''|-(2s-1)+}(1+t)^{-\frac{1+2\delta}{2}}\jap{v}\jap{x-(t+1)v}^2\derv{''}{''}{''} g}_{L^2([0,T];L^2_xH^{(2s-1)^+}_v)}\right)\\
&\qquad\times \sup_{t\in [0,T]}\left((1+t)^{-|\beta'|+(2s-1)+}(1+t)^{2+2\delta}(1+t)^{-3}\sum_{|\omega'''|\leq|\omega|+3}\norm{\jap{x-(t+1)v}^2 \derv{'}{'}{'''} g}_{L^2_xL^2_v}\right)\\
&\qquad\times \left((1+T)^{-|\beta|}\norm{(1+t)^{-\frac{1+2\delta}{2}}\jap{v}\jap{x-(t+1)v}^2\der  g}_{L^2([0,T];L^2_xL^2_v)}\right)\\
&\lesssim (1+T)^{2|\beta|}\eps^{\frac{3}{4}}\times \sup_{t\in [0,T]}\left((1+t)^{-|\beta'|+(2s-1)+}(1+t)^{-3}(1+t)^{|\beta'|}(1+t)^{2+2\delta}\eps^{\frac{3}{4}}\right)\times \eps^{\frac{3}{4}}\\
&\lesssim (1+T)^{2|\beta|}\eps^{\frac{3}{4}}\times \sup_{t\in [0,T]}\left((1+t)^{[(2s-1)+]+2\delta-1}\eps^{\frac{3}{4}}\right)\times \eps^{\frac{3}{4}}\\
&\lesssim (1+T)^{2|\beta|}\eps^{\frac{9}{4}}.
\end{align*}
The second half involving $H^{2s-1}_v$ norm is estimated similarly. Hence,
$$\mathbb{X}\lesssim (1+T)^{2|\beta|}\eps^2.$$

For $\mathbb{XI}$, we use \cref{combining_decay_1} and \eref{boot_assumption_1} to get,
\begin{align*}
\int_0^{T}&(1+t)^{1+2\delta}\norm{(1+t)^{-\frac{1+2\delta}{2}}\jap{x-(t+1)v}^2\jap{v}\part_{v_i} \derv{''}{''}{''} g}_{L^2_xL^2_v}\\
&\qquad \times\norm{\derv{'}{'}{'} g}_{L^\infty_xL^1_v}\norm{(1+t)^{-\frac{1+2\delta}{2}}\jap{x-(t+1)v}^2\jap{v}\der g}_{L^2_xL^2_v}\d t\\
&\lesssim (1+T)^{2|\beta|}\left(\norm{(1+t)^{-\frac{1+2\delta}{2}}(1+t)^{-|\beta''|-1}\jap{x-(t+1)v}^2\jap{v}\part_{v_i}\derv{''}{''}{''} g}_{L^2([0,T];L^2_xL^2_v)}\right)\\
&\qquad\times\sup_{t\in[0,T]}  \left((1+t)^{-|\beta'|+1}(1+t)^{1+2\delta}\norm{\derv{'}{'}{'} g}_{L^\infty_xL^1_v}\right)\\
&\times \norm{(1+t)^{-\frac{1+2\delta}{2}}(1+t)^{-|\beta|}\jap{x-(t+1)v}^2\jap{v}\der g}_{L^2([0,T];L^2_xL^2_v)}\\
&\lesssim (1+T)^{2|\beta|}\eps^{\frac{3}{4}}\times \sup_{t\in[0,T]} \left((1+t)^{1+2\delta}(1+t)^{-3}(1+t)^{-|\beta'|+1}\right.\\
&\hspace{10em}\times\left.\sum_{|\omega'''|\leq|\omega'|+3} \norm{\jap{x-(t+1)v}\derv{'}{'}{'''} g}_{L^2_xL^2_v}\right)\times \eps^{\frac{3}{4}}\\
&\lesssim (1+T)^{2|\beta|}\eps^{\frac{3}{4}}\times \sup_{t\in[0,T]} \left((1+t)^{2\delta-1}\eps^{\frac{3}{4}}\right)\times \eps^{\frac{3}{4}}\\
&\lesssim (1+T)^{2|\beta|}\eps^{\frac{9}{4}}.
\end{align*}
These bounds together imply,
$$\mathbb{XI}\lesssim (1+T)^{2|\beta|}\eps^2.$$

For $\mathbb{XII}$, we first use \lref{partial_derv_interpol} followed by H{\"o}lder's and \lref{L2_time_in_norm} to get,
\begin{align*}
&\norm{(1+t)^{-\frac{1+2\delta}{2}}(1+t)^{-|\beta''|-2s}\jap{v}\jap{x-(t+1)v}^2\derv{''}{''}{''} g}_{L^2([0,T];L^2_xH^{2s}_v)}\\
&\lesssim [\norm{(1+t)^{-\frac{1+2\delta}{2}}(1+t)^{-|\beta''|-2s}\jap{v}\jap{x-(t+1)v}^2\derv{''}{''}{''} g}_{L^2([0,T];L^2_xH^1_v)}\\
&\quad+\norm{(1+t)^{-\frac{1+2\delta}{2}}(1+t)^{-|\beta''|-2}\jap{v}\jap{x-(t+1)v}^2\part^2_{v_iv_j}\derv{''}{''}{''} g}^{2s-1}_{L^2([0,T];L^2_xL^2_v)}\\
&\qquad\times\norm{(1+t)^{-\frac{1+2\delta}{2}}(1+t)^{-|\beta''|-1}\jap{v}\jap{x-(t+1)v}^2\part_{v_i}\derv{''}{''}{''} g}^{2-2s}_{L^2([0,T];L^2_xL^2_v)}]\\
&\lesssim \eps^{\frac{3}{4}}.
\end{align*}

Now using \cref{combining_decay_1}, \eref{boot_assumption_1}, $\delta<1-s$ and that $s\geq \frac{1}{2}$ we get
\begin{align*}
\int_0^{T}&(1+t)^{1+2\delta}\norm{(1+t)^{-\frac{1+2\delta}{2}}\jap{v}\jap{x-(t+1)v}^2\derv{''}{''}{''} g}_{L^2_xH^{2s}_v}\\
&\qquad\times\norm{ \derv{'}{'}{'} g}_{L^\infty_xL^1_v}\norm{(1+t)^{-\frac{1+2\delta}{2}}\jap{v}\jap{x-(t+1)v}^2\der g}_{L^2_xL^2_v}\d t\\
&\lesssim (1+T)^{2|\beta|}\left(\norm{(1+t)^{-|\beta''|-2s}(1+t)^{-\frac{1+2\delta}{2}}\jap{v}\jap{x-(t+1)v}^2\derv{''}{''}{''} g}_{L^2([0,T];L^2_xH^{2s}_v)}\right)\\
&\qquad\times \sup_{t\in [0,T]}\left((1+t)^{-|\beta'|+2s}(1+t)^{1+2\delta}(1+t)^{-3}\sum_{|\omega'''|\leq|\omega|+3}\norm{\jap{x-(t+1)v}^2\derv{'}{'}{'''} g}_{L^2_xL^2_v}\right)\\
&\qquad\times \left((1+T)^{-|\beta|}\norm{(1+t)^{-\frac{1+2\delta}{2}}\jap{v}\jap{x-(t+1)v}^2\der  g}_{L^2([0,T];L^2_xL^2_v)}\right)\\
&\lesssim (1+T)^{2|\beta|}\eps^{\frac{3}{4}}\times \sup_{t\in [0,T]}\left((1+t)^{-|\beta'|+2s}(1+t)^{-3}(1+t)^{|\beta'|}(1+t)^{1+2\delta}\eps^{\frac{3}{4}}\right)\times \eps^{\frac{3}{4}}\\
&\lesssim (1+T)^{2|\beta|}\eps^{\frac{3}{4}}\times \sup_{t\in [0,T]}\left((1+t)^{2s-2+2\delta}\eps^{\frac{3}{4}}\right)\times \eps^{\frac{3}{4}}\\
&\lesssim (1+T)^{2|\beta|}\eps^{\frac{9}{4}}.
\end{align*}
Thus we get that,
$$\mathbb{XII}\lesssim (1+T)^{2|\beta|}\eps^2.$$

Putting all these bound together, we get the desired lemma.
\end{proof}
\begin{lemma}\label{l.comm_1}
For $|\alpha|+|\beta|+|\omega|\leq 10$. Then for every $\eta>0$, there exists a constant $C_\eta>0$ (depending only on $d_0$, $\gamma$, $s$ and $\eta$) such that the term $\text{Comm}_1$ as in \eref{comm_1} if bounded as follows for all $T\in[0,T_{\boot})$,
\begin{align*}
\text{Comm}_1&\lesssim \eta\norm{\jap{x-(t+1)v}^2\der g}^2_{L^\infty([0,T];L^2_xL^2_v)}\\
&\quad+C_\eta T^2\sum_{\substack{|\alpha'|\leq |\alpha|+1\\ |\beta'|\leq |\beta|-1}}\norm{\jap{x-(t+1)v}^2\derv{'}{'}{'} g}^2_{L^\infty([0,T];L^2_xL^2_v)}.
\end{align*}
\end{lemma}
\begin{proof}
By H\"{o}lder's inequality, we have that,
\begin{align*}
&\sum_{\substack{|\alpha'|\leq |\alpha|+1\\ |\beta'|\leq |\beta|-1}}\norm{\jap{x-(t+1)v}^4|\der g|\cdot |\derv{'}{'}{'} g|}_{L^1([0,T];L^1_xL^1_v)}\\
&\lesssim \sum_{\substack{|\alpha'|\leq |\alpha|+1\\ |\beta'|\leq |\beta|-1}}T\norm{\jap{x-(t+1)v}^2\der g}_{L^\infty([0,T];L^2_xL^2_v)}\norm{\jap{x-(t+1)v}^2\derv{'}{'}{'} g}_{L^\infty([0,T];L^2_xL^2_v)}.
\end{align*}
Now the required result follows by Young's inequality.
\end{proof}
\begin{lemma}\label{l.comm_2}
For $|\alpha|+|\beta|+|\omega|\leq 10$. Then for every $\eta>0$, there exists a constant $C_\eta>0$ (depending only on $d_0$, $\gamma$, $s$ and $\eta$) such that the term $\text{Comm}_2$ as in \eref{comm_2} if bounded as follows for all $T\in[0,T_{\boot})$,
\begin{align*}
\text{Comm}_2&\lesssim \eta\norm{(1+t)^{-\frac{1+\delta}{2}}\jap{v}\jap{x-(t+1)v}^2\der g}^2_{L^2([0,T];L^2_xL^2_v)}\\
&\quad+C_\eta\sum_{\substack{|\beta'|\leq |\beta|, |\omega'|\leq |\omega|\\|\beta'|+|\omega'|\leq |\beta|+|\omega|-1}}\norm{(1+t)^{-\frac{1+\delta}{2}}\jap{v}\jap{x-(t+1)v}^2\derv{'}{'}{'} g}^2_{L^2([0,T];L^2_xL^2_v)}.
\end{align*}
\end{lemma}
\begin{proof}
The result follows by Cauchy-Schwarz and Young's inequality.
\end{proof}
\section{Putting everything together}\label{s.put}
Combining \lref{eng_set_up} with \lref{kernel_boot_est}, \lref{comm_1} and \lref{comm_2}, we get
\begin{proposition}\label{p.int_est}
Let $|\alpha|+|\beta|+|\omega|\leq 10$. Then for every $\eta>0$, there is a constant $C_\eta>0$(depending on $\eta$, $d_0$, $\gamma$ and $s$) such that the following estimate holds for all $T\in [0,T_{\boot})$:
\begin{align*}
&\norm{\jap{x-(t+1)v}^{2}\der g}_{L^\infty([0,T];L^2_xL^2_v)}^2+\norm{(1+t)^{-\frac{1}{2}-\frac{\delta}{2}}\jap{v}\jap{x-(t+1)v}^{2}\der g}_{L^2([0,T];L^2_xL^2_v)}^2\\
&\leq C_\eta(\eps^2(1+T)^{2|\beta|}+\sum_{\substack{|\beta'|\leq|\beta|, |\omega'|\leq |\omega|\\ |\beta'|+|\omega'|\leq |\beta|+|\omega|-1}} \norm{(1+t)^{-\frac{1+\delta}{2}}\jap{v}\jap{x-(t+1)v}^{2}\derv{}{'}{'} g}^2_{L^2([0,T];L^2_xL^2_v)}\\
&\quad+(1+T)^{2}\sum_{\substack{|\alpha'|\leq|\alpha|+1\\ |\beta'|\leq |\beta|-1}} \norm{\jap{x-(t+1)v}^{2}\derv{'}{'}{} g}_{L^\infty([0,T];L^2_xL^2_v)}^2)\\
&\quad+\eta \norm{(1+t)^{-\frac{1+\delta}{2}}\jap{v}\jap{x-(t+1)v}^{2}\der g}^2_{L^2([0,T];L^2_xL^2_v)}\\
&\quad+\eta\norm{\jap{x-(t+1)v}^2\der g}_{L^\infty([0,T];L^2_xL^2_v)}^2.
\end{align*}
\end{proposition}
\begin{proposition}\label{p.ene_est_1}
Let $|\alpha|+|\beta|+|\omega|\leq 10$. Then the following estimate holds for all $T\in [0,T_{\boot})$:
\begin{align*}
&\norm{\jap{x-(t+1)v}^{2}\der g}_{L^\infty([0,T];L^2_xL^2_v)}^2\\
&\quad+\norm{(1+t)^{-\frac{1}{2}-\frac{\delta}{2}}\jap{v}\jap{x-(t+1)v}^{2}\der g}_{L^2([0,T];L^2_xL^2_v)}^2\\
&\lesssim \eps^2(1+T)^{2|\beta|}+\sum_{\substack{|\beta'|\leq|\beta|, |\omega'|\leq |\omega|\\ |\beta'|+|\omega'|\leq |\beta|+|\omega|-1}} \norm{(1+t)^{-\frac{1+\delta}{2}}\jap{v}\jap{x-(t+1)v}^{2}\derv{}{'}{'} g}^2_{L^2([0,T];L^2_xL^2_v)}\\
&\quad+(1+T)^{2}\sum_{\substack{|\alpha'|\leq|\alpha|+1\\ |\beta'|\leq |\beta|-1}} \norm{\jap{x-(t+1)v}^2\derv{'}{'}{} g}_{L^\infty([0,T];L^2_xL^2_v)}^2.
\end{align*}
Here, by our convention, if $|\beta|+|\omega|=0$, then the last two terms on the RHS are not present.
\end{proposition}
\begin{proof}
We just apply \pref{int_est} with $\eta=\frac{1}{2}$. Then the following term is absorbed on the LHS
\begin{align*}
\frac{1}{2}\norm{(1+t)^{-\frac{1+\delta}{2}}\jap{v}^{\frac{1}{2}}\jap{v}^{\nu_{\alpha,\beta,\omega}}\jap{x-(t+1)v}^{\omega_{\alpha,\beta,\omega}}\der g}_{L^2([0,T];L^2_xL^2_v)}^2.
\end{align*}
Since $\eta$ is fixed, $C_\eta$ is just a constant depending on $d_0$ and $\gamma$. We thus get the desired inequality.
\end{proof}
\begin{proposition}\label{p.ene_est}
Let $|\alpha|+|\beta|+|\omega|\leq 10$. Then the following estimate holds for all $T\in [0,T_{\boot})$:
\begin{align*}
&\norm{\jap{x-(t+1)v}^{2}\der g}_{L^\infty([0,T];L^2_xL^2_v)}^2\\
&+\norm{(1+t)^{-\frac{1}{2}-\frac{\delta}{2}}\jap{v}\jap{x-(t+1)v}^{2}\der g}_{L^2([0,T];L^2_xL^2_v)}^2\lesssim \eps^2(1+T)^{2|\beta|}.
\end{align*}
\end{proposition}
\begin{proof}
The proof proceeds by induction on $|\beta|+|\omega|$.\\
\emph{Step 1: Base Case: $|\beta|+|\omega|=0$.} Applying \pref{ene_est_1} when $|\beta|+|\omega|=0$, the last two terms on the RHS are not present. Hence we immediately have
\begin{align*}
\norm{\jap{x-(t+1)v}^{2}\part^{\alpha}_x g}_{L^\infty([0,T];L^2_xL^2_v)}^2+\norm{(1+t)^{-\frac{1}{2}-\frac{\delta}{2}}\jap{v}\jap{x-(t+1)v}^{2}\part^{\alpha}_x g}_{L^2([0,T];L^2_xL^2_v)}^2\lesssim \eps^2.
\end{align*}
\emph{Step 2: Inductive step.} Assume by induction that there exists a $B\in \N$ such that whenever $|\alpha|+|\beta|+|\omega|\leq 10$ and $|\beta|+|\omega|\leq B-1$,
\begin{align*}
&\norm{\jap{x-(t+1)v}^{2}\der g}_{L^\infty([0,T];L^2_xL^2_v)}^2\\
&+\norm{(1+t)^{-\frac{1}{2}-\frac{\delta}{2}}\jap{v}\jap{x-(t+1)v}^{2}\der g}_{L^2([0,T];L^2_xL^2_v)}^2\lesssim \eps^2(1+T)^{2|\beta|}.
\end{align*}
Now take some multi-indices $\alpha,$ $\beta$ and $\omega$ such that $|\alpha|+|\beta|+|\omega|\leq 10$ and $|\beta|+|\omega|=B$. We will show that the estimate as in the statement of the proposition holds for this choice of $(\alpha,\beta,\omega)$.\\
By  \pref{ene_est_1} and the inductive hypothesis,
\begin{align*}
&\norm{\jap{x-(t+1)v}^{2}\der g}_{L^\infty([0,T];L^2_xL^2_v)}^2\\
&\quad+\norm{(1+t)^{-\frac{1}{2}-\frac{\delta}{2}}\jap{v}\jap{x-(t+1)v}^{2}\der g}_{L^2([0,T];L^2_xL^2_v)}^2\\
&\lesssim \eps^2(1+T)^{2|\beta|}+\sum_{\substack{|\beta'|\leq|\beta|, |\omega'|\leq |\omega|\\ |\beta'|+|\omega'|\leq |\beta|+|\omega|-1}} \norm{(1+t)^{-\frac{1+\delta}{2}}\jap{v}\jap{x-(t+1)v}^{2}\derv{}{'}{'} g}^2_{L^2([0,T];L^2_xL^2_v)}\\
&\quad+(1+T)^{2}\sum_{\substack{|\alpha'|\leq|\alpha|+1\\ |\beta'|\leq |\beta|-1}} \norm{\jap{x-(t+1)v}^{2}\derv{'}{'}{} g}_{L^\infty([0,T];L^2_xL^2_v)}^2\\
&\lesssim \eps^2(1+T)^{2|\beta|}+\eps^2(\sum_{|\beta'|\leq |\beta|}(1+T)^{2|\beta'|})+\eps^2(\sum_{|\beta'|\leq |\beta|-1}(1+T)^{2}(1+T)^{2|\beta'|})\\
&\lesssim \eps^2(1+T)^{2|\beta|}.
\end{align*}
We thus get the desired result by induction.
\end{proof}
\pref{ene_est} also completes the proof of \tref{boot}.
\section{Proof of \tref{global}}
\begin{proof}[Proof of \tref{global}]
Let 
\begin{align*}
T_{\max}:=\sup\{T\in[0,\infty):&\text{ there exists a unique solution } f:[0,T]\times \R^3\times \R^3 \text{ to } \eref{boltz} \text{ with}\\
& f\geq 0,\: f|_{t=0}=f_{\ini} \text{ and satisfying }\eref{local} \\
& \text{ such that the bootstrap assumption } \eref{boot_assumption_1} \text{ holds}\}.
\end{align*}
Note that by \lref{local}, $T_{\max}>0$.

We will prove that $T_{\max}=\infty$. Assume for the sake of contradiction that $T_{\max}<\infty$.\\
From the definition of $T_{\max}$, we have that the assumptions of \tref{boot} hold for $T_{\boot}=T_{\max}$.\\
Therefore, by \tref{boot} (with $|\omega|=0$) we get 
\begin{equation}\label{e.uni}
\sum \limits_{|\alpha|+|\beta|\leq 10}\norm{\jap{x-(t+1)v}^{2}\part^{\alpha}_x\part^\beta_v(e^{d(t)\jap{v}^2}f)}_{L^\infty([0,T_{\max});L^2_xL^2_v)}\lesssim \eps.
\end{equation}

Take an increasing sequence $\{t_n\}_{n=1}^\infty\subset [0,T_{\max})$ such that  $t_n \to T_{\max}$. By the uniform bound \eref{uni} and the local existence result \lref{local}, there exists $T_{\sm}\in (0,1]$ such that the unique solution exists on $[0,t_n+T_{\sm}]\times \R^3\times \R^3$. In particular, taking $n$ sufficiently large, we have constructed a solution beyond the time $T_{\max}$, up to, $T_{\max}+\frac{1}{2}T_{\sm}$. The solution moveover satisfies \eref{local}.

We next prove that the estimate \eref{boot_assumption_1} holds slightly beyond $T_{\max}$. To that end we employ \tref{boot},
\begin{equation}\label{e.boot_cont}
\text{the estimate \eref{boot_assumption_1} holds in }[0,T_{\max}) \text{ with } \eps^{\frac{3}{4}} \text{ replaced by }C_{d_0,\gamma,s}\eps.
\end{equation}

By the local existence result in \lref{local}, for $|\alpha|+|\beta|\leq 10$, $$\jap{x-(t+1)v}^{2}\part^\alpha_x\part^\beta_v g(t,x,v)\in C^0([0,T_{\max}+\frac{1}{2}T_{\sm}];L^2_xL^2_v).$$
Since $Y=(t+1)\part_x+\part_v$, for all $|\alpha|+|\beta|+|\omega|\leq 10$, we also have  $$\jap{x-(t+1)v}^{2}\der g(t,x,v)\in C^0([0,T_{\max}+\frac{1}{2}T_{\sm}];L^2_xL^2_v).$$

Using \eref{uni}, after choosing $ \eps_0$ smaller (so that $\eps$ is sufficiently small) if necessary, there exists $T_{\text{ext}}\in (T_{\max},T_{\max}+\frac{1}{2}T_{\sm}]$ such that \eref{boot_assumption_1} hold upto $T_{\text{ext}}$.\\
This is a contradiction to the definition of $T_{\max}$. Thus we deduce that $T_{\max}=+\infty$.
\end{proof}

\end{document}